\documentclass[letterpaper,english, 11pt]{amsart}

\usepackage{adjustbox}
\usepackage[latin9]{inputenc}
\usepackage{hyperref}
\usepackage{lscape}
\usepackage{slashbox}
\usepackage{mathrsfs}
\usepackage[title]{appendix}
\usepackage{enumerate}
\usepackage{enumitem}
\usepackage{xcolor}
\usepackage{bm}
\usepackage[sort=def]{glossaries}
\usepackage{tikz-cd}

\usepackage{tikz}
\usepackage{url}
\usepackage{tabularx}

\usetikzlibrary{arrows}\usetikzlibrary{decorations,decorations.text,decorations.markings,decorations.shapes}
\tikzset{ closed/.style = {decoration = {markings, mark = at position 0.5 with { \node[transform shape, xscale = .8, yscale=.4] {/}; } }, postaction = {decorate} },
open/.style = {decoration = {markings, mark = at position .5 with { \node[transform shape, scale =1.2] {$\circ$}; } }, postaction = {decorate} }
}

\makeatletter

\numberwithin{equation}{subsection}

 \theoremstyle{plain}
\newtheorem{thm}[equation]{Theorem}
 \theoremstyle{plain}

\theoremstyle{plain}
\newtheorem{conj''}[equation]{Conjecture}
\theoremstyle{plain}

\theoremstyle{plain}

\theoremstyle{plain}

\theoremstyle{plain}

\theoremstyle{plain}
  \newtheorem{prop}[equation]{Proposition}
\theoremstyle{plain}
 \newtheorem{lemma}[equation]{Lemma}
\theoremstyle{plain}

\theoremstyle{plain}
\newtheorem{cor}[equation]{Corollary}
\theoremstyle{plain}

\theoremstyle{definition}
  \newtheorem{defn}[equation]{Definition}
\theoremstyle{definition}
  
 \theoremstyle{definition}

\theoremstyle{remark}
\newtheorem{rmk}[equation]{Remark}
 \theoremstyle{plain}
\newtheorem{conv}[equation]{Convention}
\usepackage{amsmath}
\usepackage{amssymb}
\usepackage{amscd}
\usepackage{amsthm}
\usepackage{array,multirow}
\usepackage[arrow,matrix,tips,all,cmtip,poly,color]{xy}

\usepackage{stmaryrd}
\usepackage{url}
\usepackage{tikz-cd}
\usepackage{color}
\usepackage{caption}
\usepackage{float}

\pdfpageheight\paperheight
\pdfpagewidth\paperwidth
\newcommand{\Z}{\mathbb{Z}}

\newcommand{\Q}{\mathbb{Q}}

\newcommand{\Qp}{\mathbb{Q}_p}

\newcommand{\C}{\mathbb{C}}
\newcommand{\bG}{\mathbb{G}}

\newcommand{\F}{\mathbb{F}}

\newcommand{\N}{\mathbb{N}}

\newcommand{\fgl}{\mathfrak{gl}}

\newcommand{\fF}{\mathfrak{f}}

\newcommand{\fM}{\mathfrak{M}}

\newcommand{\fS}{\mathfrak{S}}
\newcommand{\fT}{\mathfrak{t}}

\newcommand{\fm}{\mathfrak{m}}

\newcommand{\fs}{\mathfrak{s}}

\newcommand{\bA}{\mathbb{A}}

\newcommand{\bL}{\mathbb{L}}

\newcommand{\bP}{\mathbb{P}}

\newcommand{\bV}{\mathbb{V}}

\newcommand{\bX}{\mathbb{X}}

\newcommand{\bZ}{\mathbb{Z}}

\newcommand{\cA}{\mathcal{A}}

\newcommand{\cE}{\mathcal{E}}
\newcommand{\cF}{\mathcal{F}}
\newcommand{\cG}{\mathcal{G}}

\newcommand{\cI}{\mathcal{I}}

\newcommand{\cJ}{\mathcal{J}}
\newcommand{\cK}{\mathcal{K}}
\newcommand{\cL}{\mathcal{L}}
\newcommand{\cM}{\mathcal{M}}

\newcommand{\cO}{\mathcal{O}}

\newcommand{\cZ}{\mathcal{Z}}

\newcommand{\eps}{\varepsilon}

\newcommand{\phz}{\varphi}

\newcommand{\Zp}{\mathbb{Z}_p}

\newcommand{\Id}{\mathrm{id}}
\newcommand{\Gal}{\mathrm{Gal}}
\newcommand{\Hom}{\mathrm{Hom}}

\newcommand{\Res}{\mathrm{Res}}

\newcommand{\GL}{\mathrm{GL}}
\newcommand{\B}{\mathrm{B}}

\newcommand{\Spec}{\mathrm{Spec}\ }

\newcommand{\Fp}{\F_p}

\newcommand{\un}[1]{\underline{#1}}
\renewcommand{\bf}[1]{\mathbf{#1}}
\newcommand{\Rep}{\mathrm{Rep}}

\newcommand{\tld}[1]{\widetilde{#1}}

\newcommand{\rhobar}{\overline{\rho}}
\newcommand{\taubar}{\overline{\tau}}

\newcommand{\Spf}{\mathrm{Spf}}

\newcommand{\Adm}{\mathrm{Adm}}
\newcommand{\orient}{\mathrm{or}}

\newcommand{\nv}{\mathrm{nv}}

\newcommand{\rG}{\mathrm{G}}

\newcommand{\defeq}{\stackrel{\textrm{\tiny{def}}}{=}}

\newcommand{\ovl}[1]{\overline{#1}}

\DeclareMathOperator{\Mod}{Mod}

\DeclareMathOperator{\Lie}{Lie}

\DeclareMathOperator{\Ad}{Ad}

\DeclareMathOperator{\Gl}{GL}

\DeclareMathOperator{\Mat}{Mat}

\DeclareMathOperator{\coker}{coker}
\DeclareMathOperator{\Gr}{Gr}

\newcommand{\ra}{\rightarrow}

\newcommand{\iarrow}{\hookrightarrow}
\newcommand{\into}{\hookrightarrow}

\newcommand{\onto}{\twoheadrightarrow}

\newcommand{\tmod}{\textnormal{mod}}
\newcommand{\pr}{\textnormal{pr}}
\newcommand{\ev}{\textnormal{ev}}
\newcommand{\G}{\textnormal{G}}
\newcommand{\curlAdm}{\mathcal{A}}
\newcommand{\gA}{\text{\tt A}}
\newcommand{\gB}{\text{\tt B}}
\newcommand{\gAB}{\text{\tt AB}}
\newcommand{\gO}{\text{\tt O}}

\usepackage{xcolor} \usepackage{float}

\title[Barsotti--Tate local model theory]{Local model theory for non-generic tame potentially Barsotti--Tate deformation rings}

\author{Bao V.~Le Hung}
\address{Department of Mathematics,
Northwestern University,
2033 Sheridan Road,
Evanston, Illinois 60208, USA}
\email{lhvietbao@googlemail.com}

\author{Ariane M\'ezard}
\address{Department of Mathematics and Applications,
ENS PSL,
45 rue d'Ulm,
75005 Paris, France}
\email{ariane.mezard@ens.fr}

\author{Stefano Morra}
\address{Universit\'e Paris 8, Laboratoire d'Analyse, G\'eom\'etrie et Applications,  LAGA, Universit\'e Sorbonne Paris Nord, CNRS, UMR 7539,  F-93430, Villetaneuse, France}
\email{morra@math.univ-paris13.fr}

\begin{document}

\begin{abstract}
We develop a local model theory for moduli stacks of $2$-dimensional non-scalar tame potentially Barsotti--Tate Galois representations of the Galois group of an unramified extension of $\Q_p$.
We derive from this explicit presentations of potentially Barsotti--Tate deformation rings, allowing us to prove structural results about them, and prove various conjectures formulated by Caruso--David--M\'ezard.
\end{abstract}
\maketitle
\tableofcontents

\section{Introduction}
\let\thefootnote\relax\footnotetext{On behalf of all authors, the corresponding author states that there is no conflict of interest and that there is no data associated to this work.}

\subsection{Main results}
Let $p$ be a prime number, $K$ a $p$-adic field. We work with coefficient ring $\cO=W(\F)$ where $\F/\Fp$ is a sufficiently large finite extension.
Let $\rhobar:G_K\longrightarrow \Gl_2(\F)$ be a continuous Galois representation and $\tau=\chi\oplus \chi'$ be an inertial type where $\chi,\chi'$ are \emph{tame} characters of $I_K$. This gives rise to the universal framed deformation ring $R_{\rhobar}^{\eta,\tau}$ classifying lifts of $\rhobar$ which are potentially Barsotti--Tate (i.e.~potentially crystalline with Hodge--Tate weights $0,1$) and of type $\tau$. 
Despite their prominent role in modularity questions via the Taylor--Wiles method (e.g.~\cite{kisin-annals}) in the last decades, their internal structure is still poorly understood. 
The basic reason for this, as suggested by works of Caruso--David--M\'ezard \cite{CDM1},\cite{CDM2},\cite{CDM3}, is that even when $K$ is unramified, $R_{\rhobar}^{\eta,\tau}$ exhibits a wide range of complicated behavior (in particular, it can be highly singular), especially as the inertial weights of $\chi/\chi'$ become more degenerate (that is, when $\tau$ becomes more \emph{non-generic}).

More recently, Caraiani--Emerton--Gee--Savitt \cite{CEGSA} constructed a $p$-adic formal algebraic stack $\cZ^{\tau}$ which interpolates the deformation rings $R^{\eta,\tau}_{\rhobar}$ as $\rhobar$ varies, in the sense that the latter recovers versal rings to finite type points of the former.
The stacks $\cZ^{\tau}$ (and its analogues for other $p$-adic Hodge theory conditions) are expected to be key geometric objects in the categorical $p$-adic Langlands conjectures formulated by Emerton--Gee--Hellmann \cite{EGH}, similar to the role played by various moduli spaces of local systems in the geometric Langlands program.
Thus it is of interest to understand their geometry.

From now on, we assume $K=\Q_{p^f}$ is absolutely unramified and $\tau$ is non-scalar (the scalar case $\tau=\chi\oplus\chi$ being easily handled by Fontaine--Laffaille theory). When $\tau$ is sufficiently generic, the structure of $\cZ^{\tau}$ is well-understood (see e.g. \cite{kisin-annals,MLM}), since it can be modeled using Iwahori level local models of Shimura varieties for $\GL_2$, in particular its singularities are products of the singularity $XY=p$. In this paper, we introduce a method to probe the structure of $\cZ^{\tau}$ which is powerful enough to handle non-generic $\tau$. One concrete consequence of our study is the following general control on singularities:
\begin{thm}[Theorem \ref{thm:rational smoothness}] \label{thm:intro:Gorenstein} Assume either $p\geq 7$ or $K=\mathbb{Q}_5$. Then the normalization of $\cZ^{\tau}$ has rational singularities and is Gorenstein.
\end{thm}
Furthermore, it turns out that $\cZ^{\tau}$ is almost always normal:
\begin{thm}[Theorem \ref{thm:non-normal locus}] \label{thm:intro:normal} Assume either $p\geq 7$ or $K=\mathbb{Q}_5$. Then $\cZ^{\tau}$ is normal, unless after twisting, $\tau=\chi\oplus \chi'$ is the sum of restrictions to $I_K$ of characters of $G_K$, and the inertial weights of $\tau$ belong to $\{0,1\}$. When $\tau$ is of this form, the non-normal locus consists of exactly the $\rhobar$ which are Fontaine--Laffaille with specific (irregular) inertial weights determined $\tau$.
\end{thm} 
\begin{rmk}
\begin{enumerate}
\item Explicitly, for $K=\Qp$ Theorem \ref{thm:intro:normal} shows that $\cZ^{\tau}$ is normal unless, up to twists, $\ovl{\tau}=1\oplus \ovl{\varepsilon}$ (where $\ovl{\tau}$ and $\ovl{\varepsilon}$ denote the mod $p$-reduction of $\tau$ and the cyclotomic character respectively), in which case the non-normal locus consists of $\rhobar$ is a twist of an unramified representation by $\ovl{\varepsilon}$.

\item It is proven in \cite{CEGSC} that the special fiber of $R^{\eta,\tau}_{\rhobar}$ is generically reduced. Together with Theorem \ref{thm:rational smoothness}, this implies that $R^{\eta,\tau}_{\rhobar}$ is Cohen--Macaulay if and only if it is normal, in which case it is furthermore Gorenstein. Thus Theorem \ref{thm:intro:normal} completely classifies when $R^{\eta,\tau}_{\rhobar}$ is Cohen-Macaulay. We invite the reader to compare this to the result of Hu--Pa{\v{s}}k{\=u}nas \cite{PH} about Cohen--Macaulayness of crystabelline deformation rings: whereas \cite{PH} covers the situation where $K=\Qp$, $\tau$ is restricted to (possibly wildly ramified) principal series types (i.e. $\chi,\chi'$ can be extended to characters of $G_{\Qp}$) but allows arbitrary Hodge--Tate weights, our result allows general unramified $K$ but restricts to tame (possibly non-principal series) types $\tau$ and Hodge--Tate weights $0,1$. We also point out that the method of \cite{PH} is unlikely to establish neither the Gorenstein nor the rational singularity property. 

Finally, as is well-known (cf.~\cite{PH}), Cohen--Macaulayness for (normalizations) of deformation rings allows one to upgrade $R[\frac{1}{p}]=T[\frac{1}{p}]$ theorems to integral $R=T$ theorems, hence our results give new instances of such.

\item The fact $\cZ^{\tau}$ has rather mild singularities is expected to be useful for the categorical $p$-adic Langlands program for $\GL_2(K)$, namely it suggests the conjectural functor $\mathfrak{A}$ of \cite{EGH} to have simple effect on a certain generating set of smooth representations of $\GL_2(K)$, thus giving hope that one can construct $\mathfrak{A}$ by ``generators and relations''. We point out that any generating set must necessarily involve representations of $\GL_2(K)$ with non-generic parameters, thus it is essential that we allow arbitrarily non-generic $\tau$ for this purpose.
\end{enumerate}
\end{rmk}

Our method to probe $\cZ^{\tau}$ is to construct group-theoretic local models for it. The main local model theorem has the following form: 
\begin{thm} [Proposition \ref{prop:innomable}, Theorem \ref{thm:main:model}]
\label{thm-main-intro}
Let $\tau=\chi\oplus \chi'$ be a tame inertial type with $\chi\neq \chi'$. There exists a $p$-adic formal scheme $\tld{\cZ}^{\textnormal{mod},\tau}$ such that
\begin{itemize}
\item If either $p\geq 7$ or $K=\Q_5$, $\cZ^{\tau}/p$ is smooth locally isomorphic to $\tld{\cZ}^{\textnormal{mod},\tau}/p$.
\item If either $p>16f+7$ or $p>7$ and $K=\Q_p$, then $\cZ^{\tau}$ is smooth locally isomorphic to $\tld{\cZ}^{\textnormal{mod},\tau}$.
\end{itemize}
\end{thm}
\begin{rmk}(Features of $\tld{\cZ}^{\textnormal{mod},\tau}$)
We defer the somewhat involved definition of $\tld{\cZ}^{\textnormal{mod},\tau}$ to section \ref{subsect:methods} below, and instead note for now that:
\begin{enumerate}
\item Its construction involves the geometry of (mixed characteristic) loop groups.
\item Its geometric structure is independent of $p$, in the sense that it essentially arises as the $p$-adic completion of a natural $\Z$-scheme. In particular, this exhibits a kind of ``independence-of-$p$'' property of tame potentially Barsotti--Tate deformation rings, as suggested in \cite{CDM3} and \cite{CDM4}. 
\item It admits an explicit affine cover where each affine open can be presented as (the $p$-saturation of) explicit equations constructed using the inertial weights of $\tau$ (see Table \ref{Table:Geometric_genes} for a sense of the presentations that show up).  
\end{enumerate}
\end{rmk}
\begin{rmk} (Bounds on $p$)\label{rmk:bounds on p}
Theorem \ref{thm-main-intro} is obtained via deformation theory: we construct $\tld{\cZ}^{\textnormal{mod},\tau}$ which captures the structure of $\cZ^{\tau}$ modulo some power of $p$, and then show this property persists when deforming to mixed characteristics.
The requirement that $p$ needs to be at least some linear bound on $f$ arises for two related reasons:
\begin{itemize}
\item There are non-isomorphic charts of $\cZ^{\tau}$ that are indistinguishable modulo linear powers of $p$. For instance, there are charts equisingular to $XY=p^k$ for any $k\leq f$.
This requires us to start with a model that approximates $\cZ^{\tau}$ modulo at least $p^{O(f)}$ to distinguish these charts.
\item To show the approximation deforms, we need to overcome obstruction groups for certain lifting problems, whose $p$-torsion can have exponents as large as linear in $f$.
\end{itemize}
\end{rmk}

The explicit nature of our models yields, for the first time, an efficient algorithm to compute \emph{any} given potentially Barsotti--Tate deformation ring with tame inertial type.
The basic form of the algorithm is as follows (the details occur in section \ref{subsect:naive model}). To the pair $(\rhobar,\tau)$ we assign
\begin{itemize}
\item an $f$-tuple $\tld{w}(\rhobar,\tau)$ of elements in the extended affine Weyl group $\tld{W}$ of $\GL_2$, which measures the relative position of the $\iota$-inertial weights of $\rhobar$ and $\tau$ for each $\iota:K\into \ovl{\Q}_p$; and
\item an $f$-tuple of ``degeneracy types'' for each $\iota:K\into \ovl{\Q}_p$ which roughly measures how degenerate the $\iota$-th inertial weight of $\chi/\chi'$ is.
\end{itemize}
These datas give rise, for each $\iota$, to a basic ring $R_\iota$, as well as a collection of structure matrices with entries in $R_{\iota}$ recorded in Tables \ref{Table:Traductions}, \ref{Table:matrices}. Both these datas are independent of $f$. Then $R^{\eta,\tau}_{\rhobar}$ is given as a suitable completion of (the $p$-saturation of) the quotient of $\bigotimes_{\iota} R_{\iota}$ by the relations that certain products of the structure matrices are zero. We stress that the equations we impose will generally involve mutual interactions between $R_{\iota}$'s for arbitrary large sets of different $\iota$. For this reason, outside of the generic case, $R^{\eta,\tau}_{\rhobar}$ does not generally admit an obvious tensor product decomposition along embeddings $\iota: K\into \ovl{\Q}_p$.

In the series of work \cite{CDM1,CDM2,CDM3}, Caruso--David--M\'ezard also investigated the problem of algorithmically computing $R^{\eta,\tau}_{\rhobar}$, in the special case where $\tau=\chi\oplus \chi'$ is a principal series type (so $\chi,\chi'$ extend to characters of $G_K$) and $\rhobar$ is irreducible. In a few cases when $f\leq 3$ (\cite[Th\'eor\`eme 4.3.1]{CDM1} and \cite[5.3.3]{CDM2}), they managed to determine $R_{\rhobar}^{\eta,\tau}$ based on the fact that one can guess ``a priori" what it is, cf.~\cite[Remark 3.2.10]{CDM1}. However, their investigations also suggested that the answer becomes intrinsically complicated for large $f$, which made their strategy hopeless in general. We demonstrate the power of our algorithm by confirming various of their conjectural examples of \cite[\S 5.3.2, 5.3.3]{CDM2}, as well as computing all examples for $K=\Qp$, some of which are new as alluded to in \cite[\S 7.5.13]{EGH}. For instance, for $K=\Qp$, $p\geq 7$, $R^{\eta,\tau}_{\rhobar}$ is a power series ring over either $\cO$, $\cO[\![X,Y]\!]/(XY-p)$ or $\cO[\![X,Y]\!]/(XY-p^2)$, except when up to twist, $\ovl{\tau}=\omega_2\oplus \omega_2^p$ and $\rhobar\otimes \ovl{\varepsilon}^{-1}$ is unramified and has scalar semisimplification, or $\ovl{\tau}=1\oplus \omega_1$ and $\rhobar\otimes \ovl{\varepsilon}^{-1}$ is unramified (here $\omega_n$ is Serre's niveau $n$ character). In these exceptional cases, assuming $p>7$, we also give the presentation of $R^{\eta,\tau}_{\rhobar}$ which turns out somewhat complicated, cf.~section \ref{sec:examples f=1}.

Besides algorithmic aspects, our theory also unifies and conceptualizes Caruso--David--M\'ezard's work. More specificially, they introduced the notion of \emph{gene} $\bX(\tau,\rhobar)$ associated to $\rhobar$ and $\tau$, which is a purely combinatorial gadget, closely related to the combinatorial data inputted in our algorithm, that keeps track of the difference between the inertial weights of $\rhobar$ and $\tau$. While the motivation for $\bX(\tau,\rhobar)$ was to encode geometric features of a resolution of $R^{\eta,\tau}_{\rhobar}$ arising from integral $p$-adic Hodge theory, Caruso--David--M\'ezard conjectured that, surprisingly, $\bX(\tau,\rhobar)$ is in fact a complete invariant, cf.~\cite[Conjecture 5.1.6]{CDM2} and \cite[Conjecture 2]{CDM4}. Our model gives a geometric interpretation of $\bX(\tau,\rhobar)$, allowing us to confirm this:


%

%
%
 %

%
%

%

%

%

%
%
%
\begin{thm}[Theorem \ref{thm:indepen}]
\label{thm-intro-integral_domain} Assume $p>16f+7$.
Let $\rhobar$ be irreducible and $\tau$ be a non-scalar principal series tame inertial type.
The deformation ring $R_{\rhobar}^{\eta,\tau}$ depends only on $\bX(\tau,\rhobar)$, in an explicit way, and furthermore, is an integral domain.
\end{thm}
In fact, our algorithm can be interpreted as giving the right generalization of Caruso--David--M\'ezard's conjecture for general (i.e.~not necessarily irreducible) $\rhobar$ and general (i.e.~not necessarily principal series) $\tau$. It should be noted that in this more general setting, $R_{\rhobar}^{\eta,\tau}$ is no longer always a domain, and one can read off when this is so from our tables. 

Finally, we expect the explicit computations of the deformation ring to be useful for global applications, particularly for mod $p$ multiplicity one questions and Breuil's lattice conjecture in non-generic cases (see \cite[Theorem 10.1.1 and Theorem 8.2.1]{EGS}).%

\subsection{Methods}\label{subsect:methods}
When $\tau$ is sufficiently generic, a local model for $\cZ^{\tau}$ can be extracted from \cite{CEGSA} 
(see also \cite{MLM} for a perspective closer to the present work), owing to the fact that in the generic cases $\cZ^{\tau}$ agrees with the moduli stack $Y^{\eta,\tau}$ of Breuil--Kisin modules of type $\tau$ and there are standard local models for the latter. 
In the non-generic cases, the essential difficulty is that $\cZ^{\tau}$ is only a scheme theoretic image of a map $Y^{\eta,\tau}\to \Phi\text{-}\Mod^{\text{\'et}, 2}_K$ to the stack of rank $2$ \'{e}tale $\varphi$-modules, so that $\cZ^{\tau}$ is obtained from $Y^{\eta,\tau}$ by contracting the fibers of this map (the \emph{Kisin varieties}). This is the source of all complexities in the non-generic situation, and the main innovation of this work is to find a good group theoretic model for this contraction procedure.

The main idea in Theorem \ref{thm-main-intro} is to use deformation theory to find good approximations of the map $Y^{\eta,\tau}\rightarrow  \Phi\text{-}\Mod^{\text{\'et}, 2}_K$. 
Let $L\rG$ (resp.~$L^+\rG$) be the loop group (resp.~positive loop group) for $\GL_2$ with respect to $v(v+p)$, that is the functor  $R\mapsto L\mathrm{G}(R)\defeq  \GL_2\left(R[v]^{\wedge_{(v(v+p))}}\Big[\frac{1}{v(v+p)}\Big]\right)$ (resp. $L^+\mathrm{G}(R)\defeq  \GL_2\left(R[v]^{\wedge_{(v(v+p))}}\right)$) where $R$ is an $\cO$-algebra. We also have the ``Iwahori'' $L^+\cG$ and the first principal congruence subgroup $L_1^+\rG$ which are the inverse image in $L^+\rG$ of the upper triangular Borel, resp.~the trivial subgroup under the mod $v$ reduction. We set $\Gr_1\defeq[L^+_1\rG\backslash L\rG]$, which is a $\GL_2$-torsor over a (mixed characteristic, Beilinson--Drinfeld) affine Grassmannian.

Up to a $p$-adic completion which we suppress for the remainder of the introduction, we have $Y^{\eta,\tau}=\big[L\rG^{\tau}/_{\phz} (L^+\cG)^f\big]$ as a quotient by a (shifted) $\phz$-conjugation action (here $\phz$ sends $v$ to $v^p$) and 
$L\rG^{\tau}$ is a particular closed subset of $L\rG^f$ encoding certain elementary divisor bounds and the combinatorial data of $\tau$ (cf.~section \ref{subsec:loop groups} for the precise definitions). 
The basic idea, as in \cite{MLM}, is to ``straighten'' the $\phz$-action in the above as much as possible: due to the contraction effect of $\phz$, the $\phz$-conjugation action is equivalent to the left translation action, provided one works modulo fixed powers of $p$ and on a small enough subgroup. Unlike the generic case in \emph{loc.~cit}., it is not always possible to do this in characteristic $p$ at the $L^+\cG$ level, but if $p$ is large enough, it is possible to do so at the $L_1^+\rG$-level (Lemma \ref{lem:strightening}).
This shows that $Y^{\eta,\tau}$ is congruent to $[\Gr_1^{\tau}/ B^f\text{-sh.cnj}]$ modulo $p^{cp}$ for some absolute constant $c$ (here the $B^f$-action is via shifted conjugation, and $\Gr_1^{\tau}$ is $[(L_1^+\rG)^f\backslash L\rG^{\tau}]$).

On the other hand, $\Phi\text{-}\Mod^{\text{\'et}, 2}_K=[L\rG^{f}/_{\phz} L\rG^{f}]$, and it is never possible to straighten the action of the larger group $L\rG^f$. However, since we are only interested in the scheme theoretic image of $Y^{\eta,\tau}$ in $\Phi\text{-}\Mod^{\text{\'et}, 2}_K$, it suffices to work instead with $\big[L\rG^{\textnormal{bd}}/_{\phz} (L^+\rG)^f\big]$ where $L\rG^{\textnormal{bd}}$ is a suitable bounded region in $L\rG^f$ containing the orbit of $L\rG^{\tau}$ (denoted by $L\rG^{\textnormal{bd},(v+p)v^{  \mu}}$ in the main text), and once again the $\phz$-action can be straightened on $(L_1^+\rG)^f$. We remark that  $L\rG^{\textnormal{bd}}$ descends to $\Gr_1^f$, inducing a subvariety $\Gr_1^{\textnormal{bd}}$. 

The upshot so far is that $Y^{\eta,\tau}\rightarrow  \Phi\text{-}\Mod^{\text{\'et}, 2}_K$ is well-approximated by the natural map 
\[[\Gr_1^{\tau}/ B^f\text{-sh.cnj}]\to [\Gr_1^{\textnormal{bd}}/ \GL_2^f\text{-sh.cnj}]\]
modulo $p^{cp}$. It further turns out that the scheme theoretic image $\cZ^{\textnormal{mod},\tau}$ of this is congruent to $\cZ^{\tau}$ modulo $p^{cp-1}$, and the model $\tld{\cZ}^{\textnormal{mod},\tau}$ in Theorem \ref{thm-main-intro} is the pullback of $\cZ^{\textnormal{mod},\tau}$ to the natural $(\GL_2)^f$-torsor of the target. 
Finally, to prove Theorem \ref{thm-main-intro}, we need to show that the above congruences can be lifted to characteristic $0$, at least locally. This is achieved by a detailed study the geometry of $\tld{\cZ}^{\textnormal{mod},\tau}$ to bound the $p^\infty$-torsion of the obstruction groups for such lifting problems (whose exponent can be as large as linear in $f$, as alluded to in Remark \ref{rmk:bounds on p}), and imposing bounds on $p$ required to overcome the obstructions. A by-product of this geometric study is a control on the singularities of $\tld{\cZ}^{\textnormal{mod},\tau}$, which is robust enough that it can be transferred to $\cZ^{\tau}$ through a mod $p$ congruence, thus yielding Theorems \ref{thm:intro:Gorenstein}, \ref{thm:intro:normal}.

\subsection{Acknowledgements}
This work was initiated while B.LH. was visiting Paris as Chaire d'excellence of the Fondation de Sciences Math\'ematiques de Paris, and we would like to thank the FSMP for making this visit possible.
A.M. thanks Agn\`es David and Xavier Caruso for having imagined, tested and developed together the conjectures proved in this article.
B.LH, A.M. and S.M.~thank C.~Breuil for useful feedback on a first draft of this article, and S.M. and A.M. thank Olivier Wittenberg for his patient explanations around complete intersection.
The authors express their gratitude to the special trimester ``The Arithmetic of the Langlands Program'', held at the Hausdorff Institute of Mathematics, where a substantial portion of this paper has been written.
S.M. and A.M. are member of the Institut Universitaire de France and acknowledge its support.
B.LH. acknowledges support from the National Science Foundation under grant Nos.~DMS-1128155, DMS-1802037, DMS-2302619 and the Alfred P. Sloan Foundation.

\subsection{Notation}
\label{subsec:notations}

We fix once and for all a separable closure $\ovl{K}$ and let $G_K \defeq \Gal(\ovl{K}/K)$. 
If $K$ is a nonarchimedean local field, we let $I_K \subset G_K$ denote the inertial subgroup.
We fix a prime $p$.
Let $E$ be a finite extension $\Q_p$ with ring of integers $\cO$, uniformizer $\varpi\in \cO$ and residue field $\F$ (which we assume is large enough).

We consider the group $G\defeq\GL_2$ (defined over $\Z$).
We write $B$ for the subgroup of upper triangular matrices, $T \subset B$ for the split torus of diagonal matrices and $Z \subset T$ for the center of $G$.   
Let $X^*(T)$ be the group of characters of $T$ which we identify with $\Z^2$ in the standard way and $\eta\in X^*(T)$ the element corresponding to $(1,0)\in \Z^2$.

We write $W$ (resp.~$\tld{W}$) for the Weyl group (resp. the extended affine Weyl group) of $G$, which act naturally on $X^*(T)$.
Thus $W=\{1,w_0\}$ is the set of permutation on $2$ elements and $\tld{W} = X^*(T) \rtimes W$.
We use the notation $t_{\nu} \in \tld{W}$ to denote the image of $\nu \in X^*(T)$ in $\tld{W}$. 

Let $\alpha$ denote the positive root of $G$ and $\langle \ ,\,\rangle$ the duality pairing on $X^*(T)\times X_*(T)$, so a weight $\lambda\in X^*(T)$ is \emph{dominant} if $0\leq \langle\lambda,\alpha^\vee\rangle$.
We set $X^0(T)$ to be the subgroup consisting of characters $\lambda\in X^*(T)$ such that $\langle\lambda,\alpha^\vee\rangle=0$. 

Let now $K$ be a finite unramified extension of $\Qp$ of degree $f$, with ring of integers $\cO_K$ and residue field $k$.
Thus $\cO_K=W(k)$, and denote by $\varphi$ the arithmetic Frobenius acting on $W(k)$ (i.e.~acting by rising to the $p$-power on the residue field). 
Let $G_0 \defeq \Res_{\cO_K/\Z_p} G_{/\cO_K}$, $T_0 \defeq \Res_{\cO_K/\Z_p} T_{/\cO_K}$, and $Z_0 \defeq \Res_{\cO_K/\Z_p} Z_{/\cO_K}$. 
We assume that $\cO$ contains the image of every ring homomorphism $\cO_K \ra \ovl{\Z}_p$ and write $\cJ\defeq \Hom_{\Zp}(\cO_K,\cO)$.
We define $\un{G} \defeq (G_0)_{/\cO}$ and fix an identification of $\un{G}$ with the split reductive group $G_{/\cO}^{\cJ}$. 
We similarly define and identify $\un{T},$ and $\un{Z}$.
The notations $\un{W}$, $\tld{\un{W}}$ are clear as should be the natural isomorphisms $X^*(\un{T}) = X^*(T)^{\cJ}$.  
Given an element $j\in\cJ$, we use a subscript notation to denote $j$-components obtained from the isomorphism $\un{G}\cong G_{/\cO}^{\cJ}$ (so that, for instance, given an element $\tld{w}\in \tld{\un{W}}$ we write $\tld{w}_j$ to denote its $j$-th component via the induced identification $\tld{\un{W}}\cong \tld{W}^{\cJ}$).
For sake of readability, we abuse notation and still write $w_0$ to denote the longest element in $\un{W}$, and $\eta\in X^*(\un{T})$ for the element corresponding to $(1,0)\in\Z^2$ in all embeddings.

The Frobenius automorphism $\varphi$ of $\cO_K$ induces an automorphism $\pi$ on $X^*(\un{T})\cong X_*(\un{T}^\vee)$ by the formula $\pi(\lambda)_\sigma = \lambda_{\sigma \circ \varphi^{-1}}$ for all $\lambda\in X^*(\un{T})$ and $\sigma: \cO_K \ra \cO$.
We similarly define an automorphism $\pi$ of $\un{W}$ and $\tld{\un{W}}$.

Recall that we fixed a separable closure $\ovl{K}$ of $K$.
We choose $\pi \in \ovl{K}$ such that $\pi^{p^f-1} = -p$ and let $\omega_K : G_K \ra \cO_K^\times$ be the character defined by $g(\pi) = \omega_K(g) \pi$, which is independent of the choice of $\pi$.
We fix an embedding $\sigma_0: K \into E$ and define $\sigma_j \defeq  \sigma_0 \circ \phz^{-j}$, which identifies $\cJ = \Hom(k, \F) = \Hom_{\Qp}(K, E)$ with $\Z/f \Z$. 
In particular, the automorphism $\pi$ on $X^*(\un{T})$ satisfies $(\pi(\lambda))_j=\lambda_{j+1}$.
We write $\omega_f:G_K \ra \cO^\times$ for the character $\sigma_0 \circ \omega_K$.

Let $\eps$ denote the $p$-adic cyclotomic character.
We fix normalization so that the $p$-adic cyclotomic character $\eps$ has Hodge--Tate weight $\{1\}$ for every $\kappa:K\into E$.
%
%
%
%

%
%
%
%
%

%

%
%

%
%

\section{Tame inertial types and Breuil--Kisin modules}

\subsection{Tame inertial types and Galois representations}

\subsubsection{Tame inertial types}
\label{subsub:TIT}
An {\it inertial type} for $K$ over $\cO$ (resp.~over $\F$) is an homomorphism  $\tau:I_{K}\ra\GL_2(\cO)$ (resp.~$\tau:I_{K}\ra\GL_2(\F)$) with open kernel and which extends to the Weil group of $G_K$.
An inertial type is \emph{tame} if it factors through the tame quotient of $I_{K}$.
Given $s\in \un{W} $ and $\mu\in X^*(\un{T})$, we have a tame inertial type $\tau(s,\mu):I_{K}\ra\GL_2(\cO)$ defined as follows: let $r$ be the order of $s_0s_1\dots s_{f-1}\in W$, and define $\bm{\alpha}_{k'}\defeq (\prod_{m'=0}^{k'-1}s_{f-1-m'}^{-1})(\mu_{f-k'})$.
Then 
\[
\tau(s,\mu)\defeq \bigg(\sum_{i'=0}^{rf-1}\bm{\alpha}_{i'}p^{i'}\bigg) (\omega_{fr}).
\]
In particular if $\nu=(\nu_j)_{j\in\cJ}\in X^*(\un{Z})\cong \Z^{\cJ}$ then
\begin{equation}
\label{eq:twist}
\tau(s,\mu+\nu)\cong \tau(s,\mu)\otimes_{\cO}\omega_f^{\sum_{j\in\cJ}\nu_jp^j}.
\end{equation}

Any tame inertial type is isomorphic to some $\tau(s,\mu)$.

More explicitly, if $s=(s_j)_{\cJ}\in \un{W}$ is such that $\prod_{j=0}^{f-1}s_{j}=\Id$ then
$\tau(s,\mu)\cong\omega_{f}^{\gamma}\oplus \omega_{f}^{\gamma'}$ where 
\begin{align}
\gamma&\defeq \sum_{j=0}^{f-1}p^j\Big(\mu_{f-j,(\prod_{i=0}^{f-1-j}s_{i})(1)}\Big)\nonumber\\
\label{eq:gamma}
\gamma'&\defeq \sum_{j=0}^{f-1}p^j\Big(\mu_{f-j,(\prod_{i=0}^{f-1-j} s_{i})(2)}\Big),
\end{align}
noting that $\prod_{i=0}^{f-1-j}s_{i}=\prod_{i=0}^{j-1}s_{f-1-i}^{-1}$.

Similarly, if $s=(s_j)_{\cJ}\in \un{W}$ is such that $\prod_{j=0}^{f-1}s_{j}=(12)$ then $\tau(s,\mu)\cong\omega_{2f}^{h}\oplus \omega_{2f}^{p^fh}$ where 
\begin{equation}
\label{eq:h}
h\defeq \sum_{j=0}^{f-1}p^j\Big(\mu_{f-j,(\prod_{i=0}^{f-1-j} s^{-1}_{i})(2)}\Big)+p^f\Big(\sum_{j=0}^{f-1}p^j\mu_{f-j,(\prod_{i=0}^{f-1-j} s^{-1}_{i})(1)}\Big).
\end{equation}

\begin{rmk}
\label{rmk:congruence}
Let $\lambda,\lambda'\in X^*(\un{T})\stackrel{\sim}{\ra}(\Z^2)^f$.
Then $\lambda\equiv\lambda' \mod (p-\pi^{-1})X^0(\un{T})$ if and only if 
\[
\sum_{j=0}^{f-1}p^j\lambda_{j,1}+p^f\Big(\sum_{j=0}^{f-1}p^j\lambda_{j,2}\Big)\equiv
\sum_{j=0}^{f-1}p^j\lambda'_{j,1}+p^f\Big(\sum_{j=0}^{f-1}p^j\lambda'_{j,2}\Big)\mod p^{2f}-1.
\]
\end{rmk}

We say that $(s,\mu)$ is \emph{a presentation} for the tame inertial type $\tau(s,\mu)$.
Note that any tame inertial type will have infinitely many presentations since
\begin{equation}
\label{ciao}
\tau(s,\mu)\cong\tau(\sigma s \pi(\sigma)^{-1}, \sigma(\mu)+p\nu-\sigma s \pi(\sigma)^{-1}\pi(\nu))
\end{equation}
for any $(\sigma,\nu)\in \un{W}\times X^*(\un{T})$.
We will also record a presentation $(s,\mu)$ by the element $\tld{w}^*(\tau)=s^{-1}t_{\mu}\in\tld{\un{W}}$.

\begin{lemma}\label{lem:write:tau}
Let $\tau: I_K\rightarrow \GL_2(\cO)$ be a tame inertial type.

Then, there exists $n\in \Z$, $(k_j)_{j\in\cJ}\in\big\{0,\dots,\frac{p+1}{2}\big\}^{\cJ}$ and $s\in W^{\cJ}$ such that
\[
\tau\cong \tau\big(s,\big((k_j,0)\big)_{j\in\cJ}\big)\otimes_{\cO} \omega_f^n
\]
and moreover $s_j=\Id$ if $k_j=0$.
\end{lemma}
\begin{proof}
In this proof, given two tame inertial types $\tau$ and $\tau'$ we write $\tau\sim \tau'$ if $\tau\cong \tau'\otimes_{\cO}\omega_f^{n}$ for some $n\in\Z$.
Let $(s',\mu')$ be a presentation of $\tau$.
Using \eqref{eq:twist} and applying repeatedly \eqref{ciao} with $\sigma=\Id$, we see that $\tau(s',\mu')\sim \tau(s',\mu'')$ where $\mu''\in X^*(\un{T})$ is such that $\langle \mu''_j,\alpha^\vee\rangle\in [-\frac{p+1}{2},\dots,\frac{p+1}{2}]$ for all $j\in\cJ$ (see also \cite[Lemma 2.3.3]{MLM}).
Hence, using again \eqref{eq:twist}, we have $\tau(s',\mu'')\sim \tau\big(s',(k'_j,0)_{j\in\cJ}\big)$ where $|k'_j|\leq \frac{p+1}{2}$.
Finally, using \eqref{eq:twist} and applying repeatedly \eqref{ciao} with $\nu=0$, we obtain $\tau\big(s',(k'_j,0)_{j\in\cJ}\big)\sim \tau\big(s,(k_j,0)_{j\in\cJ}\big)$, where at each step $\sigma$ can be chosen so that $\sigma_js_j\sigma_{j+1}^{-1}=\Id$ when $k'_j=0$.
\end{proof}
\begin{defn}
\label{defn:pres}
A presentation $(s,\mu)$ of a tame inertial type $\tau$ is \emph{small} if $0\leq \langle \mu_j,\alpha^\vee \rangle\leq \frac{p+1}{2}$ for all $j\in \cJ$ and moreover $s_j=\Id$ whenever $\langle \mu_j,\alpha^\vee\rangle=0$.
\end{defn}

If $\tau$ is a tame inertial type over $\cO$ we let $\ovl{\tau}\defeq\tau\otimes_{\cO}\F$.
This construction gives a bijection between isomorphism classes of tame inertial types over $\cO$ and tame inertial types over $\F$ so that the whole discussion above holds for the latter.

\begin{lemma}
\label{lem:s_or}
Let $\tau=\tau(s,\mu)$ be a tame inertial type with small presentation $(s,\mu)$.
For $j'\in\cJ'$ define $\bf{a}^{\prime(j')}\defeq\sum_{i'=0}^{rf-1}\bm{\alpha}_{-j'+i'}p^{i'}$.
There exists a unique element $(s'_{\orient,j'})_{j'\in\cJ'}\in W^{\cJ'}$ such that $(s'_{\orient,j'})^{-1}(\bf{a}^{\prime(j')})$ is strictly dominant for all $j'\in\cJ'$.
Moreover the embedding $\sigma_0 : k \into \F$ and $s$ can be chosen so that $(s'_{\orient,rf-1})=\Id$.
\end{lemma}
\begin{proof} 
We can assume without loss of generality that $\langle \mu_0,\alpha^\vee\rangle>0$.
As $\bm{\alpha}_{0}=\mu_0$ we thus have $\bf{a}^{\prime(0)}=p^{rf-1}\mu_0+\sum_{i'=0}^{rf-2}\bm{\alpha}_{-j'+i'}p^{i'}$ which is dominant since $\langle \mu_0,\alpha^\vee\rangle>0$ and $p-1>\langle \bm{\alpha}_{-j'+i'},\alpha^\vee\rangle$ for $i'=0,\dots, rf-2$.

We can now conclude by decreasing induction: using the relation
\[
\bf{a}^{\prime(rf-j')}=\frac{\bf{a}^{\prime(rf-j'+1)}-\bm{\alpha}_{-(rf-j'+1)}}{p}+p^{rf-1}\bm{\alpha}_{j'-1}
\]
we see that either $\bm{\alpha}_{j'-1}\notin X^0(T)$ and hence $\bf{a}^{\prime(rf-j')}$ is strictly dominant if and only if 
$\bm{\alpha}_{j'-1}$ is strictly dominant (in which case $s'_{\orient,rf-j'}$ is uniquely determined), or $\bm{\alpha}_{j'}\in X^0(T)$ and hence $\bf{a}^{\prime(rf-j')}$ is strictly dominant if and only if $\bf{a}^{\prime(rf-j'+1)}$ is strictly dominant (in which case $s'_{\orient,rf-j'}=s'_{\orient,rf-j'+1}$ and $s'_{\orient,rf-j'+1}$ is uniquely determined the inductive hypothesis).
\end{proof}

\subsubsection{Galois deformation rings}

We let $\rhobar:G_{K}\ra\GL_2(\F)$ be a continuous Galois representation.
Let $\mathcal{C}_\cO$ be the category of Noetherian complete local $\cO$-algebras with residue field $\F$ and local $\cO$-algebra homomorphisms. The functor that assigns to $(A,\fm_A) \in \mathcal{C}_\cO$ the set of lifts $\rho_A: G_{K} \ra \GL_2(A)$ of $\rhobar$ is representable by $R_{\rhobar}^\square$, the (unrestricted) lifting ring of $\rhobar$.


Given a tame inertial type $\tau$ over $\cO$ we let $R_{\rhobar}^{\eta,\tau}$ be the reduced $\cO$-flat quotient of $R_{\rhobar}^\square$ such that the 
$\ovl{\Q}_p$-points of
$\Spec R_{\rhobar}^{\eta,\tau}[1/p]$ correspond to the subset of $\rho: G_K \ra \GL_2(\ovl{\Q}_p)$ (inside $\Spec(R_{\rhobar}[1/p])$) which are potentially Barsotti--Tate and such that the \emph{covariant} Weil--Deligne inertial type is isomorphic to $\tau\otimes_{\cO} \ovl{\Q}_p$.
The rings $R_{\rhobar}^{\eta,\tau}$ are known in ``generic'' cases (cf. \cite{EGS,MLM}):
\begin{thm}
\label{thm:generic_cases}
Assume that $\tau$ has a presentation $(s,\mu)$ where $2\leq \langle \mu_j,\alpha^\vee\rangle\leq p-2$.
Then either $R_{\rhobar}^{\eta,\tau}=0$ or
\[
R_{\rhobar}^{\eta,\tau}\cong \cO[\![Z_{1},\ldots,Z_{f+4-n},X_1,Y_1,\ldots,X_n,Y_n]\!]/(X_iY_i-p,\, i=1,\ldots,n)
\]
for some $n\in\{0,\dots,f\}$.
\end{thm}

One of the main goals of this paper, accomplished in \S \ref{subsub:Gal:def}, is to provide the analog of Theorem \ref{thm:generic_cases} in highly non-generic situations.%

\subsection{Breuil--Kisin modules and Emerton--Gee stack}

\subsubsection{}
We introduce the necessary background on Breuil--Kisin modules with tame descent data.
%

%

Let $\tau=\tau(s,\mu)$ be a tame inertial type with presentation $(s, \mu)$ which we fix throughout this section.
Recall that $r\in\{1,2\}$ is the order of $s_0 s_{1} s_{2} \cdots s_{f-1} \in W$.
Let $K'/K$ be the unramified extension of degree $r$ contained in $\ovl{K}$, set  $f'\defeq fr$, $e' \defeq p^{f'}-1$, and identify $\Hom_{\Qp}(K',E)$ with $\Z/f' \Z$ via $\sigma_{j'} \defeq \sigma'_0\circ \phz^{-j'} \mapsto j'$ where $\sigma'_0:K' \iarrow E$ is a fixed choice of an embedding extending $\sigma_0:K \iarrow E$. %
(In particular, restriction of embeddings corresponds to reduction modulo $f$ in the above identifications.)

We let $\pi'\in \ovl{K}$ be an $e'$-th root of $-p$,  $L' \defeq K'(\pi')$ and $\Delta' \defeq \Gal(L'/K') \subset \Delta \defeq \Gal(L'/K)$. 
We have the character $\omega_{K'}(g) \defeq  \frac{g(\pi')}{\pi'}$ for $g \in \Delta'$ (which does not depend on the choice of $\pi'$) and given an $\cO$-algebra $R$, we set $\fS_{L', R} \defeq (W(k') \otimes_{\Zp} R)[\![u']\!]$. 
The latter ring is endowed with the endomorphism $\varphi:\fS_{L', R} \ra \fS_{L', R}$ acting as the Frobenius on $W(k')$ and sending $u'$ to $(u')^{p}$, and is endowed moreover with an action of $\Delta$ by $g'(u') = \frac{g'(\pi')}{\pi'} u' = \omega_{K'}(g') u'$ if $g'\in\Delta'$ and, letting $\sigma^f \in\Delta$ be the lift of the $p^f$-Frobenius on $W(k')$ which fixes $\pi'$, then $\sigma^f$  acts in natural way on $W(k')$ and trivially on $u'$ (all the endomorphism above act trivially on $R$ by default).
Finally, $v\defeq  (u')^{e'}$, 
\[
\fS_R \defeq (\fS_{L', R})^{\Delta = 1} = (W(k) \otimes_{\Zp} R)[\![v]\!]
\]
and $E(v) \defeq v + p = (u')^{e'} + p$.

\begin{defn}
A Breuil--Kisin module $\fM$ of rank $2$ over $\fS_{L', R}$ with descent data of type $\tau$ and height $\leq 1$ is the datum of:
\begin{enumerate} 
\item  a rank $2$ projective $\fS_{L', R}$-module $\fM$;
\item an injective $\fS_{L', R}$-linear map $\phi_\fM:\phz^*(\fM)\ra\fM$ whose cokernel is annihilated by $E(v)$; and
\item a semilinear action of $\Delta$ on $\fM$ which commutes with $\phi_{\fM}$, and such that, for each $j' \in \Hom_{\Qp}(K',E)$, 
\[
(\fM\otimes_{W(k'),\sigma_{j'}}R) \mod u' \cong \tau^{\vee} \otimes_{\cO} R 
\]  
as $\Delta'$-representations.
\end{enumerate}
\end{defn}
We write $Y^{[0,1],\tau}$ to be the groupoid of such objects (cf.~ \cite[Definition 5.1.3]{MLM}).
We also define $Y^{\eta,\tau}(R)\subset Y^{[0,1],\tau}(R)$ to be the subgroupoid  satisfying the additional \emph{determinant condition}
\[
\det (\phi_{\fM})\in (R[\![v]\!])^{\times} (v+p).
\]

We consider $\fM^{(j')}\defeq \fM\otimes_{W(k'),\sigma_{j'}}R$ as a $R[\![u']\!]$-submodule of $\fM$ in the standard way, so it is endowed with a semilinear action of $\Delta'$. The Frobenius $\phi_{\fM}$ induces $\Delta'$-equivariant morphisms $\phi_\fM^{(j')}:\phz^*(\fM^{(j'-1)})=(\phz^*(\fM))^{(j')}\ra \fM^{(j')}$   
(here the pull back on the first object is with respect to the $R$-algebra map $\phz: R[\![u']\!]\ra R[\![u']\!] $ such that $u'\mapsto {u'}^p$).
We remark that, by letting $\tau'$ denote the tame inertial type for $K'$ obtained from $\tau$ via the identification $I_{K'}=I_K$ induced by the inclusion $K'\subseteq \ovl{K}$, the semilinear action of $\Delta$ induces an isomorphism $\iota_{\fM}:(\sigma^f)^*(\fM) \cong \fM$ (see \cite[\S 6.1]{LLLM}) as elements of $Y^{[0,1], \tau'}(R)$.

Let $\fM \in Y^{[0, 1], \tau}(R)$.
Recall that an \emph{eigenbasis} of $\fM$ is a collection of bases ${\beta}^{(j')}=(f_1^{(j')},f_2^{(j')})$ for each $\fM^{(j')}$ such that $\Delta'$ acts on $f_i^{(j')}$ via the character $\omega_{f'}^{-\mathbf{a}^{\prime\,(0)}_{i}}$ and such that $\iota_{\fM} ((\sigma^f)^*(\beta^{(j')})) = \beta^{(j' + f)}$ for all $j'\in \Hom_{\Qp}(K',E)$.
Given an eigenbasis $\beta$ for $\fM$, we let $C^{(j')}_{\fM, \beta}$ be the matrix of $\phi_\fM^{(j')}:\phz^*(\fM^{(j' - 1)})\ra \fM^{(j')}$ with respect to the bases $\phz^*(\beta^{(j'-1)})$ and $\beta^{(j')}$ and set
\[
A^{(j')}_{\fM,\beta}\defeq 
\Ad\left(
(s'_{\orient,j'})^{-1}  (u^{\prime})^{-\bf{a}^{\prime\,(j')}}
\right)(C^{(j')}_{\fM,\beta})
\]
for $j'\in \Hom_{\Qp}(K',E)$.
\begin{lemma}
Let $\fM \in Y^{[0, 1], \tau}(R)$ with eigenbasis ${\beta}$.
The element $A^{(j')}_{\fM,\beta}$ has coefficients in $R[\![v]\!]$ and is upper triangular modulo $v$.
Finally, it only depends on the restriction of $j'$ to $K$.
\end{lemma}
\begin{proof}
By the definition of eigenbases and of the action of $\Delta'$ on $\fS_R\otimes_{W(k'),\sigma_{j'}}R$, we see that $(C^{(j')}_{\fM,\beta})_{\alpha}\in \big((u')^{\langle \bf{a}^{\prime\,(j')},\alpha^\vee\rangle}R[\![v]\!]\big)\cap R[\![u']\!]$ for $\alpha\in \Phi$.
Explicitly, letting $\delta_{j'}\in\{0,1\}$ be such that $\delta_{j'}=0$ if and only if $\bf{a}^{\prime\,(j')}$ is dominant, we have for $\alpha\in\Phi^+$ that $(C^{(j')}_{\fM,\beta})_{\alpha}\in (u')^{e'\delta_{j'}+\langle \bf{a}^{\prime\,(j')},\alpha^\vee\rangle}R[\![v]\!]$ and $(C^{(j')}_{\fM,\beta})_{-\alpha}\in (u')^{e'(1-\delta_{j'})+\langle \bf{a}^{\prime\,(j')},-\alpha^\vee\rangle}R[\![v]\!]$.
Thus for $\alpha\in \Phi^+$, the $\alpha$-entry of $\Ad\left(
(u^{\prime})^{-\bf{a}^{\prime\,(j')}}
\right)(C^{(j')}_{\fM,\beta})$ is in $v^{\delta_{j'}}R[\![v]\!]$ and the $-\alpha$-entry is in $v^{1-\delta_{j'}}R[\![v]\!]$.
By the definition of $\delta_{j'}$ and the fact that $s'_{\orient,j'}=\Id$ if and only if $\bf{a}^{\prime\,(j')}$ is dominant, we conclude that $A^{(j')}_{\fM,\beta}$ is upper triangular modulo $v$.
The fact that $A^{(j')}_{\fM,\beta}$ depends only on $j'$ modulo $f$ follows from \cite[Lemma 6.2, Proposition 6.9]{LLLM}. 
\end{proof}

\subsubsection{\'Etale $\Phi$-modules.}
\label{subsub:etale:Phi}
We recall the notion of \'etale $\Phi$-modules.

Recall that $\cO_{\cE,K}$ denotes the $p$-adic completion of $(W(k)[\![v]\!])[1/v]$. 
It is endowed with a continuous Frobenius morphism $\phz$ extending the Frobenius on $W(k)$ and such that $\phz(v)=v^p$.
Given a $p$-adically complete Noetherian $\cO$-algebra $R$ we let $\Phi\text{-}\Mod^{\text{\'et}, 2}_K(R)$ be the groupoid consisting of projective modules $\cM$ of rank $2$ over $\cO_{\cE,K}\widehat{\otimes}_{\Zp}R$ endowed with a Frobenius semilinear endomorphism $\phi_{\cM}:\cM\ra\cM$  inducing an isomorphism on the pull-back: $\Id\otimes_{\phz}\phi_{\cM}:\phz^*(\cM)\stackrel{\sim}{\longrightarrow}\cM$.

We similarly define the ring $\cO_{\cE,L'}$, with Frobenius $\phz$, and the groupoid $\Phi\text{-}\Mod^{\text{\'et},2}_{dd,L'}(R)$ of rank $2$ \'etale $(\phz,\cO_{\cE,L'}\widehat{\otimes}_{\Zp}R)$-modules with descent data from $L'$ to $K$.

The groupoids $\Phi\text{-}\Mod^{\text{\'et}, 2}_K$,  $\Phi\text{-}\Mod^{\text{\'et}, 2}_{dd,L'}$ form fppf stacks over $\Spf\,\cO$
(see \cite[\S 3.1]{CEGSA}).

Given $\fM\in Y^{\eta,\tau}(R)$, $\fM \otimes_{\fS_{L',R}} (\cO_{\cE,L'}\widehat{\otimes}_{\Zp}R)$ is an object $\Phi\text{-}\Mod^{\text{\'et},2}_{dd,L'}(R)$, which we can descend to an \'etale $\phz$-module $\cM \in \Phi\text{-}\Mod^{\text{\'et},2}_{K}(R)$ by 
\[
\cM \defeq (\fM \otimes_{\fS_{L',R}} (\cO_{\cE,L'}\widehat{\otimes}_{\Zp}R))^{\Delta=1}.
\]
This defines a morphism of stacks $\eps_\tau: Y^{\eta,\tau}\ra\Phi\text{-}\Mod^{\text{\'et},2}_{K}$ which is representable by algebraic spaces, proper, and of finite presentation by \cite[Corollary 3.1.8(3), Proposition 3.3.5]{CEGSA} (and the fact that taking $\Delta$-invariants is an isomorphism of groupoids). 
Moreover, $\eps_\tau$ is independent of any $\tld{\un{W}}$-presentation of $\tau$.

Given $(\cM,\phi_\cM)\in \Phi\text{-}\Mod^{\text{\'et},2}_{K}(R)$, we decompose $\cM=\oplus_{j \in \cJ} \cM^{(j)}$ over the embeddings $\sigma_j: W(k)\ra\cO$, with induced maps $\phi_\cM^{(j)}:\cM^{(j-1)}\ra\cM^{(j)}$. 
The following proposition is a direct computation on the definition of the $A^{(j)}_{\fM,\beta}$:
\begin{prop}$($\cite[Proposition 5.4.2]{MLM}$)$ 
\label{prop:expeps}  
Let $\fM \in Y^{[0, 1], \tau}(R)$ and $\beta$ an eigenbasis of $\fM$.
Let $(s, \mu)$ be a small presentation of $\tau$.

Then there exists a basis $\fF$ for $\eps_\tau(\fM)$ such that the matrix of $\phi_{\eps_\tau(\fM)}^{(j)}$ with respect to $\fF$ is given by 
\[
A^{(j)}_{\fM, \beta} s^{-1}_j v^{\mu_j} = A^{(j)}_{\fM, \beta} \tld{w}^*(\tau)_j .
\]
\end{prop}

Finally, when $R$ is a complete local Noetherian $\cO$-algebra with finite residue field we have an exact functor
\begin{align*}
\bV^*_K:\Phi\text{-}\Mod^{\text{\'et},2}_{K}(R)&\ra\Rep^2_R(G_{K_\infty})
\end{align*}
establishing an anti-equivalence of categories (by the theory of fields of norms, cf.~\cite[\S 2.3 and \S 6.1]{LLLM}) and therefore a functor $T^*_{dd}: Y^{\eta,\tau}(R)\rightarrow \Rep^2_R(G_{K_\infty})$ defined as the composite of $\eps_{\tau}$ followed by $\bV^*_K$.
(Note that the formula of \emph{loc.~cit}.~is inaccurate and should be modified as follows: $\bV^*_K(\cM)=\Hom_R\big((\cM\otimes_{\cO_{\cE}\widehat{\otimes}R} \cO_{\cE^{\textnormal{un}}}\widehat{\otimes}R)^{\phz=1},R\big)$.)
Finally, we recall that given $(\cM,\phi_\cM)\in \Phi\text{-}\Mod^{\text{\'et}}_{K}(R)$, we can define an \'etale $(\phz^f,\cO_{\cE,K}\widehat{\otimes}_{W(k),\sigma_0}R)$-module obtained as the $f$-fold composite of the partial Frobenii acting on $\cM^{(0)}$ (see for instance \cite[\S 2.3]{LLLM}).

Let $\cZ^\tau$ be the scheme theoretic image of  $\eps_\tau$.
This is the moduli stack of tame potentially Barsotti--Tate representations of type $\tau$ constructed and studied in \cite[\S 5.1]{CEGSA}.

A point $\rhobar\in \cZ^\tau(\F)$ gives rise to a mod $p$-representation of $G_K$.
Then $R^{\eta,\tau}_{\rhobar}$, the tame potentially Barsotti--Tate deformation ring of type $\tau$, is a  versal ring to $\cZ^\tau$ at $\rhobar$ (\cite[Corollary 5.2.19]{CEGSA}).

\section{Models}
\label{sec:models}
Recall from \S~\ref{subsub:etale:Phi} the proper, birational morphism $Y^{\eta,\tau}\ra \cZ^\tau$ which is an isomorphism on generic fibers.

\subsection{Loop groups and open charts}
\label{subsec:loop groups}
Given a Noetherian $\cO$-algebra $R$ let $R[v]^{\wedge_{(v(v+p))}}$ denote the $(v(v+p))$-adic completion of $R[v]$.
We denote by $R[v]\left[\frac{1}{v(v+p)}\right]_{\leq0}\subset R[v]^{\wedge_{(v(v+p))}}\left[\frac{1}{v(v+p)}\right]$ the subring of elements the form $\frac{P}{(v(v+p))^m}$ with $P\in R[v]$ such that $\deg P\leq2m$.

We define:
\begin{align*}
L\mathrm{G}(R)&\defeq  \GL_2\left(R[v]^{\wedge_{(v(v+p))}}\Big[\frac{1}{v(v+p)}\Big]\right);
\\
L^+ \mathrm{G}(R)&\defeq \GL_2\left(R[v]^{\wedge_{(v(v+p))}}\right);\\
L^{-} \mathrm{G}(R)&\defeq \left\{A \in L \G(R) \text{ and $A$ has coefficients in $ R[v]\left[\frac{1}{v(v+p)}\right]_{\leq0}$} \right\}
\end{align*}
Now we have surjections
\begin{align*}
\ev^+:L^+\rG&\onto \GL_2
\\
\ev^-:L^-\rG&\onto \GL_2
\end{align*}
obtained by evaluation modulo $v$ and $1/v$ respectively.

We define
\begin{align*}
L^+_1 \mathrm{G}(R)&\defeq \ker \ev^+\\
L^{-}_1 \mathrm{G}(R)&\defeq \ker \ev^-\\
L^+\cG(R)&
\defeq (\ev^+)^{-1}(B)
\\
L^{--}\cG(R)&\defeq
(\ev^+)^{-1}(\ovl{N})
\end{align*}
Note that the functors $L^+\cG$ and $L^{--}\cG$ have a slightly different meaning than the corresponding functors in \cite{MLM}, however they coincide after $p$-adic completion.
We define a closed subfunctor $\cA(\eta)\subset L\mathrm{G}$ whose $R$-valued points consist of $A\in L\mathrm{G}(R)$ satisfying 
\begin{enumerate}
\item 
\label{it:Adm:1}
$
\det A\in \big(R[v]^{\wedge_{(v(v+p))}}\big)^\times(v+p)$;
\item 
\label{it:Adm:2}
$A\in \Mat_2(R[v]^{\wedge_{(v(v+p))}})
$ and upper triangular mod $v$;
\item 
\label{it:Adm:3}
$(v+p)A^{-1}$ is upper triangular mod $v$.
\end{enumerate}
Note that $\cA(\eta)(\ovl{\F})$ identifies with 
\[
\bigcup_{\tld{z}_j\in \Adm(\eta)}L^+\cG(\ovl{\F})\tld{z}_jL^+\cG(\ovl{\F})
\]
where $\Adm(\eta)=\{\begin{pmatrix}v&0\\0&1\end{pmatrix},\begin{pmatrix}1&0\\0&v\end{pmatrix}, \begin{pmatrix}0&1\\v&0\end{pmatrix}\}$ is the $\eta$-admissible set.

Given a type $\tau$ with small presentation $(s,\mu)\in \un{W}\times X^*(\un{T})$ we define
\begin{equation}
\label{eq:LGtau}
L\mathrm{G}^{\tau}(R)\defeq \prod_{j\in\cJ}\cA(\eta)(R)s_jv^{\mu_j}\subset L\mathrm{G}^{\cJ}(R)
\end{equation}
We also define
\[
L\mathrm{G}^{\textnormal{bd},(v+p)v^\mu}(R)
\defeq \left\{
(A_j) \in (L\mathrm{G}(R))^{\cJ},\text{such that :}
\begin{array}{l}
\det A_j\in \big(R[v]^{\wedge_{(v(v+p))}}\big)^\times  (v+p)v^{\langle \mu_j,\alpha^\vee\rangle}\\
\,\,\,\quad A_j\in \Mat_2(R[v]^{\wedge_{(v(v+p))}})\end{array}
\right\}.
\]
\begin{lemma}
\label{lem:basic}
\begin{enumerate}
\item
\label{eq:basic:1}
$L_1^+\G$ is a normal subgroup of both of $L^+\G$ and $L^+\cG$, and
\begin{align*}
L^+\G&=L_1^+\G\ltimes \GL_2\\
L^+\cG&=L_1^+\G\ltimes B
\end{align*}
\item
\label{eq:basic:2.5}
The multiplication maps $L^+_1\rG\times L^-\rG\ra L\rG$, $L^+\cG\times L^{--}\cG\ra L\rG$ are formally \'etale monomorphisms after $p$-adic completion.
\item 
\label{eq:basic:2}
$\cA(\eta)^{\wedge_p}$ has an affine open cover $\cA(\tld{z})^{\wedge_p}$ where $\tld{z}$ runs over $\Adm(\eta)$ (cf.~Table \ref{Table1FV}).
\end{enumerate}
\end{lemma}
\begin{proof}
Item \eqref{eq:basic:1} is clear.

Item \eqref{eq:basic:2.5} for the map $L^+\cG\times L^{--}\cG\ra L\rG$ follows from \cite[Lemmas 3.2.2, 3.2.6]{MLM} (using that $R[v]^{v(v+p)}=R[\![v]\!]$ on rings where $p$ is nilpotent), which implies the statement for $L^+_1\rG\times L^-\rG\ra L\rG$.

We prove item \eqref{eq:basic:2}.
By the previous item, $\cA^{\wedge_p}(\tld{z})\defeq L^+\cG \cdot L\cG^{--}\tld{z}\cap \cA(\eta)^{\wedge_p}=L^+\cG\cdot (L\cG^{--}\tld{z}\cap \cA(\eta)^{\wedge_p})$ is an open subfunctor (since $L^+\cG\backslash \cA(\eta)$ is finite type, cf.~the proof of \cite[Corollary 3.2.10]{MLM}).
Consideration on $\ovl{\F}$-points shows that $\cA^{\wedge_p}(\tld{z})$ for $\tld{z}\in\Adm(\eta)$ form an open cover.

Finally, if $R$ is an $\cO$-algebra where $p$ is nilpotent, $L^{--}\cG(R)\tld{z}$ consists of matrices $A$ of the form
\[
A=\begin{pmatrix} 1+\frac{1}{v}a & \frac{1}{v}b\\c & 1+\frac{1}{v}d\end{pmatrix}\tld{z}
\]
for $a,b,c,d,\in R\left[\frac{1}{v}\right]$,
and imposing the conditions \eqref{it:Adm:1}, \eqref{it:Adm:2} on it gives the explicit description in table \ref{Table1FV}.
\end{proof}

\begin{table}[H]
\captionsetup{justification=centering}
\caption[Foo content]{The affine cover of $L^+\cG\backslash\cA(\eta)^{\wedge_p}$.
}
\label{Table1FV}
\centering
\footnotesize
\adjustbox{max width=\textwidth}{
\begin{tabular}{| c || c | c | c |}
\hline
&&&\\
$\tld{z}$ & $t_\eta$ & $w_0t_\eta$&$t_{w_0\eta}$
\\
&&&\\
\hline
&&&\\
$\cA^{\wedge_p}(\tld{z})$&$\begin{pmatrix}(v+p)&0\\vx&1\end{pmatrix}$&$\begin{matrix}
\begin{pmatrix}X&1\\v&Y\end{pmatrix}\\
\\
XY+p=0
\end{matrix}$&$
\begin{pmatrix}1&y\\0&v+p\end{pmatrix}$
\\
&&&\\
\hline
\hline
\end{tabular}}
\end{table}

\subsection{Loop groups and moduli of Breuil--Kisin modules}
$Y^{\eta,\tau}$ has the following description as a quotient stack (cf.~\cite[5.2.1]{MLM}):
\begin{lemma}
\label{lem:iso:KM}
Let $\tau$ be a tame inertial type with small presentation $(s,\mu)\in \un{W}\times X^*(\un{T})$.
Then any $\fM\in Y^{\eta,\tau}$ has an eigenbasis Zariski locally on $Y^{\eta,\tau}$ and the assignment $\fM\mapsto (A^{(j)}_{\fM,\beta})_{j\in\cJ}\tld{w}^*(\tau)$ defines an isomorphism of $p$-adic formal algebraic stacks
\begin{equation}
\label{eq:iso:Kisin:mod}
Y^{\eta,\tau}\stackrel{\sim}{\longrightarrow}\bigg[L\G^{\tau}\Big/_{\phz}\prod_{j\in\cJ}L^+\cG\bigg]^{\wedge_p}
\end{equation}
and hence a morphism
\begin{equation}
\label{eq:mor:Kisin:mod}
Y^{\eta,\tau}\ra \bigg[L\G^{\textnormal{bd},(v+p)v^{\mu}}\Big/_{\phz}\prod_{j\in\cJ}L^+\G\bigg]^{\wedge_p}
\end{equation}
\end{lemma}

We have a morphism
\[
\iota:\bigg[L\G^{\textnormal{bd},(v+p)v^{\mu}}\Big/_{\phz}\prod_{j\in\cJ}L^+\G\bigg]^{\wedge_p}\longrightarrow
\Phi\text{-}\Mod^{\textnormal{\'et}, 2}_K
\]
sending the class of $A\defeq (A^{(j)})_{j\in\cJ}$ to the \'etale $\varphi$-module $\iota(A)$ which is free of rank $2$ and such that $\phi_{\iota(A)}^{(j)}: \iota(A)^{(j-1)}\ra \iota(A)^{(j)}$ has matrix $A^{(j)}$ in the standard basis.
\begin{prop}
\label{prop:innomable}
Assume $p-2>\max_j\langle\mu_j,\alpha^\vee\rangle$.
Let $\tau$ be a tame inertial type with a small presentation $(s,\mu)\in \un{W}\times X^*(\un{T})$.
We have a commutative diagram of $p$-adic formal algebraic stacks over $\Spf \cO$:
\begin{equation}
\label{eq:diag}
\xymatrix{
Y^{\eta,\tau}\ar^-{\ref{eq:mor:Kisin:mod}}[rr]\ar_{\eps_\tau}[ddr]\ar[dr]&&\bigg[L\G^{\textnormal{bd},(v+p)v^{\mu}}\Big/_{\phz}\prod_{j\in\cJ}L^+\G\bigg]^{\wedge_p}\ar@{^{(}->}^{\iota}[ddl]
\\
&\cZ^{\tau}\ar@{^{(}->}[d]\ar^-{\exists}@{..>}[ur]&\\
&\Phi\text{-}\Mod^{\textnormal{\'et}, 2}_K&
}
\end{equation}
where the hooked diagonal arrow is a closed immersion.
In particular, the dotted arrow exists and makes the diagram commute.
\end{prop}
\begin{proof}
Denote by $Y^{\textnormal{bd},\tau}$ the groupoid in the upper right vertex of the diagram \eqref{eq:diag}.

The external triangle is commutative by Proposition \ref{prop:expeps} and Lemma \ref{lem:iso:KM}.
The factorization of $\eps_\tau$ is by definition of $\cZ^\tau$.

Hence, the existence of the dotted arrow will follow once we prove that the diagonal hooked arrow is a closed immersion.
Since $\iota$ is proper (as $L\cG/L^+\cG$ is ind-proper and $Y^{\textnormal{bd},\tau}$ is a finite type $p$-adic formal algebraic stack), it suffices to show that it is a monomorphism.

We prove that for any Noetherian $\cO/\varpi^{a}$-algebra $R$, and any pair $A_1,A_2\in Y^{\textnormal{bd},\tau}(R)$,
the morphism $\iota$ induces a bijection
\[
\Hom_{Y^{\textnormal{bd},\tau}}(A_1,A_2)\stackrel{\sim}{\ra}\Hom_{\Phi\text{-}\Mod}(\iota(A_1),\iota(A_2)).
\]
The induced map is clearly injective, and we thus prove its surjectivity.
Assume that there exists $X=(X_j)_{\cJ}\in LG(R)^{\cJ}$ such that
\begin{equation}
\label{eq:fund:comparison:0}
A_1^{(j)}=
X_jA_2^{(j)}\Big(\phz(X_{j-1})\Big)^{-1}
\end{equation}
or, equivalently, 
\begin{equation}
\label{eq:fund:comparison}
\big(\phz(X_{j-1})\big)=(A_1^{(j)})^{-1}
X_{j}A_2^{(j)}
\end{equation}
for all $j\in\cJ$.

We show by induction on $a$ that \eqref{eq:fund:comparison} forces $X_j\in L^+G(R)$ for all $j\in\cJ$.

If $a=1$ then $p=0$ in $R$ and, noting that $\det((A_2^{(j)})^{-1})\det(A_1^{(j)})\in R^\times$ by assumption, we deduce from \eqref{eq:fund:comparison} that $\phz^{f}(\det(X_{j}))=u_j\det(X_{j})$ for a unit $u_j\in R^\times$. This shows that $\det(X_{j})\in R^\times$.

We now show that $X_{j}\in \Mat_2(R[\![v]\!])$.
Let $\kappa_j\in \Z$ be the pole order of $X_{j}$ at $v$, i.e.~$\kappa_j\in \Z$ is minimal such that  $v^{\kappa_j}X_{j}\in \Mat_2(R[\![v]\!])$.
As $\det(A_1^{(j)})\in R^\times v^{1+\langle\mu_j,\alpha^\vee\rangle}$ we deduce from \eqref{eq:fund:comparison} that
\[
1+\kappa_j+\langle\mu_j,\alpha^\vee\rangle\geq p\kappa_{j-1}
\]
which forces $\kappa_j\leq\max_i\{\frac{\langle \mu_i,\alpha\rangle+1}{p-1}\}\leq \frac{p+3}{2(p-1)}<1$.
We conclude that $X_{j}\in L^+\rG(R)$ for all $j\in\cJ$ when $a=1$.

Assume the assertion up to $a-1$.
We can thus write %
$X_{j}=\tld{X}_{j}(1+\eps_j)$ where $\tld{X}_{j}\in L^+G(R)$, $\eps_j\in \Lie LG((\varpi^{a-1}))$, for all $j\in\cJ$.
Replacing $A_1^{(j)}$ by $(\tld{X}_{j})^{-1}A_1^{(j)}\Big(\phz(\tld{X}_{j-1})\Big)$, we can assume that \eqref{eq:fund:comparison:0} is true with $X_{j}=1+\eps_j$.
In particular $A_2^{(j)}=(1+\delta_j)A_1^{(j)}$ for some $\delta_j\in \Lie LG((\omega^{a-1}))$.
We now prove that actually $\eps_j\in \Lie L^+G((\omega^{a-1}))$.
Then \eqref{eq:fund:comparison:0} becomes $1=(1+\eps_j)(1+\delta_j)\Ad(A_1^{(j)})(1-\phz(\eps_{j-1}))$
so that 
\begin{equation}
\label{eq:fund:comparison:1}
(A_1^{(j)})^{-1}\eps_jA_1^{(j)}+(A_1^{(j)})^{-1}(A_2^{(j)}-A_1^{(j)})=\varphi(\eps_{j-1}).
\end{equation}
Let $\kappa_j$ be the pole order of $\eps_j$. 
Observe that
\begin{itemize}
\item 
$v^{\kappa_j+\langle\mu_j,\alpha^\vee\rangle}(v+p)\Big((A_1^{(j)})^{-1}\eps_jA_1^{(j)}\Big)\in \Mat_2(R[v]^{\wedge_{v(v+p)}})$
\item
$v^{\langle\mu_j,\alpha^\vee\rangle}(v+p)\Big((A_1^{(j)})^{-1}(A_2^{(j)}-A_1^{(j)})\Big)\in \Mat_2(R[v]^{\wedge_{v(v+p)}})$.
\end{itemize}
Since $(v+p)\eps_j=v\eps_j$ we deduce that 
\[
p\kappa_{j-1}\leq \max\{\kappa_j+\langle\mu_j,\alpha^\vee\rangle+1,\langle\mu_j,\alpha^\vee\rangle+1\}.
\]
This forces $\kappa_j\leq \max_{i}\{\frac{\langle\mu_{i},\alpha^\vee\rangle+1}{p-1}\}\leq \frac{p+3}{2(p-1)}<1$.
So that indeed $\eps_j\in \Lie L^+G((\omega^{a-1}))$.
\end{proof}

By taking fiber product, we thus obtain a commutative diagram:
\begin{equation}
\label{diag:cmpts}
\begin{tikzcd}[column sep=small]
\tld{Y}^{\eta,\tau}\ar{r}{\pi}\arrow[dr, phantom, "\square"]\ar[swap]{d}{\prod_{\cJ}\GL_{2}}&\tld{\cZ}^{\tau}\ar[hook]{r}\ar{d}{\prod_{\cJ}\GL_{2}}\arrow[dr, phantom, "\qquad\,\square"]&\bigg[L\G^{\textnormal{bd},(v+p)v^{\mu}}\Big/_{\phz}\prod_{j\in\cJ}L_1^+\G\bigg]^{\wedge_p}\ar{d}{\prod_{\cJ}\GL_{2}}\\
Y^{\eta,\tau}\ar{r}\ar{dr}{\eps_\tau}&\cZ^\tau\ar[hook]{r}\ar{d}&\bigg[L\G^{\textnormal{bd},(v+p)v^{\mu}}\Big/_{\phz}\prod_{j\in\cJ}L^+\G\bigg]^{\wedge_p}\ar[hook]{dl}\\
&\Phi\text{-}\Mod^{\textnormal{\'et}, 2}_K&
\end{tikzcd}
\end{equation}
where the hooked arrows are closed immersions, the arrows decorated with $\prod_{\cJ}\GL_{2}$ are $\prod_{\cJ}\GL_{2}$-torsors, and the central squares are cartesian, which defines the stacks $\tld{Y}^{\eta,\tau}$, $\tld{\cZ}^{\tau}$.

\subsection{Models for moduli of Breuil--Kisin modules}
\label{subsec:models:BK}

We now define $\Gr^{\tau}_{1}\into\Gr_{1}^{\textnormal{bd},(v+p)v^{\mu}}$ as the fpqc quotients $\prod_{j\in\cJ}L_1^+\rG\backslash L\rG^{\tau}\into \prod_{j\in\cJ}L^+_1\rG\backslash L\rG^{^{\textnormal{bd},(v+p)v^{\mu}}}$.
We define ${Y}^{\textnormal{mod},\eta,\tau}$ as the quotient
\[
\bigg[\Gr_1^\tau\big/\prod_{\cJ}B\text{-sh.cnj}\bigg]^{\wedge_p}
\]
by the shifted conjugation action $(b_j)\cdot(g_j)=(b_jg_jb^{-1}_{j-1})$.
Define also $\tld{Y}^{\textnormal{mod},\eta,\tau}$ as the fiber product
\[
\xymatrix{
\tld{Y}^{\textnormal{mod},\eta,\tau}\ar[r]\ar[d]\ar@{}[dr]|{\Box}&\bigg(\Gr_1^{\textnormal{bd},(v+p)v^{\mu}}\bigg)^{\wedge_p}\ar^{{\GL_2}^{\cJ}}[d]\\
{Y}^{\textnormal{mod},\eta,\tau}\ar[r]&
\bigg[\Gr_1^{\textnormal{bd},(v+p)v^{\mu}}/\prod_{\cJ}\GL_2\text{-sh.cnj}\bigg]^{\wedge_p}
}
\]
\begin{prop}
\label{prop:moduli:description}
The $p$-adic formal scheme $\tld{Y}^{\mathrm{mod},\eta,\tau}$ identifies with the (closed formal) subscheme of $\bigg(\prod_{\cJ}(\Gr_1\times (B\backslash\GL_{2,\cO}))\bigg)^{\wedge_p}$ consisting of pairs $(X,g)$ such that $(g_jX_{j}g_{j-1}^{-1})_{j\in\cJ}\in \Gr_1^{\tau}$.
In particular, $(X_{j})_{j\in\cJ}\in\Gr_1^{\textnormal{bd},(v+p)v^{\mu}}$.
\end{prop}
\begin{proof}
It follows from the definition that $\tld{Y}^{\mathrm{mod},\eta,\tau}$ is the quotient of the space of triples $((X_j),(Y_j),(g_j))\in \Gr_1^{\textnormal{bd},(v+p)v^{\langle\mu_j,\alpha^\vee\rangle}}\times\Gr_1^\tau\times\GL_2^{\cJ}$ satisfying $Y_j=g_jX_j(g_{j-1})^{-1}$ by the action of $B^{\cJ}$ given by 
\[
(b_j)\cdot ((X_j),(Y_j),(g_j))=((X_j),(b_jY_jb_{j-1}^{-1}),(b_jg_j)).
\]
This finishes the proof because the (class of the) triple $((X_j),(Y_j),(g_j))$ is uniquely determined by (the class of) $((X_j),(g_j))$.
\end{proof}

We \emph{define} $\tld{\cZ}^{\mathrm{mod},\tau}$ as the scheme theoretic image of projection map $\textnormal{pr}:\tld{Y}^{\mathrm{mod},\eta,\tau}\ra \Gr_1^{\textnormal{bd},(v+p)v^{\mu}}$ sending $(X,g)$ to $X$.
This gives  a factorization
\begin{equation}
\label{eq:diag:fact}
\xymatrix{
\tld{Y}^{\mathrm{mod},\eta,\tau}\ar^{\pi^{\textnormal{mod}}}[r]\ar^{\textnormal{pr}}[rd]&\tld{\cZ}^{\mathrm{mod},\tau}\ar^{\imath}@{^{(}->}[d]
\\
&\Gr_1^{\textnormal{bd},(v+p)v^{\mu}}
}
\end{equation}
For any $\tld{z}\in\un{\tld{W}}^\vee$ by Lemma \ref{lem:basic}\eqref{eq:basic:2.5} we have a formally \'etale monomorphism
\begin{equation}
\label{eq:Utld}
\Gr_{1}^{\textnormal{bd}, (v+p)v^\mu}(\tld{z})\defeq [(L^{-}\rG)^{\cJ} \tld{z}]\cap \Gr_{1}^{\textnormal{bd}, (v+p)v^\mu}\into \Gr_{1}^{\textnormal{bd}, (v+p)v^\mu}
\end{equation}
after $p$-adic completion, which is an open immersion because the target is of finite type.
Define $\tld{U}(\tld{z})$ to be the $p$-adic completion of the LHS of \eqref{eq:Utld}.
Note that $\tld{U}(\tld{z})=\prod_j\tld{U}(\tld{z}_j)$ has an obvious product structure.

\begin{lemma}
\label{lem:recouvrement}
$\{\tld{U}(\tld{z})\}_{\tld{z}\in\tld{\un{W}}}$ is a Zariski open covering for the $p$-adic completion of $\Gr_{1}^{\textnormal{bd}, (v+p)v^\mu}$.
\end{lemma}
\begin{proof}
Since $\Gr_{1}^{\textnormal{bd}, (v+p)v^\mu}$ is finite type the $\tld{U}(\tld{z})$ are actually Zariski open formal subschemes of its $p$-adic completion. 
Looking at $\ovl{\F}$-points shows that they cover.
\end{proof}

In particular, by diagram \eqref{eq:diag:fact}, $\tld{U}(\tld{z})$ induce open substacks $\tld{Y}^{\mathrm{mod},\eta,\tau}(\tld{z})\subseteq \tld{Y}^{\mathrm{mod},\eta,\tau}$,  $\tld{\cZ}^{\mathrm{mod},\tau}(\tld{z})\subseteq \tld{\cZ}^{\mathrm{mod},\tau}$.

\begin{lemma}
\label{lem:adm:cover}
The $\tld{Y}^{\tmod,\eta,\tau}(\tld{z})$ ($\tld{\cZ}^{\tmod,\tau}(\tld{z})$) for $\tld{z}\in\Adm(\eta)^{\cJ}s^{-1}v^\mu$ form a Zariski open cover of $\tld{Y}^{\tmod,\eta,\tau}$ (resp.~$\tld{\cZ}^{\tmod,\eta,\tau}$).
\end{lemma}
\begin{proof}
This follows from Lemma \ref{lem:basic}\eqref{eq:basic:2}.
\end{proof}

\begin{lemma}
\label{lem:LCI}
$\tld{Y}^{\tmod,\eta,\tau}$ is an $\cO$-flat local complete intersection of dimension $5f$ over $\cO$.
\end{lemma}
\begin{proof}
It suffices to show that $Y^{\tmod,\eta,\tau}$ is a $\cO$-flat local complete intersection of dimension $f$ over $\cO$.
Since $\Gr_1^\tau$ is a $B^{\cJ}$-torsor over $Y^{\tmod,\eta,\tau}$, the result follows from Table \ref{Table1FV}.
\end{proof}

\begin{lemma}
\label{lem:strightening}
Assume that $p-2-\langle \mu_j,\alpha^\vee\rangle\geq N$ for all $j\in\cJ$.
We have an isomorphism
\[
\xymatrix{
\tld{Y}^{\eta,\tau}\otimes_{\cO}\cO/p^{N}\ar^{\cong}[d]\ar[r]&
\left[L\rG^{\textnormal{bd}, (v+p)v^\mu}/_{\phz}\prod_{j\in\cJ}L_1^+\rG\right]\otimes_{\cO}\cO/p^{N}\ar^{\cong}[d]\\
\tld{Y}^{\textnormal{mod},\eta,\tau}\otimes_{\cO}\cO/p^{N}\ar[r]&
 \Gr_{1}^{\textnormal{bd}, (v+p)v^\mu}\otimes_{\cO}\cO/p^{N}
}
\]
\end{lemma}
\begin{proof}
The fact that the left vertical arrow is an isomorphism follows from the act that the right vertical arrow is an isomorphism.
The latter fact can be proven similar to the proof of \cite[Lemma 5.2.2]{MLM}: the result would follow from the fact that for any  $\cO/p^N$-algebra $R$ and $A\in L\rG^{\textnormal{bd}, (v+p)v^\mu}(R)$ the map 
\[
(X_j)\mapsto (X_j\Ad(A_j)\phz(X_{j-1}))
\]
is an automorphism of $(L_1^+\rG(R))^{\cJ}$.
In turn this follows from the fact that 
\[
v^pA_j^{-1}\in \frac{v^p}{(v+p)v^{\langle\mu_j,\alpha\rangle}}\mathrm{Mat}_2(R[\![v]\!])\subset v^{p-N-\langle\mu_j,\alpha\rangle}\mathrm{Mat}_2(R[\![v]\!])\subset v^2\mathrm{Mat}_2(R[\![v]\!]),
\] 
where we use that $\frac{1}{(v+p)}R[\![v]\!]\subset \frac{1}{v^N}R[\![v]\!]$.

\end{proof}
As a consequence of Lemma \ref{lem:strightening} and diagram \ref{diag:cmpts}, $\tld{U}(\tld{z})$ also induces open substacks $\tld{Y}^{\eta,\tau}(\tld{z})$, $\tld{\cZ}^{\tau}(\tld{z})$.
The following Theorem is the main result of the paper.
\begin{thm}
\label{thm:main:model}
Fix a small presentation $(s,\mu)$ of $\tau$.
Assume either $p> 8f+3+\max_j\langle \mu_j,\alpha^\vee\rangle$ or $p>7$ and $K=\Qp$. 
Then we have an isomoprhism
\[
\tld{\cZ}^{\mathrm{mod},\tau}(\tld{z})\cong \tld{\cZ}^{\tau}(\tld{z}).
\]
\end{thm}
\begin{rmk} \begin{enumerate}
\item
Since $(s,\mu)$ is small, $\max_j\langle \mu_j,\alpha^\vee\rangle\leq \frac{p+1}{2}$. In particular the hypothesis on $p$ is satisfied for all $\tau$ when $p>16f+7$.
\item Theorem \ref{thm:main:model} is proven under the first hypothesis in section \ref{subsect:main proof}. The proof under the improved bound $p>7$ for $\Qp$ is completed in section \ref{sec:examples f=1}, cf.~Remark \ref{rmk:bd:Qp} for the source of the improvements.
\end{enumerate}
\end{rmk}

The proof of Theorem \ref{thm:main:model} will be performed in two steps:
\begin{enumerate} 
\item 
\label{it:step:1}
We first show there is a isomorphism $\tld{\cZ}^{\tau}(\tld{z})\otimes_{\cO}\cO/p^N\stackrel{\sim}{\ra}\tld{\cZ}^{\mathrm{mod},\tau}(\tld{z})\otimes_{\cO}\cO/p^N$ for sufficiently large $N$.
\item 
\label{it:step:2}
We bound the power of $p$ that belongs to the ideal of singularity of $\tld{\cZ}^{\tmod,\eta,\tau}$ over $\cO$. This allows us to lift the isomorphism mod $p^N$ in the above step to an isomorphism over $\cO$.
\end{enumerate}
In the remainder of this section we will carry out the first step of the above strategy modulo some geometric facts about $\tld{Y}^{\textnormal{mod},\eta,\tau}\ra\tld{\cZ}^{\textnormal{mod},\tau}$ which will established in the later sections. The second step of the strategy is carried out in Section \ref{subsect:main proof}

\subsubsection{Modeling mod $p^N$}
Set $N\defeq p-2-\max_j{\langle \mu_j,\alpha^\vee\rangle}$.
Define $\tld{\cZ}_N^{\mathrm{apx}}$ as the scheme theoretic image of the map
\[
\tld{Y}^{\eta,\tau}\otimes_{\cO}\cO/p^N\ra
\bigg[L\G^{\textnormal{bd},(v+p)v^{\mu}}\Big/_{\phz}\prod_{j\in\cJ}L_1^+\G\bigg]\otimes_{\cO}\cO/p^N=\Gr_{1}^{\textnormal{bd}, (v+p)v^\mu}\otimes_{\cO}\cO/p^{N}.
\]
We have the commutative diagram
\begin{equation}
\label{diag:cmpts:1}
\begin{tikzcd}[column sep=small]
\tld{Y}^{\eta,\tau}\otimes_{\cO}\cO/p^N\ar{rr}{\sim \textnormal{Lemma}~\ref{lem:strightening} }\ar{d}{\pi}\ar{dr}&&\tld{Y}^{\textnormal{mod},\eta,\tau}\otimes_{\cO}\cO/p^N\ar{d}{\pi^{\textnormal{mod}}}
\\
\tld{\cZ}^{\tau}\otimes_{\cO}\cO/p^N\ar[hook]{dr}{\eqref{diag:cmpts}}&\tld{\cZ}^{\mathrm{apx}}_N\ar[hook]{d}\ar[hook]{r}{\heartsuit\heartsuit}\ar[hook]{l}{\heartsuit}&\tld{\cZ}^{\textnormal{mod},\tau}\otimes_{\cO}\cO/p^N\ar[hook]{d}{\imath}
\\
&\left[L\rG^{\textnormal{bd}, (v+p)v^\mu}/_{\phz}\prod_{\cJ}L_1^+\rG\right]\otimes_{\cO}\cO/p^N\ar{r}{\sim}&\Gr_1^{\textnormal{bd}, (v+p)v^\mu}\otimes_{\cO}\cO/p^N%
\end{tikzcd}
\end{equation}

\begin{prop}
\label{prop:heart:inverse}
The morphisms $\heartsuit$, $\heartsuit\heartsuit$ have a factorization
\[
\xymatrix{
\tld{\cZ}^{\tau}\otimes_{\cO}\cO/p^N&\tld{\cZ}^{\mathrm{apx}}_N\ar@{^{(}->}_-{\heartsuit}[l]\ar@{^{(}->}_-{\heartsuit\heartsuit}[r]&\tld{\cZ}^{\mathrm{mod},\tau}\otimes_{\cO}\cO/p^N
\\
\tld{\cZ}^{\tau}\otimes_{\cO}\cO/p^{N-1}\ar@{^{(}.>}_-{{\sim}}[r]\ar@{^{(}->}[u]
&
\tld{\cZ}^{\mathrm{apx}}_N\otimes_{\cO}\cO/p^{N-1}\ar@{^{(}->}[u]&
\tld{\cZ}^{\mathrm{mod},\tau}\otimes_{\cO}\cO/p^{N-1}
\ar@{^{(}.>}_-{{\sim}}[l]\ar@{^{(}->}[u]}
\]
In particular, the natural morphisms $\heartsuit$, $\heartsuit\heartsuit$ induce an isomorphism
\[
\tld{\cZ}^{\tau}\otimes_{\cO}\cO/p^{N-1}\stackrel{\sim}{\longrightarrow}\tld{\cZ}^{\textnormal{mod},\tau}\otimes_{\cO}\cO/p^{N-1}.
\]
\end{prop}
The proof of Proposition \ref{prop:heart:inverse} crucially relies on the following result, whose proof will be postponed until section \ref{subsect:main proof}.
\begin{prop}
\label{prop:coker}
We have:
\begin{equation}
\label{eq:coker}
p
\coker
\left(
\cO_{
\tld{\cZ}^{\mathrm{mod},\tau}}
\rightarrow \pi^{\mathrm{mod}}_*\left(\cO_{\tld{Y}^{\mathrm{mod},\eta,\tau}}\right)\right)=0.
\end{equation}
\end{prop}

\begin{prop}
\label{prop:R1:null}
For $\ell>0$
\[
R^\ell \pi^{\mathrm{mod}}_*\cO_{\tld{Y}^{\mathrm{mod},\eta,\tau}}=0.
\]
\end{prop}

\begin{proof}[Proof of Proposition \ref{prop:heart:inverse}]
The statement is local on $\Gr_1^{\textnormal{bd}, (v+p)v^\mu}\otimes_{\cO}\cO/p^N$, so it suffices to prove it after intersecting everything with $\tld{U}(\tld{z})$.
Write $\tld{\cZ}^{\tau}(\tld{z})=\Spf\, R$, $\tld{\cZ}^{\textnormal{mod},\tau}(\tld{z})=\Spf\, R^{\textnormal{mod}}$, and $\tld{\cZ}^{\textnormal{mod}}_N(\tld{z})=\Spf\,R^{\textnormal{apx}}_N$.
Also let $S=\pi_*\left(\cO_{\tld{Y}^{\eta,\tau}(\tld{z})}\right)$ and $S^{\textnormal{mod}}=\pi^{\mathrm{mod}}_*\left(\cO_{\tld{Y}^{\mathrm{mod},\eta,\tau}(\tld{z})}\right)$ and $S_N=\pi^{\mathrm{mod}}_*\left(\cO_{\tld{Y}^{\mathrm{mod},\eta,\tau}(\tld{z})}/p^N\right)$.
Note that $S$ and $S^{\tmod}$ are $p$-torsion free and, $S/p^N\into S_N=S^{\tmod}/p^N$ by Proposition \ref{prop:R1:null}.

Since $\tld{\cZ}^{\tmod,\tau}(\tld{z})$ is the scheme theoretic image of $\tld{Y}^{\tmod,\eta,\tau}\ra\tld{U}(\tld{z})$ we have $R^{\tmod}\subset S^{\tmod}$.
Similarly $R\subset S$. 
Finally equation \eqref{eq:coker} shows that $S^{\tmod}/R^{\tmod}$ is $p$-torsion.
We thus have the following commutative diagram
\begin{equation}
\label{diag:cmpts:2}
\begin{tikzcd}[column sep=small]
S\ar[twoheadrightarrow]{r}&S/p^N\ar[hook]{r}&S_N&S^{\textnormal{mod}}/p^N\ar[equal]{l}&S^{\textnormal{mod}}\ar[twoheadrightarrow]{l}
\\
R\ar[hook]{u}{\pi}\ar[twoheadrightarrow]{r}&R/p^N\ar[twoheadrightarrow]{r}{\heartsuit\heartsuit}\ar{u}&R_N^{\mathrm{apx}}\ar[hook]{u}&R^{\textnormal{mod}}/p^N\ar{u}\ar[twoheadrightarrow]{l}{\heartsuit}&R^{\textnormal{mod}}\ar[twoheadrightarrow]{l}\ar[hook]{u}{\pi^{\textnormal{mod}}}
\end{tikzcd}
\end{equation}
where the hooked arrow are injective.
Let $C$, $C^{\textnormal{mod}}$ denote the cokernel of the maps $\pi$, $\pi^{\textnormal{mod}}$ respectively.
As $S$, $S^{\textnormal{mod}}$ are both $p$-flat, we deduce from \eqref{diag:cmpts:2} the following commutative diagram with exact rows (this defines $C_N$):
\begin{equation}
\label{diag:cmpts:3}
\begin{tikzcd}[column sep=small]
&&\ker(\heartsuit\heartsuit)
\ar[equal]{dl}\ar{d}&&&\\
0\ar{r}&C^{\textnormal{mod}}[p^N]\ar{r}\ar{d}&
R^{\textnormal{mod}}/p^N\ar{r}{\pi^{\textnormal{mod}}}\ar[twoheadrightarrow]{d}{\heartsuit\heartsuit}&S^{\textnormal{mod}}/p^N\ar{r}\ar[equal]{d}&C^{\textnormal{mod}}/p^N\ar{r}\ar[equal]{d}&0\\
0\ar{r}&0\ar{r}&
R_N^{\mathrm{apx}}\ar[hook]{r}&S_N\ar{r}&C_N\ar{r}&0
\\
0\ar{r}&C[p^N]\ar{u}\ar{r}&
R/p^N\ar{r}{\pi}\ar[twoheadrightarrow]{u}{\heartsuit}&S/p^N\ar{r}\ar[hook]{u}&C/p^N\ar{r}\ar[hook]{u}\ar{r}&0
\\
&&\ker(\heartsuit)
\ar[equal]{ul}\ar{u}&&&
\end{tikzcd}
\end{equation}

Proposition \ref{prop:coker} imply that $pC^{\textnormal{mod}}=0$ 
and hence $C^{\textnormal{mod}}[p^N]=C^{\textnormal{mod}}[p]$.
Thus $\ker(\heartsuit\heartsuit)$ is annihilated by $p$.
We conclude that  $p\ker(\heartsuit\heartsuit)\subseteq p^NR^{\textnormal{mod}}$ and, as $R^{\textnormal{mod}}$ is $p$-flat, that $\ker(\heartsuit\heartsuit)\subseteq p^{N-1}R^{\textnormal{mod}}$.
This implies the factorization
\[
\xymatrix{
R^{\textnormal{mod}}/p^N\ar^{\heartsuit\heartsuit}@{->>}[r]\ar@{->>}[d]&R_N^{\mathrm{apx}}\ar^{\exists}@{.>>}[dl]\\
R^{\textnormal{mod}}/p^{N-1}&
}
\]
and hence an isomorphism $R^{\textnormal{mod}}/p^{N-1}\cong R_N^{\mathrm{apx}}/p^{N-1}$.

Now  $pC_N=0$ so $p(C/p^N)=0$. 
Hence $pC\subset p^NC\subset p^{2}C$ and as $C$ is $p$-adically separated we learn that $pC=0$.
We repeat the argument in the previous paragraph to obtain the factorization
\[
\xymatrix{
R/p^N\ar^{\heartsuit}@{->>}[r]\ar@{->>}[d]&R_N^{\mathrm{apx}}\ar^{\exists}@{.>>}[dl]\\
R/p^{N-1}&
}
\]
and hence an isomorphism $R/p^{N-1}\cong R_N^{\mathrm{apx}}/p^{N-1}$.
\end{proof}

\section{Geometry of local models}
\label{sec:geometry}
\subsection{Equations for $\tld{Y}^{\tmod,\eta,\tau}$}

Let $\tau$ be a tame inertial type with small presentation $(s,\mu)$.
Everything we do depends on this choice but we usually suppress this dependence from the notation.
Let $\tld{w}\in\Adm^\vee(\eta)^{\cJ}$ and set $\tld{z}=\tld{w}s^{-1}v^{\mu}$.
We consider the morphism $\GL_{2,\cO}\onto B\backslash \GL_{2,\cO}\stackrel{\sim}{\longrightarrow}\mathbb{P}^1_{\cO}$ sending $\text{\tiny{$\begin{pmatrix}\alpha&\beta\\\gamma&\delta\end{pmatrix}$}}$ to $[\gamma:\delta]\in \mathbb{P}^1_{\cO}$.
Given a Noetherian $\cO$-algebra $R$ where $p$ is nilpotent, by Proposition \ref{prop:moduli:description} and the definition of $\curlAdm(\eta)$,  we see that $\tld{Y}^{\text{mod},\eta,\tau}(\tld{z})(R)$ the groupoid of tuples $(l_j,\kappa_j, X_j)_{j\in \cJ}\in (\bP^1\times \GL_2\times L_1^{-}\rG)^{\cJ}$
subject to the following conditions that for some (equivalently, any) lift $\tld{l}_{j}\in\GL_2(R)$ of $l_j$:
\begin{enumerate}
\item 
\label{eq:cond:Y:1}
$X_j\tld{w}_j\Ad\big(s^{-1}_jv^{\mu_j}\big)(\tld{l}_{j-1}^{-1})\in  \Mat_2(R[\![v]\!])$
\item
\label{eq:cond:Y:2}
$\det\big(X_j\tld{w}_j\big)\in R^\times(v+p)$
\item
\label{eq:cond:Y:3}
we have $l_{j}\kappa_j\cdot \Big(X_j\tld{w}_j\Ad\big(s^{-1}_jv^{\mu_j}\big)(\tld{l}_{j-1}^{-1})\Big)|_{v=0}=[0:1]$.
\end{enumerate}
Note that $\Ad\big(s_j^{-1}v^{\mu_j}\big)(B(R))\in B(R[v])$ since $\mu$ is dominant, which justifies the independence of the choice of the lift $\tld{l}_{j-1}\in \GL_2(R)$ in items \eqref{eq:cond:Y:1},\eqref{eq:cond:Y:2},\eqref{eq:cond:Y:3} above.

In item \eqref{eq:cond:Y:3} we have used the following 
\begin{conv}
Given $\text{\tiny{$\begin{pmatrix}\alpha&\beta\\\gamma&\delta\end{pmatrix}$}}\in \Mat_2(R)$ and $x,y,z,t\in R$, the equality 
\[
\raisebox{-.6em}{$[x:y]$}\begin{pmatrix}\alpha&\beta\\\gamma&\delta\end{pmatrix}\raisebox{-.6em}{$=[z:t]$}
\] 
is interpreted as $t(\alpha x+\gamma y)-z(\beta x+\delta y)=0$.

Equivalently, this can be written in matrix form
\[
\raisebox{-.5em}{$\begin{pmatrix}x& y \end{pmatrix}$}\begin{pmatrix}\alpha&\beta\\\gamma&\delta\end{pmatrix}\begin{pmatrix}t\\ -z \end{pmatrix}\raisebox{-.5em}{$=0.$}
\]
\end{conv}
We will always adopt the above convention in the tables below, note that this allows us to interpret equations as $[x:y]=[z:t]$ even if neither side are actual elements of $\bP^1$.
We will also consider the following auxiliary space $\tld{Ba}(\tld{z})$ of tuples $(l_{j},\kappa_j, X_j,r_j)_{j\in \cJ}\in (\bP^1\times \GL_2\times L_1^{-}\rG\times \bP^1)^{\cJ}$ satisfiyng the following variant of the conditions above:
\begin{enumerate}
\item 
\label{eq:cond:Y:1'}
$X_j\tld{w}_j\Ad\big(s^{-1}_jv^{\mu_j}\big)(\tld{r}_{j}^{-1})\in  \Mat_2(R[\![v]\!])$
\item
\label{eq:cond:Y:2'}
$\det\big(X_j\tld{w}_j\big)\in R^\times(v+p)$
\item
\label{eq:cond:Y:3'}
we have $l_{j}\kappa_j\cdot \Big(X_j\tld{w}_j\Ad\big(s^{-1}_jv^{\mu_j}\big)(\tld{r}_{j}^{-1})\Big)|_{v=0}=[0:1]$.
\end{enumerate}
Clearly, we have a decomposition $\tld{Ba}(\tld{z})=\prod_{j\in\cJ}\tld{Ba}_j(\tld{z}_j)$ where the $j$-th factor classify the quadruples $(l_j,\kappa_j,X_j,r_j)$.
Furthermore $\tld{Y}^{\tmod,\eta,\tau}$ is exactly the subspace of $\tld{Ba}(\tld{z})$ where we impose that $r_j=l_{j-1}$ for all $j$.

\begin{lemma}
\label{lem:parametrization_1}
The $f$-tuples $(X_j, r_j)_{j\in \cJ}\in (L_1^{-}\rG\times \bP^1)^{\cJ}(R)$ satisfying condition \eqref{eq:cond:Y:1} are precisely those given in Table \ref{Table:Matrices:avec:bornes}, according to $\tld{w}_j\in\Adm^\vee(\eta_j)$, $s_j\in W$ and $k_j\defeq \langle \mu_j,\alpha^\vee\rangle$.
\end{lemma}

\begin{table}[H]
\captionsetup{justification=centering}
\caption[Foo content]{\textbf{Description of $X_j\in L_1^{-}\rG$}
}
\label{Table:Matrices:avec:bornes}
\centering
\adjustbox{max width=\textwidth}{
\begin{tabular}{| c || c | c | }
\hline
\hline
$(\tld{w}_j,s_j)$&$(t_\eta,(12)),\ (w_0t_\eta,(12)),\ (t_{w_0(\eta)},\Id)$&$(t_\eta,\Id),\ (w_0t_\eta,\Id),\ (t_{w_0(\eta)},(12))$\\
\hline
\hline
&&\\
&$\begin{aligned}
s_jw_j^{-1}X_jw_js_j^{-1}&=\begin{pmatrix}1+\frac{A}{v^{k_j}}&\frac{B}{v}\\\frac{C}{v^{k_j}}+\frac{C'}{v^{k_j-1}}&1+\frac{D}{v}\end{pmatrix}
\\
A,B,C,C',D&\in R,
\\
\end{aligned}$&$\begin{aligned}
s_jw_j^{-1}X_jw_js_j^{-1}&=\begin{pmatrix}1+\frac{A}{v}&0\\\frac{C}{v}+\frac{C'}{v^{k_j+1}}&1\end{pmatrix}\\
A,C,C'&\in R
\end{aligned}$
\\
&&\\
\hline
$k_j=0$&$A=C=C'=0$&$C'=0$\\
\hline
$k_j=1$&$[A:B]=[C:D]=r_{j},C'=0$&$[C':1]=r_{j}$\\
\hline
$k_j>1$&$[A:B]=[C:D]=[C':1]=r_{j}$&$[C':1]=r_{j}$\\
\hline
\hline
\end{tabular}}
\captionsetup{justification=raggedright,
singlelinecheck=false
}
\end{table}
\begin{proof}
Let $\text{\tiny{$\begin{pmatrix}a&b\\c&d\end{pmatrix}$}}\in\GL_2(R)$ so that $r_{j}=[c:d]$.
Also write $s_jw_j^{-1} X_j
w_js_j^{-1}=\begin{pmatrix}
1+\alpha&\beta\\
\gamma&1+\delta
\end{pmatrix}$ so that $\alpha,\beta,\gamma,\delta\in \frac{1}{v}R[\frac{1}{v}]$.
We abbreviate $k=k_j$.

Letting $w_jt_{\nu_j}\defeq \tld{w}_j$  condition \eqref{eq:cond:Y:1} is equivalent to
\begin{equation}
\label{eq:check:eqns}
\begin{pmatrix}
(1+\alpha)v^{k+\eps}&\beta v^{1-\eps_j}\\
\gamma v^{k+\eps}&(1+\delta)v^{1-\eps}
\end{pmatrix}
\begin{pmatrix}d&-b\\-c&a\end{pmatrix}
\in \begin{pmatrix}v^{k}R[\![v]\!]&R[\![v]\!]\\v^{k}R[\![v]\!]&R[\![v]\!]\end{pmatrix}
\end{equation}
where $\eps=0$ (resp.~$\eps=1$) if $(\tld{w}_j,s_j)$ is as in the first (resp.~second) column of Table \ref{Table:Matrices:avec:bornes}.
From this we learn that $v^{1-\eps}\beta$, $v^{1-\eps}(1+\delta)$, $(1+\alpha)v^{k+\eps}$, $\gamma v^{k+\eps}\in R[\![v]\!]$.
The first two conditions show that the second column of $\begin{pmatrix}
1+\alpha&\beta\\
\gamma&1+\delta
\end{pmatrix}$ has the form specified in the table.

We now show that $\alpha,\gamma$ also have the form specified in the table.

\begin{enumerate}
\item[\emph{Case $\eps=0$.}]
Looking at the second column of \eqref{eq:check:eqns} we get 
\begin{align}
v^k(1+\alpha)d-v \beta c&\in v^kR[\![v]\!]\label{eq:check:1}\\
v^k\gamma d-v (1+\delta) c&\in v^kR[\![v]\!]\label{eq:check:2}
\end{align}
Note that these two equations imply that $\alpha d$ and $\gamma d$ have the form specified in the table.

Since we already know $(1+\alpha)v^{k}$, $\gamma v^{k}\in R[\![v]\!]$ we only have to consider the case $k>1$. 
But then equation \eqref{eq:check:2} shows that $c\in R d$ hence $d\in R^\times$ because $(c,d)=R$.
But this implies $\alpha$ and $\gamma$ are of the desired form.

\item[\emph{Case $\eps=1$.}]
Looking at the second column of \eqref{eq:check:eqns} we get (using that $\beta=\delta=0$)
\begin{align}
v^{k+1}(1+\alpha)d&\in v^kR[\![v]\!]\label{eq:check:3}\\
v^{k+1}\gamma d-c&\in v^kR[\![v]\!].\label{eq:check:4}
\end{align}
Note that these two equations imply that $\alpha d$ and $\gamma d$ have the form specified in the table.

Since we already know $(1+\alpha)v^{k+1}$, $\gamma v^{k+1}\in R[\![v]\!]$ we only have to consider the case $k>0$. 
But then equation \eqref{eq:check:4} shows that $c=0$ hence $d\in R^\times$.
Again this implies $\alpha$ and $\gamma$ are of the desired form.
\end{enumerate}

We have now shown that $\alpha,\beta,\gamma,\delta$ have the desired form. 
Plugging this information into equations \eqref{eq:check:1}, \eqref{eq:check:2},\eqref{eq:check:3},\eqref{eq:check:4} immediately yields the remaining equations in the table.
\end{proof}

\begin{lemma}
For each $j$ the spaces $\tld{Ba}_j(\tld{z}_j)$ have explicit presentations given by Table \ref{Table:Traductions}, in terms of the corresponding spaces in Table \ref{Table:Matrices:avec:bornes}.
\end{lemma}

\begin{table}[H]
\captionsetup{justification=centering}
\caption[Foo content]{\textbf{Presentations of $\tld{Ba}_j(\tld{z}_j)$.}
}
\label{Table:Traductions}
\centering
\adjustbox{max width=\textwidth}{
\begin{tabular}{| c | c || c | c | c | }
\hline
\hline
\backslashbox{$\langle\mu_j,\alpha^\vee\rangle$}{$\tld{w}_j$}&&$t_\eta$&$w_0t_\eta$&$t_{w_0(\eta)}$\\
\hline
&$s_j$&&&\\
\hline
\hline
&&&&\\
\multirow{ 2}{*}{$>1$}&$(12)$&$
\begin{aligned}
\textrm{Variables: } A,B,C,&\,C',D\\
&\\
A=BC',&\, C=pC',\, D=p\\ 
&\\
r_j&=[C':1]\\
l_j\kappa_j&=[B:-p]
\end{aligned}$&$\begin{aligned}
\textrm{Variables: } A,B,C,&\,C',D\\
&\\
A=BC',&\ \, C=pC',\ D=p\\ 
&\\
r_j&=[C':1]\\
l_j\kappa_j&=[-p:B]
\end{aligned}$&$\begin{aligned}
\textrm{Variables: } A,C,&\,C'\\
&\\
A&=p\\
&\\
r_j&=[C':1]\\
l_j\kappa_j&=[0:1]
\end{aligned}$\\
&&&&\\
\cline{2-5}
&&&&\\
&$\Id$&$\begin{aligned}
\textrm{Variables: } A,C,&\,C'\\
&\\
A&=p\\
&\\
r_j&=[C':1]\\
l_j\kappa_j&=[C:-p]
\end{aligned}$&$\begin{aligned}
\textrm{Variables: } A,C,&\,C'\\
&\\
A&=p\\
&\\
r_j&=[C':1]\\
l_j\kappa_j&=[-p:C]
\end{aligned}$&$\begin{aligned}\textrm{Variables: } A,B,C,&\,C',D\\
&\\
A=BC',&\ \, C=pC',\ D=p\\ 
&\\
r_j&=[C':1]\\
l_j\kappa_j&=[0:1]
\end{aligned}$\\
&&&&\\
\hline
\hline
&&&&\\
\multirow{ 2}{*}{$=1$}&$(12)$&{\color{blue}$\begin{aligned}
\textrm{Variables: } A,B,C,&\,C',D\\
&\\
A&=p-D\\
&\\
r_j&=[C:D]=[p-D:B]\\
l_j\kappa_j&=[D-p:C]=[-B:D]
\end{aligned}$}&{\color{blue}$\begin{aligned}
\textrm{Variables: } A,B,C,&\,C',D\\
&\\
A&=p-D\\
&\\
r_j&=[C:D]=[p-D:B]\\
l_j\kappa_j&=[C:D-p]=[-D:B]
\end{aligned}$}&$\begin{aligned}
\textrm{Variables: } A,C,&\,C'\\
&\\
A&=p\\
&\\
r_j&=[C':1]\\
l_j\kappa_j&=[0:1]
\end{aligned}$\\
&&&&\\
\cline{2-5}
&&&&\\
&$\Id$&$\begin{aligned}
\textrm{Variables: } A,C,&\,C'\\
&\\
A&=p\\
&\\
r_j&=[C':1]\\
l_j\kappa_j&=[C:-p]
\end{aligned}$&$\begin{aligned}
\textrm{Variables: } A,C,&\,C'\\
&\\
A&=p\\
&\\
r_j&=[C':1]\\
l_j\kappa_j&=[-p:C]
\end{aligned}$&
{\color{blue}$\begin{aligned}
\textrm{Variables: } A,B,C,&\,C',D\\
&\\
A&=p-D\\
&\\
r_j&=[C:D]=[p-D:B]\\
l_j\kappa_j&=r_j
\end{aligned}$}\\
&&&&\\
\hline
\hline
&&&&\\
$=0$&$\Id$&{\color{orange}$\begin{aligned}
\textrm{Variables: } A,C,&\,C'\\
&\\
A&=p,\, C'=0\\
&\\
l_{j}\kappa_j&=r_j\begin{pmatrix}1&0\\-C&p\end{pmatrix}
\end{aligned}$}&{\color{orange}$\begin{aligned}
\textrm{Variables: } A,C,&\,C'\\
&\\
A&=p,\, C'=0\\
&\\
l_{j}\kappa_j&=r_j\begin{pmatrix}0&-1\\-p&C\end{pmatrix}
\end{aligned}$}&{\color{orange}$\begin{aligned}
\textrm{Variables: } A,B,C,&\,C',D\\
&\\
A=0,\ C&=0,\ C'=0,\ D=p\\
&\\
l_{j}\kappa_j&=r_j\begin{pmatrix}p&-B\\0&1\end{pmatrix}
\end{aligned}$}\\
&&&&\\
\hline
\hline
\end{tabular}}
\captionsetup{justification=raggedright,
singlelinecheck=false
}
\caption*{
The meaning of the variables $A,B,C,C',D$ is in terms 
 of the $X_j\in L^{-}_1\rG$ extracted from the corresponding entries in Table \ref{Table:Matrices:avec:bornes}.}\end{table}
\begin{proof}
The equations involving $A,B,C,C',D$ are exactly obtained from Table \ref{Table:Matrices:avec:bornes} by imposing the condition $\det(X_j)\in R^\times \frac{(v+p)}{v}$.
The formula for $r_j$ also follows from Table \ref{Table:Matrices:avec:bornes}.

Thus the only thing we need to verify is the equation involving $l_j$.
This is immediate in all cases except when $(\tld{w}_j,\langle\mu_j,\alpha^\vee\rangle,s_j)
=(t_{w_0(\eta)},1,\Id)$.
In this case let $\begin{pmatrix}a'&b'\\c'&d'\end{pmatrix}\in \GL_2(R)$ (resp.~$\begin{pmatrix}a&b\\c&d\end{pmatrix}\in \GL_2(R)$) be a lift of $l_j\kappa_j$ (resp.~$r_j$).
Thus condition \eqref{eq:cond:Y:3'} (together with condition \eqref{eq:cond:Y:1'}) is the condition that the $(2,1)$-entry of $\begin{pmatrix}a'&b'\\c'&d'\end{pmatrix}\begin{pmatrix}d&aB-Ab\\-c&aD-Cb\end{pmatrix}$
is zero, which exactly means that $l_j\kappa_j=r_j$.
\end{proof}

\begin{rmk}
So far we only considered $\tld{Ba}_j(\tld{z}_j)$, $\tld{U}(\tld{z}_j)$ as $p$-adic formal schemes. 
However Table \ref{Table:Traductions} gives obvious decompletions and we will work with the scheme version of these spaces in what follows.
\end{rmk}

\subsection{Cohomological properties of $\tld{Y}^{\tmod,\eta,\tau}\ra \Gr_1^{\textnormal{bd},(v+p)v^{\mu}}$}
Recall that diagram \eqref{eq:diag:fact} gives a factorization
\begin{equation}
\xymatrix{
\tld{Y}^{\mathrm{mod},\eta,\tau}(\tld{z})\ar^-{\pi^{\textnormal{mod}}}[r]&\tld{\cZ}^{\mathrm{mod},\tau}(\tld{z})\ar^-{\imath}@{^{(}->}[r]&\tld{U}(\tld{z})
}
\end{equation}
In this section we will use the explicit presentations from the previous section to study $R\pi^{\tmod}_*\cO_{\tld{Y}^{\mathrm{mod},\eta,\tau}}$ and hence prove Propositions \ref{prop:coker}, \ref{prop:R1:null}.

Since $\imath$ is a closed immersion, we have $\iota_*\circ R\pi^{\mathrm{mod}}_*\cO_{\tld{Y}^{\mathrm{mod},\eta,\tau}}=R\textnormal{pr}_*\cO_{\tld{Y}^{\mathrm{mod},\eta,\tau}}$.
In particular Proposition \ref{prop:coker} is equivalent to
\begin{equation*}
p
\coker
\left(
\cO\left(\Gr_1^{\textnormal{bd}, (v+p)v^\mu}\right)
\rightarrow (\textnormal{pr})_*\left(\cO\left(\tld{Y}^{\mathrm{mod},\eta,\tau}\right)\right)\right)=0
\end{equation*}
and Proposition \ref{prop:R1:null} is equivalent to $R\textnormal{pr}_*\cO_{\tld{Y}^{\mathrm{mod},\eta,\tau}}$ concentrating in degree $0$.
In other words we can replace $\pi^{\tmod}$ by $\textnormal{pr}$ in the statements of interest.
Furthermore, these statements are local on the target so it suffices to analyze the situation after intersecting with $\tld{U}(\tld{z})$ for $\tld{z}\in\Adm^\vee(\eta)s^{-1}v^\mu$.

To analyze $\textnormal{pr}$ we factorize 
\[
\xymatrix{
\tld{Y}^{\tmod,\eta,\tau}(\tld{z})\ar@{^{(}->}[r] &\tld{Ba}(\tld{z})=\prod_j\tld{Ba}_j(\tld{z}_j)\ar[r]^{\prod \pr_j}&\tld{U}(\tld{z})=\prod_j\tld{U}(\tld{z}_j)
}
\]
where $\pr_j: \tld{Ba}_j(\tld{z}_j)\ra\tld{U}(\tld{z}_j)$ is the map 
\[
(l_j,\kappa_j,X_j,r_j)\mapsto \kappa_j X_j.
\]
Observe that we have an isomorphism  $\tld{Ba}_j(\tld{z}_j)\cong \GL_2\times Ba_j(\tld{z}_j)$ where $Ba_j(\tld{z}_j)=\tld{Ba}_j(\tld{z}_j)\times_{\GL_2}\{1\}$, given by $(l_j,\kappa_j,X_j,r_j)\mapsto (\kappa_j,(l_j\kappa_j,X_j,r_j))$.
Let $p_j$ (resp.~$q_j$) be the obvious projections from $Ba_j(\tld{z}_j)\subset \bP^1\times \tld{U}(\tld{z}_j)\times \bP^1$ to the left (resp.~right) $\bP^1$ factor.
We continue to denote $\pr_j$ the projection from $Ba_j(\tld{z}_j)$ to the middle factor (this is compatible with the projection from $\tld{Ba}_j(\tld{z}_j)$).

The following is immediate from Table \ref{Table:Traductions}
\begin{lemma}\label{lem:five cases} Up to isomorphism, there are the following possibilities for $Ba_j(\tld{z}_j)$
 \begin{enumerate}
\item 
\label{basic:space:1}
$Ba_j(\tld{z}_j)=\bA^2$ with $p_j,q_j$ are constant maps from $\bA^2$ to $\bP^1$ and $\pr_j$ is the identity. 
This covers the cases $(\tld{w}_j,s_j,k_j)=(t_{w_0(\eta)},\Id,>1), (t_{w_0(\eta)},(12),\geq1)$.
\item
\label{basic:space:2}
$Ba_j(\tld{z}_j)=Bl_{(p,0)}\bA^1\times\bA^1$ where $Bl_{(p,0)}\bA^1=\{([x:y],C)\ |\ Cx=py\}\subset \bP^1\times \bA^1$ is the blowup of $\bA^1_{/\cO}$ at the origin in its special fiber.
Then $p_j$ is the natural projection map from $Bl_{(p,0)}\bA^1\ra\bP^1$ and $q_j$ is the natural inclusion $\bA^1\into \bP^1$ given by $C'\mapsto [C':1]$, and $\pr_j$ is the natural map to $\bA^2$ extracting $C,C'$.
This covers $(\tld{w}_j,s_j,k_j)=(t_{\eta},(12),>1), (t_{\eta},\Id,\geq 1), (w_0t_{\eta},(12),>1), (w_0t_{\eta},\Id,\geq 1)$ .
\item
\label{basic:space:3}
$Ba_j(\tld{z}_j)\subset\bP^1\times \bA^1\times\bP^1$ consists of $([x:y],C,[x':y'])$ such that $pxy'-yx'=0$.
Then $p_j,q_j$ are the projections to the left, resp.~right $\bP^1$ factor, and $\pr_j$ is the natural map to $\bA^1$ extracting $C$.
This covers the cases when $\langle\mu_j,\alpha^\vee\rangle=0$.
\item
\label{basic:space:4}
$Ba_j(\tld{z}_j)\subset \bP^1\times M\times\bP^1$ consists of $([x:y],\begin{pmatrix}A&B\\C&D\end{pmatrix},[x':y'])$ such that \begin{itemize}
\item $\begin{pmatrix}A&B\\C&D\end{pmatrix}$ has determinant $0$ and trace $p$;
\item
$\begin{pmatrix}x'&y'\end{pmatrix}\begin{pmatrix}D&-B\\-C&A\end{pmatrix}=0$
\item
$\begin{pmatrix}x&y\end{pmatrix}\begin{pmatrix}D&-C\\-B&A\end{pmatrix}=0$
\end{itemize}
Then $p_j,q_j$ are the projections to the left, resp.~right $\bP^1$ factor, and $\pr_j$ is the map extracting $A,B,C,D$.
This covers $(\tld{w}_j,s_j,k_j)=(t_{\eta},(12),1), (w_0t_{\eta},(12),1)$.
\item
\label{basic:space:5}
It is the same as the previous case, except that instead of the third item we impose $[x:y]=[x':y']$.
This correspond to the case $(\tld{w}_j,s_j,k_j)=(t_{w_0(\eta)},\Id,1)$.
\end{enumerate}
\end{lemma}

\begin{lemma}
\label{lem:prop:pr}
\begin{enumerate}
\item 
\label{eq:prop:pr:1}
$\pr_j$ is proper and when $\langle\mu_j,\alpha^\vee\rangle>0$, $\pr_j$ becomes a closed immersion after inverting $p$; 
\item 
\label{eq:prop:pr:2}
$Ba_j(\tld{z}_j)$ is a $\cO$-flat local complete intersection of relative dimension $2$ over $\cO$. The relative dualizing sheaf of $Ba_j(\tld{z}_j)/\cO$ is $q_j^*\cO_{\bP^1}(-1)\otimes p_j^*\cO_{\bP^1}(-1)$.
\item
$\cO_{\tld{U}(\tld{z}_j)}\ra\pr_{j*}\cO_{Ba_j(\tld{z}_j)}$ is surjective.
\end{enumerate}
The same assertions hold for $\tld{Ba}_j(\tld{z_j})$ and $\tld{Ba}(\tld{z})$, but with relative dimensions $6$ and $6f$.
\end{lemma}
\begin{proof}
The first assertion is clear from Table \ref{Table:Traductions} once we  observe that $(l_j,r_j)$ can be uniquely solved for when $p$ is invertible and $k_j>0$.

For the second assertion, we inspect the five cases in Lemma \ref{lem:five cases}.
The result is obvious for case \eqref{basic:space:1} and follows from Lemma \ref{lem:cohomology of blowup} below for case \eqref{basic:space:2}. 

For case \eqref{basic:space:3}, the result follows from the fact that $Ba_j(\tld{z}_j)\into \bP^1\times \bA^1\times \bP^1$ is a regular immersion cut out by an equation of bidegree $(1,1)$, which has normal bundle $q_j^*\cO_{\bP^1}(1)\otimes p_j^*\cO_{\bP^1}(1)$, and thus relative dualizing sheaf $q_j^*\cO_{\bP^1}(1)\otimes p_j^*\cO_{\bP^1}(1)\otimes q_j^*\cO_{\bP^1}(-2)\otimes p_j^*\cO_{\bP^1}(-2)=q_j^*\cO_{\bP^1}(-1)\otimes p_j^*\cO_{\bP^1}(-1)$.

We turn to case \eqref{basic:space:4}. Set $t=x/y$, $s=x'/y'$. The $Ba_j(\tld{z}_j)$ has an open cover given by
\begin{itemize}
\item $\Spec \cO[A,B,C,D,s,t]/(A-stD,B-tD,C-sD,p-D-stD)$.
\item $\Spec \cO[A,B,C,D,s^{-1},t]/(A-tC,B-ts^{-1}C,D-s^{-1}C,p-s^{-1}C-tC)$.
\item $\Spec \cO[AB,C,D,s,t^{-1}]/(A-sB,C-st^{-1}B,D-tB,p-t^{-1}B-sB)$.
\item $\Spec \cO[A,B,C,D,s,t^{-1}]/(B-s^{-1}A,C-t^{-1}A,D-s^{-1}t^{-1}A,p-A-s^{-1}t^{-1}A)$.
\end{itemize}
Thus we obtain locally a regular immersion from $Ba_j(\tld{z}_j)$ to $\bA^6$. 
A simple computation shows that the determinant of the normal bundle of the regular immersion $Ba_j(\tld{z}_j)\into \bP^1 \times \bA^4\times \bP^1$ is $q_j^*\cO_{\bP^1}(1)\otimes p_j^*\cO_{\bP^1}(1)$, hence again the relative dualizing sheaf is $q_j^*\cO_{\bP^1}(-1)\otimes p_j^*\cO_{\bP^1}(-1)$.

We finally deal with case \eqref{basic:space:5}. Then $Ba_j(\tld{z}_j)\into \bP^1\times \bA^4$ is a regular immersion cut out by the equations
\[A+D=p, AD-BC=0, \begin{pmatrix} x & y\end{pmatrix}\begin{pmatrix} D & -B\\ -C & A\end{pmatrix}=0\]
A similar computation as in the previous case yields the desired result.

The third assertion follows from the fact that $\pr_j$ factors through the subscheme $Z\subset \tld{U}(\tld{z}_j)$ such that 
\begin{itemize}
\item
$Z$ is normal;
\item
$\pr_{j\,*}\cO_{Ba_j(\tld{z}_j)}=\cO_{Z}$ after inverting $p$.
\end{itemize}
The existence of such a subscheme $Z$ follows immediately from inspecting the cases in Lemma \ref{lem:five cases}: in fact $Z=\bA^2$ with coordinate $C,C'$ or $Z=\Spec \cO[B,C,D]/((p-D)D-BC)$ (if $\langle \mu_j,\alpha^{\vee} \rangle>0$) or $Z=\bA^1$ with coordinate $C$ if $\langle \mu_j,\alpha^\vee\rangle=0$.
\end{proof}

We will make use of the following elementary computation.
\begin{lemma}\label{lem:cohomology of blowup}
Let $Bl_{(0,p)}\bA^1=\{([x:y],t)\ |\  xp=ty\}\subset \bP^1\times \bA^1$ be the blowup of $\bA^1_{/\cO}$ at the ideal $(0,p)$.
Let $\cO(-k)$ be the pull back of $\cO_{\bP^1}(-k)$ by the projection to $\bP^1$.
Then 
\begin{enumerate}
\item For $k\geq 0$
\[
H^n(Bl_{(0,p)}\bA^1,\cO(-k))
=\begin{cases}
p^k\cO[t]&\text{if $n=0$};\\
\textnormal{annihilated by $p^{k-1}$}&\text{if $n=1$};\\
0&\text{if $n>1$}.
\end{cases}
\]
\item The relative dualizing sheaf of $Bl_{(0,p)}\bA^1/\cO$ is $\cO(-1)$.
\end{enumerate}
\end{lemma}
\begin{proof}
\begin{enumerate}
\item
$R\Gamma(Bl_{(0,p)}\bA^1,\cO(-k))$ is computed by the \v{C}ech complex:
\[
\cO[t,\frac{t}{p}]\oplus \left(\frac{p}{t}\right)^k\cO[t,\frac{p}{t}]\ra\cO[t,\left(\frac{p}{t}\right)^{\pm1}]
\]
where all the terms are viewed as $\cO[t]$-submodules of $E[t^{\pm1}]$ and the differential is given by $(f,g)\mapsto f-g$.
The result now follows from an explicit computation.
\item Since $Bl_{(0,p)}\bA^1\into \bP^1\times \bA^1$ is a regular immersion cut out by a degree $1$ equation in the $\bP^1$ coordinates, the normal bundle is $\cO(1)$, hence the dualizing complex is $\cO(-2)\otimes \cO(1)=\cO(-1)$.
\end{enumerate}
\end{proof}

The following Lemma will be the key to our analysis:
\begin{lemma}
\label{lem:coh:comp}
Let $j\in\cJ$ and let $\eps_j,\delta_j\in\{0,1\}$.
\begin{enumerate}
\item
 $\textnormal{pr}_{j*}\bigg(q_j^*(\cO_{\mathbf{P}^1}(-1))^{\eps_{j}}\otimes_{\cO_{Ba(\tld{z}_j)}}p_j^*(\cO_{\mathbf{P}^1}(-1))^{\delta_{j}}\bigg)$ is $p$-torsion free.
\item
The complex $R\textnormal{pr}_{j*}\bigg(q_j^*(\cO_{\mathbf{P}^1}(-1))^{\eps_{j}}\otimes_{\cO_{Ba(\tld{z}_j)}}p_j^*(\cO_{\mathbf{P}^1}(-1))^{\delta_{j}}\bigg)$ is concentrated in degree $0$ if $(\eps_j,\delta_j)\neq (1,1)$, and is  concentrated in degrees $0$ and $1$ if $(\eps_j,\delta_j)= (1,1)$.
\item
\label{lem:coh:comp:3}
If $(\eps_j,\delta_j)=(1,1)$ then $R^1\textnormal{pr}_{j*}\bigg(q_j^*(\cO_{\mathbf{P}^1}(-1))^{\eps_{j}}\otimes_{\cO_{{Ba}(\tld{z}_j)}}p_j^*(\cO_{\mathbf{P}^1}(-1))^{\delta_{j}}\bigg)$ is
\[
\begin{cases}
\text{$\cO$-torsion free}&\text{if $\langle \mu_j,\alpha^\vee\rangle=0$;}\\
\text{isomorphic to $\F$}&\text{if $\langle \mu_j,\alpha^\vee\rangle=1$ and $(s_j,\tld{w}_j)=(\Id,t_{w_0(\eta)})$};\\
0&\text{otherwise.}
\end{cases}
\]
\end{enumerate}
\end{lemma}
\begin{proof}
The first item is obvious because $Ba_j(\tld{z}_j)$ is flat over $\cO$. We now explain the cohomological computations.

For the remainder of the proof we set $\eps\defeq\eps_j$, $\delta\defeq\delta_j$ and $\cF\defeq q_j^*(\cO_{\mathbf{P}^1}(-1))^{\eps}\otimes_{\cO_{{Ba}(\tld{z})}}p_j^*(\cO_{\mathbf{P}^1}(-1))^{\delta}$. We work with the five cases in Lemma \ref{lem:five cases}.

The computation is trivial for case \eqref{basic:space:1}, and follows from Lemma \ref{lem:cohomology of blowup} for case \eqref{basic:space:2}.

We now turn to case \eqref{basic:space:4}.
We have an open cover  $U_1=\{y\neq 0\}$, $U_2=\{x\neq 0\}$.
Then setting $t=x/y$ we see that $U_1$ is the space of $(t,C,[x':y'])\subset \bA^2\times \bP^1$ such that $x'(p-Ct)-Cy'=0$. 
Hence $U_1$ is isomorphic to $Bl_{(p,0)}\bA^1\times\bA^1$ and $\cF|_{U_1}$ is the pull-back of $\cO(-1)^{\eps}$ from the obvious map to $\bP^1$.
Thus
\begin{itemize}
\item
$R\Gamma(U_1,\cF)=\cO[C,t]p^{\eps}$;
\item
Similarly $R\Gamma(U_2,\cF)=\cO[B,t^{-1}]p^{\eps}$
\item
$R\Gamma(U_1\cap U_2,\cF)=\cO[C,t,t^{-1}]p^{\eps}$
\end{itemize}
are all concentrated in degree zero.
Thus the \v{C}ech cohomology spectral sequence computing $R\Gamma (Ba_j(\tld{z}_j),\cF)$ degenerates at $E_1$ and $R\Gamma (Ba_j(\tld{z}_j),\cF)$ is computed by 
\begin{align}
\cO[t,C]\oplus t^{-\delta}\cO[t^{-1},B]&\stackrel{d}{\ra} \cO[t^{\pm 1},C]
\end{align}
where the differential are the obvious inclusion induced by the relation $B=t(p-tC)$.
In turn, this complex is quasi isomorphic to the complex 
\begin{align}
t^{-\delta}\cO[t^{-1},B]&\stackrel{d}{\ra} \cO[t^{\pm 1},C]/\cO[t,C]=\oplus_{k\geq 1}t^{-k}\cO[C].
\end{align}
Since $\delta\in\{0,1\}$, $t^{-k}\in \mathrm{Im}(d)$ for $k\geq 1$.
Now for $\ell >0$ and $k\geq 0$
\begin{equation}
t^{-2k-\ell}(t(p-tC))^k=
(-1)^k\frac{C^k}{t^\ell}+\dots
\in \textnormal{Im}(d)
\label{eq:key:coh:comp}
\end{equation}
where $\dots$ is a $\cO$-linear combination of $\frac{C^n}{t^m}$ where $0\leq n< k$. 
Hence $\frac{C^k}{t^\ell}\in \textnormal{Im}(d)$ by induction on $k$.

We now deal with case \eqref{basic:space:5}.
Similar to the previous case, we have an open cover $U_2=\{y\neq 0\}$, $U_1=\{x\neq 0\}$ and set $t=y/x$.

We have, after choosing a trivialization of $\cF$ on $U_1$
\begin{align*}
&\cF(U_1)=\frac{\cO[t,B,C,D]}{(D-tC,p-D-Bt)}\cong \cO[t,C]\\
&
\cF(U_2)=t^{-(\eps+\delta)}\frac{\cO[t^{-1},B,C,D]}{((p-D)-Bt^{-1},C-Dt^{-1})}\cong t^{-(\eps+\delta)}\cO[t^{-1},B]
\end{align*}
Thus $R\Gamma(Ba_j(\tld{z}_j),\cF)$ is computed by the \v{C}ech complex
\begin{align}
\cO[t,C]\oplus t^{-\eps-\delta}\cO[t^{-1},B]\stackrel{d}{\ra} \cO[t^{\pm},C]
\end{align}

where the differential are the obvious inclusion induced by the relation $B=t(p-tC)$.
As before this complex is quasi isomorphic to 
\[
t^{-\eps-\delta}\cO[t^{-1},B]\stackrel{d}{\ra} \cO[t^{\pm},C]/\cO[t,C]=\oplus_{k\geq 1}t^{-k}\cO[C].
\]
Compared to case \eqref{basic:space:4} the only new computation we have to make is when $(\eps,\delta)=(1,1)$.
In this case, we have $t^{-k}\in \textnormal{Im}(d)$ for $k\geq 2$ but only $pt^{-1}\in \textnormal{Im}(d)$ (and $t^{-1} \notin \textnormal{Im}(d)$).
Equation \eqref{eq:key:coh:comp} for  $k,\ell\geq 1$ then shows that $\frac{C^k}{t^\ell}\in \textnormal{Im}(d)$ for all $k,\ell\geq 1$.
This means that the $H^1$ of the \v{C}ech complex is isomorphic to $\F$.

Finally we deal with case \eqref{basic:space:3}.
Consider the open cover $U_1=\{x\neq 0\}$, $U_2=\{y\neq 0\}$.
Write $t=y/x$, $s=x'/y'$.
Then $U_2\cong \bA^2=\Spec(\cO[s,C,t^{-1}]/(s-pt^{-1}))$, while $U_1\cong Bl_{(0,p)}\bA^1\times \bA^1$, where the coordinate on the blownup $\bA^1$ is $t$.
Now $\cF|_{U_1}\cong \cO(-1)^{\eps}$ so 
$R\Gamma(U_1,\cF)=p^{\eps}\cO[C,t]$ is concentrated in degree $0$.
On the other hand, $U_2$ is affine so $R\Gamma(U_2,\cF)=s^{\eps}t^{-\delta}\cO[s,C,t^{-1}]/(s-pt^{-1})$.
Moreover $R\Gamma(U_1\cap U_2,\cF)=p^{\eps}\cO[s,C,t^{\pm}]/(s-pt^{-1})$.
Thus the \v{C}ech complex computing $R\Gamma(Ba_j(\tld{z}_j),\cF)$ is given by 
\[
p^{\eps}\cO[C,t]\oplus s^{\eps}t^{-\delta}\cO[s,C,t^{-1}]/(s-pt^{-1})\ra p^{\eps}\cO[s,C,t^{\pm1}]/(s-pt^{-1}).
\]
This is quasi-isomorphic to the complex
\[
t^{-\eps-\delta}\cO[C,t^{-1}]\ra \cO[C,t^{\pm 1}]/\cO[C,t]
\]
with differential induced by the natural inclusion.
This has no $H^1$ if $(\eps,\delta)\neq (1,1)$, and has $H^1\cong \cO[C]$ if $(\eps,\delta)= (1,1)$.
We are thus done with the cohomological computations. This finishes the proof.
\end{proof}

\begin{cor}
\label{cor:Delta}
Consider the commutative diagram:
\[
\xymatrix@=3pc{
\tld{Y}^{\textnormal{mod},\eta,\tau}(\tld{z})\ar^-{\Delta}@{^{(}->}[r]\ar_{\textnormal{pr}}[dr]&\tld{Ba}(\tld{z})\ar^{\textnormal{pr}_{\tld{B}}}[d]\\
&\tld{U}(\tld{z})}
\]
Then $R^i{\pr_{\tld{B}\raisebox{.2em}{$*$}}}\cO_{\tld{Ba}(\tld{z})}=0$ if $i>0$ and $\cO_{\tld{U}(\tld{z})}\onto {\pr_{\tld{B}\raisebox{.2em}{$*$}}}\cO_{\tld{Ba}(\tld{z})}$.

In particular, by letting $\cI(\tld{z})$ be the ideal sheaf defining the closed immersion $\Delta$
\begin{itemize}
\item 
$\coker
\left(
\cO_{
\tld{U}(\tld{z})}
\rightarrow \textnormal{pr}_*\cO_{\tld{Y}^{\mathrm{mod},(\eta,\tau)}(\tld{z})}\right)=R^1{\pr_{\tld{B}\raisebox{.2em}{$*$}}}\cI(\tld{z})$
\item
$R^i\pr_*\cO_{\tld{Y}^{\mathrm{mod},(\eta,\tau)}(\tld{z})}=R^{i+1}{\pr_{\tld{B}\,\raisebox{.2em}{$*$}}}\cI(\tld{z})$ for $i>1$.
\end{itemize}
\end{cor}
\begin{proof}
The first part follows from the fact that $R{\pr_{\tld{B}\raisebox{.2em}{$*$}}}\cO_{\tld{Ba}(\tld{z})}=\boxtimes_j R\pr_{j*}\tld{Ba}_j(\tld{z}_j)$ and hence is concentrated in degree $0$ by Lemma \ref{lem:coh:comp} (with $\eps_j=\delta_j=0$).

For the second part, we have the exact triangle
\[
R{\pr_{\tld{B}\,\raisebox{.2em}{$*$}}}\cI(\tld{z})\ra R{\pr_{\tld{B}\,\raisebox{.2em}{$*$}}}\cO_{\tld{Ba}(\tld{z})}\ra R\pr_*\cO_{\tld{Y}^{\mathrm{mod},(\eta,\tau)}(\tld{z})}\ra
\]
which immediately gives the second item.
For the first item, observe that the exact triangle implies
$\coker
\left(
{\pr_{\tld{B}\,\raisebox{.2em}{$*$}}}\cO_{\tld{Ba}(\tld{z})}
\rightarrow \textnormal{pr}_*\cO_{\tld{Y}^{\mathrm{mod},(\eta,\tau)}(\tld{z})}\right)=R^1{\pr_{\tld{B}\,\raisebox{.2em}{$*$}}}\cI(\tld{z})$
but also 
\[
\coker
\left(
{\pr_{\tld{B}\,\raisebox{.2em}{$*$}}}\cO_{\tld{Ba}(\tld{z})}
\rightarrow \textnormal{pr}_*\cO_{\tld{Y}^{\mathrm{mod},(\eta,\tau)}(\tld{z})}\right)=\coker
\left(
\cO_{
\tld{U}(\tld{z})}
\rightarrow \textnormal{pr}_*\cO_{\tld{Y}^{\mathrm{mod},(\eta,\tau)}(\tld{z})}\right)
\]
because $\cO_{\tld{U}(\tld{z})}\onto {\pr_{\tld{B}\,\raisebox{.2em}{$*$}}}\cO_{\tld{Ba}(\tld{z})}$.
\end{proof}

\begin{lemma}
\label{lem:koszul:res}
Under the composite
\[
\tld{Y}^{\tmod,\eta,\tau}(\tld{z})
\stackrel{\Delta}{\into}\tld{Ba}(\tld{z})=\prod_j\tld{Ba}_j(\tld{z}_j)=\prod_j\GL_2\times Ba_j(\tld{z}_j)
\]
$\tld{Y}^{\tmod,\eta,\tau}(\tld{z})$ is a complete intersection defined by the zero locus of maps $\fs_j:\cL_j\defeq q_{j}^*(\cO_{\bP^1}(-1))\otimes p_{j-1}^*(\cO_{\bP^1}(-1))\ra \cO_{\prod_j\GL_2\times Ba_j(\tld{z}_j)}$.
In particular, the ideal sheaf $\cI(\tld{z})$ defining $\Delta$ identifies with 
\[
\tau_{<0}\bigg(\textnormal{Kos}_\bullet\Big(\bigoplus_{j\in\cJ}\cL_j,(\fs_j)\Big)\bigg)
\]
\end{lemma}
\begin{proof} It follows from the definitions that we have a diagram
\[
\xymatrix@=3pc{
\tld{Y}^{\textnormal{mod},\eta,\tau}(\tld{z})\ar[d]\ar^-{\Delta}@{^{(}->}[r]\ar@{}[dr]|{\Box}&
\prod_j\tld{Ba}_j(\tld{z}_j) \ar[r]\ar[d]&\prod_j\GL_2\times Ba_j(\tld{z}_j) \ar[d]_{\mathrm{id}\times p_j\times q_j}\\
\prod_\cJ\GL_2\times \bP^1 \ar^-{\Delta}@{^{(}->}[r]  &\prod_\cJ\GL_2\times \bP^1\times  \bP^1 \ar[r]^{\cong}  & \prod_{\cJ}\GL_2\times \bP^1\times \bP^1
}
\]
where the bottom left map is $(\kappa_j,l_j)_j\mapsto (\kappa_j,l_j,l_{j-1})_j$ and the bottom right map is $(\kappa_j,l_j,r_j)\mapsto (\kappa_j,l_j\kappa_j,r_j)$.

The bottom right isomorphism only commutes with projection to the right $\bP^1$ factor, but not to the left $\bP^1$ factor. Nevertheless, the $\GL_2$-equivariance of $\cO_{\bP^1}(-1)$ shows that its pullback via projection to the $j$-th left (respectively, $j$-th right) $\bP^1$ are compatible with the bottom isomorphism. 

It follows that $\tld{Y}^{\textnormal{mod},\eta,\tau}(\tld{z})$ is the zero common locus of $f$ maps $\fs_j:\cL_j\ra \cO_{\tld{Ba}(\tld{z})}$.
By Lemma \ref{lem:LCI} and Lemma \ref{lem:prop:pr}, $\tld{Y}^{\textnormal{mod},\eta,\tau}(\tld{z})\into \tld{Ba}(\tld{z})$ has codimension $f$ and $\tld{Ba}(\tld{z})$ is local complete intersection. 
This implies that $\tld{Y}^{\textnormal{mod},\eta,\tau}(\tld{z})$ is the global complete intersection in $\tld{Ba}(\tld{z})$ cut out by the $\fs_j$.
\end{proof}

We compute $R\Gamma(\cI(\tld{z}))$ using the resolution from Lemma \ref{lem:koszul:res} and the K\"unneth formula.
Let $\ell\in\{1,\dots,f\}$ and consider $f$-tuples $\un{\eps}=(\eps_j)_{j\in\cJ},\un{\delta}=(\delta_j)_{j\in\cJ}\in\{0,1\}^{\cJ}$ satisfying $\eps_j=\delta_{j+1}$ and $\#\{j\in\cJ, \eps_{j}=1\}=\ell$.
Then by Lemma \ref{lem:koszul:res} we have
\[
\bigwedge^\ell\bigg(\bigoplus_{j\in\cJ}\cL_j\bigg)=\bigoplus_{\un{\eps},\un{\delta}}\bigotimes_{j\in\cJ,\cO_{\tld{Ba}}}\bigg(q_j^*(\cO_{\mathbf{P}^1}(-1))^{\eps_{j}}\otimes p_j^*(\cO_{\mathbf{P}^1}(-1))^{\delta_{j}}\bigg)
\]
where the direct sum runs over $f$-tuples $\un{\eps},\un{\delta}$ as above.
As the $\tld{Ba}(\tld{z}_j)$ are Tor-independent over $\cO$, we deduce from the K\"unneth formula \cite[\href{https://stacks.math.columbia.edu/tag/0FLQ}{Tag 0FLQ}]{stacks-project} that 
{\small
\begin{align*}
&R\Gamma\bigg(\bigotimes_{j\in\cJ,\cO_{\tld{Ba}}}\bigg(q_j^*(\cO_{\mathbf{P}^1}(-1))^{\eps_{j}}\otimes p_j^*(\cO_{\mathbf{P}^1}(-1))^{\delta_{j}}\bigg)\bigg) \\
&\cong \bigotimes^{\bL}_{j\in\cJ,\cO}R\Gamma\bigg(q_j^*(\cO_{\mathbf{P}^1}(-1))^{\eps_{j}}\otimes_{\cO_{\tld{Ba}(\tld{z}_j)}}p_j^*(\cO_{\mathbf{P}^1}(-1))^{\delta_{j}}\bigg)
\end{align*}
}%
From Lemma \ref{lem:coh:comp} we  obtain the following corollaries:
\begin{cor}
\label{cor:vanish}
\begin{itemize}
\item
For $0<\ell<f$, $R^{\ell}\textnormal{pr}_*\Bigg(\bigwedge^\ell\bigg(\bigoplus_{j\in\cJ}\cL_j\bigg)\Bigg)=0$.
\item
 $pR^{f}\textnormal{pr}_*\Bigg(\bigotimes_{j\in\cJ}\cL_j)\Bigg)=0$; and
\item $R^{f}\textnormal{pr}_*\Bigg(\bigotimes_{j\in\cJ}\cL_j)\Bigg)\neq 0$ if and only if for each $j\in \cJ$, either $\langle \mu_j,\alpha^\vee\rangle=0$ or $\langle \mu_j,\alpha^\vee\rangle=1$ and $(s_j,\tld{w}_j)=(\mathrm{id},t_{w_0(\eta)})$. 
\end{itemize}
\end{cor}
\begin{proof} Since the image of $\mathrm{pr}$ is affine, it suffices to check the statements after taking global sections, so we can replace all occurences of $R^j\mathrm{pr}_*$ with $R^j\Gamma$.

Given our above discussion, it suffices to analyze
\[ \mathrm{H}^\ell\bigg(\bigotimes^{\bL}_{j\in\cJ,\cO}R\Gamma\bigg(q_j^*(\cO_{\mathbf{P}^1}(-1))^{\eps_{j}}\otimes_{\cO_{\tld{Ba}(\tld{z}_j)}}p_j^*(\cO_{\mathbf{P}^1}(-1))^{\delta_{j}}\bigg)\bigg)
\]
for $f$-tuples $\un{\eps}=(\eps_j)_{j\in\cJ},\un{\delta}=(\delta_j)_{j\in\cJ}\in\{0,1\}^{\cJ}$ satisfying $\eps_j=\delta_{j+1}$ and $\#\{j\in\cJ, \eps_{j}=1\}=\ell$.
Note that these conditions imply that there are at most $\ell$ indices $j$ such that $(\eps_j,\delta_j)=(1,1)$, with strict inequality if $\ell<f$.
Thus the amplitude bound of Lemma \ref{lem:coh:comp} shows the above cohomology group vanishes for $\ell<f$, while
\begin{align*} 
&\mathrm{H}^f\bigg(\bigotimes^{\bL}_{j\in\cJ,\cO}R\Gamma\bigg(q_j^*(\cO_{\mathbf{P}^1}(-1))^{\eps_{j}}\otimes_{\cO_{\tld{Ba}(\tld{z}_j)}}p_j^*(\cO_{\mathbf{P}^1}(-1))^{\delta_{j}}\bigg)\bigg)\cong\\
&\qquad \cong \bigotimes_{j\in\cJ,\cO} R^1\Gamma\bigg(q_j^*(\cO_{\mathbf{P}^1}(-1))\otimes_{\cO_{\tld{Ba}(\tld{z}_j)}}p_j^*(\cO_{\mathbf{P}^1}(-1))\bigg)
\end{align*}
The result now follows from Lemma \ref{lem:coh:comp}
\end{proof}
We note that the proof of the above corollary also shows
\begin{cor}
\label{cor:vanish:0}
For any $\ell\in\{0,\dots,f\}$ and $k>\ell$
\[
R^{k}\textnormal{pr}_*\Bigg(\bigwedge^\ell\textnormal{pr}_{\tld{Ba}}^*\bigg(\bigoplus_{j\in\cJ}\cL_j\bigg)\Bigg)=0.
\]
\end{cor}

\begin{proof}[Proof of Propositions \ref{prop:coker}, \ref{prop:R1:null}]
By Lemma \ref{lem:koszul:res} $R\pr_*\cI(\tld{z})$ is filtered (in the derived sense) by $R\pr_*\Bigg(\bigwedge^\ell\textnormal{pr}_{\tld{Ba}}^*\bigg(\bigoplus_{j\in\cJ}\cL_j\bigg)\Bigg)[\ell-1]$ for $1\leq \ell\leq f$.
Then by Corollary \ref{cor:Delta}, Proposition \ref{prop:coker} follows from Corollary \ref{cor:vanish} and Proposition \ref{prop:R1:null} from  Corollary \ref{cor:vanish:0}.
\end{proof}
We also record the following, which will be used in subsection \ref{subsect:rational smoothness}
\begin{prop}\label{prop:dualizing complex model} The relative dualizing sheaf of $\tld{Y}^{\tmod,\eta,\tau}(\tld{z})/\cO$ is trivial.
\end{prop}
\begin{proof} By Lemma \ref{lem:koszul:res}, $\tld{Y}^{\tmod,\eta,\tau}(\tld{z})\into \tld{Ba}(\tld{z})$ is a regular immersion with normal bundle
\[\bigoplus q_j^*\cO_{\bP^1}(1)\otimes p_{j-1}^*\cO_{\bP^1}(1) \]
which thus has determinant $\bigotimes_{\cJ} q_j^*\cO_{\bP^1}(1)\otimes p_{j-1}^*\cO_{\bP^1}(1)$.
The result now follows from Lemma \ref{lem:prop:pr} \eqref{eq:prop:pr:2}
\end{proof}

\subsection{The naive models}\label{subsect:naive model}
Recall that $\tld{\cZ}^{\textnormal{mod},\tau}(\tld{z})$ is the scheme theoretic image of $\tld{Y}^{\textnormal{mod},\eta,\tau}(\tld{z})$ under the map $\pi^{\mathrm{mod}}$ which forgets the elements $r_j=l_{j-1}\in \bP^1$. In other words the equations for $\tld{\cZ}^{\textnormal{mod},\tau}(\tld{z})$ are obtained by eliminating the $r_j=l_{j-1}$ from the defining equations of $\tld{Y}^{\textnormal{mod},\eta,\tau}(\tld{z})$ which are extracted from Table \ref{Table:Traductions}. The goal of this section is to construct a slight enlargement $\tld{\cZ}^{\nv,\tau}(\tld{z})$ of $\tld{\cZ}^{\textnormal{mod},\tau}(\tld{z})$, which has the advantage of being given by explicit equations.

We first introduce some auxilliary notation. We view $\cJ=\bZ/f\bZ$ as an oriented graph with edges going from $j$ to $j-1$. 
\begin{defn} 
\label{def:aux:not}
Given the data $(s,\mu,\tld{w})$, and $\tld{z}\defeq\tld{w}s^{-1}v^{\mu}$
\begin{enumerate}
\item Define $M_j(\tld{z}_j)$ to be the scheme theoretic image of $Ba_j(\tld{z}_j)\to \tld{U}(\tld{z}_j)$.
\item Let $j\in \cJ$. We say 
\begin{itemize}
\item $j$ is of type $II$ if either $k_j>1$ or
$(\tld{w}_j,s_j,k_j)=(t_\eta,\Id,1),(w_0t_\eta,\Id,1),(t_{w_0(\eta)},(12),1)$; 
\item $j$ is of type $I$ if $(\tld{w}_j,s_j,k_j)=(t_\eta,(12),1),(w_0t_\eta,(12),1),(t_{w_0(\eta)},\Id,1)$; 
\item 
\label{def:fragmentation}
$j$ is of type $0$ if $k_j=0$.
\end{itemize}
\item A \emph{fragmentation} of $\cJ$ is the decomposition $\cJ=\bigcup\cJ_k$ into subsets $\cJ_k$ such that:
\begin{itemize}
\item $\cJ_k$ is an oriented path in $\cJ$, i.e an ordered subset of the form $[j,j+\ell]=(j,j-1,\cdots, j-\ell)$ for some $\ell\leq f$.
\item The endpoints of the path $\cJ_k$ are not of type $0$.
\item The interior points of the the path $\cJ_k$ are of type $0$.
\item $\cJ_k$ is not a singleton unless $f=1$.
\end{itemize}
Each $\cJ_k$ is called a fragment of $\cJ$.
\end{enumerate}
\end{defn}
\begin{rmk}\label{rmk:basic images}
\begin{enumerate}
\item The scheme theoretic image of $\tld{Ba}_j(\tld{z}_j)\to \tld{U}(\tld{z}_j)$ is $\GL_2\times M_j(\tld{z}_j)$.
\item Under our running assumption that $\tau$ is regular, $\cJ$ must have a vertex not of type $0$. This implies $\cJ$ has a unique fragmentation, obtained by the minimal paths joining the vertices not of type $0$. We also note that each fragment $\cJ_k$ has a well-defined starting point and ending point (which may coincide, in which case the fragment is all of $\cJ$).
\item It follows from Table \ref{Table:Traductions} that
\begin{itemize}
\item If $j$ is type $II$: $M_j(\tld{z}_j)\cong \bA^2$.
\item If $j$ is type $I$: $M_j(\tld{z}_j)\cong M$, the space of $\begin{pmatrix}A&B\\C&D\end{pmatrix}$ which has determinant $0$ and trace $p$.
\item If $j$ is type $0$: $M_j(\tld{z}_j)\cong \bA^1$.
\end{itemize}
\end{enumerate}
\end{rmk}
We now define a subspace of $\prod_{\cJ} \GL_2\times M_j(\tld{z}_j)$ using the fragmentation $\cJ=\bigcup \cJ_k$.
\begin{defn}\label{defn:naive equations}
Let $\cJ=\bigcup \cJ_k$ be the fragmentation of $\cJ$. We define $\tld{\cZ}^{\nv,\tau}(\tld{z})$ to be the closed subscheme of $\prod_{\cJ} \GL_2\times M_j(\tld{z}_j)\into  \GL_2^{\cJ}\times \tld{U}(\tld{z})$ cut out by the matrix equations
\[M_{\textnormal{out},o} \bigg(\prod_{\ell=o+1}^{i-1}T_\ell \bigg)M_{\textnormal{in},i}=0\]
for each fragment $\cJ_k=\{i,\cdots, o\}$, where
\begin{itemize}
\item $M_{\textnormal{in},i}$ is the initial matrix for $j=i$ in Table \ref{Table:matrices}
\item $M_{\textnormal{out},o}$ is the final matrix for $j=o$ in Table \ref{Table:matrices}
\item $T_{\ell}$ is the transition matrix for $j=\ell$ (which are of type $0$) in Table \ref{Table:matrices}.
\end{itemize}
We also define $\cZ^{\nv,\tau}(\tld{z})$ to be the fiber of $\tld{\cZ}^{\nv,\tau}(\tld{z})$ above $1\in \GL_2^{\cJ}$.
\end{defn}
\begin{table}[htb] 
\captionsetup{justification=centering}
\caption[Foo content]{\textbf{}
}
\label{Table:matrices}
\centering
\adjustbox{max width=\textwidth}{
\begin{tabular}{| c | c || c | c | c | }
\hline
\hline
\backslashbox{$\langle\mu_j,\alpha^\vee\rangle$}{$\tld{w}_j$}&&$t_\eta$&$w_0t_\eta$&$t_{w_0(\eta)}$\\
\hline
&$s_j$&&&\\
\hline
\hline
&&&&\\
\multirow{ 2}{*}{$>1$}&$(12)$&
$\begin{aligned}
&\textnormal{Type $II$}\\
&\textnormal{Initial Matrix:\ } &&\begin{pmatrix}1\\-C'\end{pmatrix}\\
&\textnormal{Final Matrix:\ }
&&\begin{pmatrix}B&-p\end{pmatrix}\kappa_j^{-1}
\end{aligned}$
&$
\begin{aligned}
&\textnormal{Type $II$}\\
&\textnormal{Initial Matrix:\ } &&\begin{pmatrix}1\\-C'\end{pmatrix}\\
&\textnormal{Final Matrix:\ }
&&\begin{pmatrix}-p&B\end{pmatrix}\kappa_j^{-1}
\end{aligned}
$&$\begin{aligned}
&\textnormal{Type $II$}\\
&\textnormal{Initial Matrix:\ } &&\begin{pmatrix}1\\-C'\end{pmatrix}\\
&\textnormal{Final Matrix:\ }
&&\begin{pmatrix}0&1\end{pmatrix}\kappa_j^{-1}
\end{aligned}$\\
&&&&\\
\cline{2-5}
&&&&\\
&$\Id$&$\begin{aligned}
&
\textnormal{Type $II$}\\
&\textnormal{Initial Matrix:\ } &&\begin{pmatrix}1\\-C'\end{pmatrix}\\
&\textnormal{Final Matrix:\ }
&&\begin{pmatrix}C&-p\end{pmatrix}\kappa_j^{-1}
\end{aligned}$&$\begin{aligned}
&\textnormal{Type $II$}\\
&\textnormal{Initial Matrix:\ } &&\begin{pmatrix}1\\-C'\end{pmatrix}\\
&\textnormal{Final Matrix:\ }
&&\begin{pmatrix}-p&C\end{pmatrix}\kappa_j^{-1}
\end{aligned}$&$\begin{aligned}
&
\textnormal{Type $II$}\\
&\textnormal{Initial Matrix:\ } &&\begin{pmatrix}1\\-C'\end{pmatrix}\\
&\textnormal{Final Matrix:\ }
&&\begin{pmatrix}0&1\end{pmatrix}\kappa_j^{-1}
\end{aligned}$\\
&&&&\\
\hline
\hline
&&&&\\
\multirow{ 2}{*}{$=1$}&$(12)$&{\color{blue}$\begin{aligned}
&
\textnormal{Type $I$}\\
&\textnormal{Initial Matrix:\ } &&\begin{pmatrix}D&-B\\-C&p-D\end{pmatrix}\\
&\textnormal{Final Matrix:\ }
&&\begin{pmatrix}p-D&-C\\-B&D\end{pmatrix}\kappa_j^{-1}
\end{aligned}$}&{\color{blue}$\begin{aligned}
&
\textnormal{Type $I$}\\
&\textnormal{Initial Matrix:\ } &&\begin{pmatrix}D&-B\\-C&p-D\end{pmatrix}\\
&\textnormal{Final Matrix:\ }
&&\begin{pmatrix}D&-B\\-C&p-D\end{pmatrix}\kappa_j^{-1}
\end{aligned}$}&$\begin{aligned}
&
\textnormal{Type $II$}\\
&\textnormal{Initial Matrix:\ } &&\begin{pmatrix}1\\-C'\end{pmatrix}\\
&\textnormal{Final Matrix:\ }
&&\begin{pmatrix}0&1\end{pmatrix}\kappa_j^{-1}
\end{aligned}$\\
&&&&\\
\cline{2-5}
&&&&\\
&$\Id$&$\begin{aligned}
&
\textnormal{Type $II$}\\
&\textnormal{Initial Matrix:\ } &&\begin{pmatrix}1\\-C'\end{pmatrix}\\
&\textnormal{Final Matrix:\ }
&&\begin{pmatrix}C&-p\end{pmatrix}\kappa_j^{-1}
\end{aligned}$&$\begin{aligned}
&
\textnormal{Type $II$}\\
&\textnormal{Initial Matrix:\ } &&\begin{pmatrix}1\\-C'\end{pmatrix}\\
&\textnormal{Final Matrix:\ }
&&\begin{pmatrix}-p&C\end{pmatrix}\kappa_j^{-1}
\end{aligned}$&
{\color{blue}$\begin{aligned}
&
\textnormal{Type $I$}\\
&\textnormal{Initial Matrix:\ } &&\begin{pmatrix}D&-B\\-C&p-D\end{pmatrix}\\
&\textnormal{Final Matrix:\ }
&&\begin{pmatrix}p-D&B\\C&D\end{pmatrix}\kappa_j^{-1}
\end{aligned}$}\\
&&&&\\
\hline
\hline
&&&&\\
$=0$&$\Id$&{\color{orange}$\begin{aligned}
&
\textnormal{Type $0$}\\
&\textnormal{Transition Matrix:\ } &&\begin{pmatrix}1&0\\-C&p\end{pmatrix}\kappa_j^{-1}
\end{aligned}$}&{\color{orange}$
\begin{aligned}
&
\textnormal{Type $0$}\\
&\textnormal{Transition Matrix:\ } &&\begin{pmatrix}0&-1\\-p&C\end{pmatrix}\kappa_j^{-1}
\end{aligned}
$}&{\color{orange}$
\begin{aligned}
&
\textnormal{Type $0$}\\
&\textnormal{Transition Matrix:\ } &&\begin{pmatrix}p&-B\\0&1\end{pmatrix}\kappa_j^{-1}
\end{aligned}$}\\
&&&&\\
\hline
\hline
\end{tabular}}
\captionsetup{justification=raggedright,
singlelinecheck=false
}
\caption*{
The meaning of the variables in this table are the same as that of the corresponding entry in Table \ref{Table:Traductions}.
}%
\end{table}
\begin{prop}\label{prop:naive vs sat model} \begin{enumerate}
\item The inclusion $\tld{\cZ}^{\tmod,\tau}(\tld{z})\into \tld{U}(\tld{z})$ factors through $\tld{\cZ}^{\nv,\tau}(\tld{z})$. 
\item We have $\tld{\cZ}^{\tmod,\tau}(\tld{z})[\frac{1}{p}]=\tld{\cZ}^{\nv,\tau}(\tld{z})[\frac{1}{p}]=\tld{Y}^{\tmod,\eta,\tau}(\tld{z})[\frac{1}{p}]$.
\end{enumerate}
In particular, $\tld{\cZ}^{\tmod,\tau}(\tld{z})$ is the $p$-saturation (synonymously, the $\cO$-flat part) of $\tld{\cZ}^{\nv,\tau}(\tld{z})$.
\end{prop}
\begin{proof} The first assertion follows from the fact that  $\tld{\cZ}^{\tmod,\tau}(\tld{z})$ obey the defining equations of $\tld{\cZ}^{\nv,\tau}(\tld{z})$, which is a consequence of the relations in Table \ref{Table:Traductions} and the relation $r_j=l_{j-1}$ in $\tld{Y}^{\tmod,\eta,\tau}$. Indeed, these defining equations were obtained by repeatedly substituting the relations between $l_j$ and $r_j$ when $j$ is type $0$ and $r_j=l_{j-1}$ until it becomes a relation between $r_a$ and $l_b$ where $a,b$ are not type $0$, in which case one substitutes for $r_a$, $l_b$ an expression in the variables on $\GL_2\times M_a(\tld{z}_a)$, $\GL_2\times M_b(\tld{z}_b)$. 

We now establish the second assertion. 
First, we show that the map $\tld{Y}^{\tmod,\eta,\tau}(\tld{z})\to \tld{U}(\tld{z})$ is a closed immersion after inverting $p$, i.e.~we need to show $r_j,l_j$ are determined by the remaining variables. For each $j$ not of type $0$, we can solve for $r_j,l_j$ when $p$ is invertible. Using the relation $r_j=l_{j-1}$ and the relations in Table \ref{Table:Traductions}, we can solve for the $r_{j'},l_{j'}$ where $j'$ is of type $0$. Thus $\tld{\cZ}^{\tmod,\tau}(\tld{z})[\frac{1}{p}]=\tld{Y}^{\tmod,\eta,\tau}(\tld{z})[\frac{1}{p}]$.

To finish the proof, we need to show that $\tld{Y}^{\tmod,\eta,\tau}(\tld{z})[\frac{1}{p}]$ surjects onto $\tld{\cZ}^{\nv,\tau}(\tld{z})[\frac{1}{p}]$, i.e.~we need to produce $r_j,l_j$ satisfying all requisite relations. We use the same procedure to define $r_j,l_j$ as in the previous paragraph. The only potential issue is that for a fragment $\cJ_k=\{i,\cdots, o\}$, the procedure gives two definitions of $r_i$: one by recursion in terms of $l_o$ (and some variables at $j\in \{i-1,\cdots, o+1\}$), the other by directly solving in terms of the variables at $i$. However the defining equations of $\tld{\cZ}^{\nv,\tau}(\tld{z})$ exactly guarantee that the these two definitions coincide.
\end{proof}

\subsection{Obstruction bounds for naive models}
We wish to establish a quantitative bound for the singular ideal of $\tld{\cZ}^{\nv,\tau}(\tld{z})/\cO$. This will be used in the next subsection to finish the proof of Theorem \ref{thm:main:model}.

To save notation, in this section we abbreviate $\cZ=\tld{\cZ}^{\nv,\tau}(\tld{z})$, $M_j=M_j(\tld{z}_j)$ and $\cM=\prod_{\cJ} \GL_2\times M_j$. We also define affine spaces $\cA_j$ with closed immersions $M_j\into \cA_j$ as follows:
\begin{itemize}
\item $\cA_j=M_j$ if $j$ is not type $I$.
\item If $j$ is type $I$, then $M_j$ is the space $M$ of matrices $\begin{pmatrix} A & B\\C&D \end{pmatrix}$ with determinant $0$ and trace $p$, and we define $\cA_j=\bA^3$ and $M_j\into \cA_j$ to be the map 
\[\begin{pmatrix} A & B\\C&D \end{pmatrix}\mapsto (B,C,D)\]
so that $M_j$ identifies with the hypersurface $D(p-D)=BC$ in $\cA_j$.
\end{itemize}
We thus have an inclusion $\iota: \cZ\into \cA\defeq \prod_{\cJ} \GL_2\times \cA_j$. Let $\cI$ be the ideal of $\cA$ defining $\iota$.

We first analyze a generating set of $\cI$. To do this, recall the fragmentation $\cJ=\bigcup_{k\in \cK} \cJ_k$. We also denote by $\cJ_I\subset \cJ$ the subset of $j$ that are type $I$. For each $j$ of type $I$, we also let $A_j,B_j,C_j,D_j$ denote the natural coordinates on $M_j$, so that $A_j+D_j=p$ and $A_jD_j=B_jC_j$.

It follows from Definition \ref{defn:naive equations} that we have a decomposition
\begin{equation}
\label{eq:ideal:I_J_k}
\cI=\sum_k \cI_{\cJ_k}+\sum_{j \in \cJ_I} (D_j(p-D_j)-B_jC_j)
\end{equation}
where $\cI_{\cJ_k}$ is the ideal generated by the entries of the matrix equation associated to $\cJ_k$ in Definition \ref{defn:naive equations}. This gives a presentation of $\cZ$ in terms of $\cA$. Note that $\cZ[\frac{1}{p}]$ has codimension $c\defeq |\cK|+|\cJ_I|$ in $\cA[\frac{1}{p}]$.

\begin{prop}\label{prop:elkik bound} Let $\mathrm{J}_c$ denote the ideal generated by the $c\times c$ minors of the Jacobian matrix of the presentation $\cO(\cA)/\cI$. Then $p^{2|\cJ_I|+f+|\cK|}\in \mathrm{J}_c$. In particular $p^{4f}\in \mathrm{J}_c$.
\end{prop}
\begin{rmk} \label{rmk:regular generic fiber} Proposition \ref{prop:elkik bound} implies $\tld{\cZ}^{\nv,\tau}(\tld{z})[\frac{1}{p}]$ is smooth over $E$. Of course this can also be seen directly from the explicit description of $\tld{Y}^{\textnormal{mod},\eta,\tau}(\tld{z})[\frac{1}{p}]$.
\end{rmk}

We observe the following structural properties of $\cI_{\cJ_k}$:
\begin{lemma}\label{lem:ideal structure} Suppose the fragment $\cJ_k=\{i,\cdots ,o\}$. Let $\kappa_o=\begin{pmatrix}a_o & b_o \\ c_o &d_o \end{pmatrix}$ record the $o$-th $\GL_2$ factor of $\cA$.

There exists $\alpha,\beta,\gamma,\delta$ only involving coordinates on the $(i-1)$-th to $(o+1)$-th factor of $\cA$ with
\[\det \begin{pmatrix} \alpha &\beta \\\gamma &\delta \end{pmatrix}=\pm p^{|\cJ_k\setminus \{i,o\}|}\]
such that: 
\begin{enumerate}
\item If $i, o$ are type $II$ then $\cI_{\cJ_k}$ is principal, generated by an element of the form
\[\begin{pmatrix}Y_o &-p\end{pmatrix}\kappa_o\begin{pmatrix} \alpha &\beta \\\gamma &\delta \end{pmatrix} \begin{pmatrix} 1 \\ -X_i\end{pmatrix}\]
where $Y_o$, $X_i$ are coordinates of $\cA_o,\cA_i$.
\item If $i, o$ are type $I$ then $\cI_{\cJ_k}$ is generated by the entries of a matrix either of the form
\begin{itemize}
\item
\[\begin{pmatrix} D_o &-B_o\\-C_o & A_o\end{pmatrix}\kappa_o\begin{pmatrix} \alpha &\beta \\\gamma &\delta \end{pmatrix} \begin{pmatrix} D_i &-B_i\\ -C_i& A_i\end{pmatrix};\]
or
\item
\[\begin{pmatrix} D_o &C_o\\B_o & A_o\end{pmatrix}\kappa_o\begin{pmatrix} \alpha &\beta \\\gamma &\delta \end{pmatrix} \begin{pmatrix} D_i &-B_i\\ -C_i& A_i\end{pmatrix};\]
\end{itemize}

\item If $i$ is type $II$ and $o$ is type $I$ then $\cI_{\cJ_k}$ is generated by the entries of a matrix of the form
\[\begin{pmatrix} Y_o &-p\end{pmatrix}\kappa_o\begin{pmatrix} \alpha &\beta \\\gamma &\delta \end{pmatrix} \begin{pmatrix} D_i &-B_i\\ -C_i& A_i\end{pmatrix}\]
where $Y_o$ is a coordinate in $\cA_o$.

\item If $i$ is type $I$ and $o$ is type $II$ then $\cI_{\cJ_k}$ is generated by the entries of a matrix either of the form
\begin{itemize}
\item
\[\begin{pmatrix} D_o &-B_o\\-C_o & A_o\end{pmatrix}\kappa_o\begin{pmatrix} \alpha &\beta \\\gamma &\delta \end{pmatrix} \begin{pmatrix} 1 \\ -X_i\end{pmatrix};\]
or
\item
\[\begin{pmatrix} D_o &C_o\\B_o & A_o\end{pmatrix}\kappa_o\begin{pmatrix} \alpha &\beta \\\gamma &\delta \end{pmatrix} \begin{pmatrix} 1\\-X_i\end{pmatrix};\]
\end{itemize}
where $X_i$ is a coordinate on $\cA_i$.

\end{enumerate}
\end{lemma}
\begin{proof}
The form of $\cI_{\cJ_k}$ follows from Table \ref{Table:matrices}, after possibly rearranging the entries of the matrix equations.

\end{proof}
\begin{proof}[Proof of Proposition \ref{prop:elkik bound}]
We will show that for any choices $G_j,H_j\in \{A_j,D_j\}$ with $j\in \cJ_I$, the element
$p^{2|\cJ_I|+\sum_k |\cJ_k\setminus \{i,o\}|}\prod_{j\in \cJ_I}G_jH_j$ belongs to $\mathrm{J}_c$. This finishes the proof, since this implies $\mathrm{J}_c$ contains
\[p^{2|\cJ_I|+\sum_k |\cJ_k\setminus \{i,o\}|}\prod_{j\in\cJ_I}(A_j,D_j)\ni p^{2|\cJ_I|+\sum_k |\cJ_k\setminus \{i,o\}|+2} \]

Since the role of $A_j$, $D_j$ is essentially symmetric in our argument, we will deal with the case $G_j=H_j=D_j$ for all $j\in \cJ$.

Write $\mathrm{J}$ for the Jacobian matrix of our chosen presentation of $\cZ$. Our convention is that the columns are named by the variables and rows are labeleld by (the chosen) generators of $\cI$. Then clearly $\mathrm{J}_c$ contains any $c\times c$ minor of any matrix of the form $\mathrm{J}U$. In other words, it suffices to find, after modifying $\mathrm{J}$ by column operations, a $c\times c$ minor equal to
$p^{2|\cJ_I|+\sum_k |\cJ_k\setminus \{i,o\}|}\prod_{j\in \cJ_I}G_jH_j$.

For $k\in \cK$, we let $g_k$ denote the generator of $\cI_{\cJ_k}$ that correspond to the $(1,1)$-th entry of the matrix equation described in Lemma \ref{lem:ideal structure} (this choice correspond to our choice of $G_j=H_j=D_j$). We will choose our $c\times c$ minor to have rows corresponding to the $g_k$ with $k\in \cK$ and the generators $(p-D_j)D_j-B_jC_j$ for $j\in\cJ_I$.

We now explain the column operations we will perform on $\mathrm{J}$.

First, consider $k\in \cK$, giving the fragment $\cJ_k=\{i,\cdots, o\}$. Write the matrix generating $\cI_{\cJ_k}$ in factorized form 
\[X\begin{pmatrix}a_o &b_o\\c_o &d_o \end{pmatrix}\begin{pmatrix}\alpha &\beta\\ \gamma &\delta \end{pmatrix}Y\]
 as prescribed by Lemma \ref{lem:ideal structure}. Then the entries of $\cJ_k$ corresponding to row $g_k$ and column $a_o,b_o,c_o,d_o$ are exactly the $(1,1)$-th entry of the matrices
\begin{align*}
X\begin{pmatrix}1 &0\\0&0 \end{pmatrix}\begin{pmatrix}\alpha &\beta\\ \gamma &\delta \end{pmatrix}Y,&& X\begin{pmatrix}0 &1\\0&0 \end{pmatrix}\begin{pmatrix}\alpha &\beta\\ \gamma &\delta \end{pmatrix}Y,&& X\begin{pmatrix}0 &0\\1&0 \end{pmatrix}\begin{pmatrix}\alpha &\beta\\ \gamma &\delta \end{pmatrix}Y,&&X\begin{pmatrix}0 &0\\0&1 \end{pmatrix}\begin{pmatrix}\alpha &\beta\\ \gamma &\delta \end{pmatrix}Y
\end{align*}
respectively. Furthermore, no other row of $\mathrm{J}$ has a non-zero entry at the columns $a_o,b_o,c_o,d_o$. Hence, given $x,y,z,t\in \cO(\cA)$, by modifying the columns $\mathrm{J}$ to $\mathrm{J}U$ where we take linear combinations of columns $a_o,b_o,c_o,d_o$ but do not modify the remaining columns, we can guarantee that the resulting matrix has a column which has vanishing entry for any row other than $g_k$, and at row $g_k$ the entry is the $(1,1)$-th entry of  
\[X\begin{pmatrix} x &y\\z&t \end{pmatrix}\begin{pmatrix}\alpha &\beta\\ \gamma &\delta \end{pmatrix}Y\]
Now $\det \begin{pmatrix}\alpha &\beta\\ \gamma &\delta \end{pmatrix}=\pm p^{|\cJ_k\setminus \{i,o\}|}$, so choosing $x,y,z,t$ appropriately, we can make $\begin{pmatrix} x &y\\z&t \end{pmatrix}\begin{pmatrix}\alpha &\beta\\ \gamma &\delta \end{pmatrix}$ become any matrix among 
\[p^{|\cJ_k\setminus \{i,o\}|}\{\begin{pmatrix} 1 &0\\0&0 \end{pmatrix},\begin{pmatrix} 0 &1\\0&0 \end{pmatrix},\begin{pmatrix} 0 &0\\1&0 \end{pmatrix},\begin{pmatrix} 0 &0\\0&1 \end{pmatrix}\}\]
Using the explicit form of $X$, $Y$ given in Lemma \ref{lem:ideal structure}, we thus learn that after taking a linear combination of columns $a_o,b_o,c_o,d_o$, we can make the entry in row $g_k$ become $p^{|\cJ_k\setminus \{i,o\}|}p$, $p^{|\cJ_k\setminus \{i,o\}|}D_oD_i$, $p^{|\cJ_k\setminus \{i,o\}|}pD_i$ or $p^{|\cJ_k\setminus \{i,o\}|}D_o$ (corresponding to the four cases in that Lemma, respectively).

Next, we consider $j\in \cJ_I$, giving a generator $(p-D_j)D_j-B_jC_j$ of $\cI$. The part of $\mathrm{J}$ corresponding to row $(p-D_j)D_j-B_jC_j$ and column $B_j,C_j,D_j$ is 
\[\begin{pmatrix} -C_j & -B_j & p-2D_j\end{pmatrix}\]
Since 
\[\begin{pmatrix} -C_j & -B_j & p-2D_j\end{pmatrix}\begin{pmatrix} -4C_j \\ 0 \\ p-2D_j\end{pmatrix}=(p-2D_j)^2+4B_jC_j=p^2\]
we see that taking a linear combination of column $B_j,C_j,D_j$ produces the entry $p^2$ in row $(p-D_j)D_j-B_jC_j$. Moreover that the entries of this linear combination at rows $(p-D_{j'})D_{j'}-B_{j'}C_{j'}$ are $0$, where $j'\in \cJ_I$ but $j'\neq j$ (note however that we have no control on the entries on the remaining rows). 

To summarize, by taking appropriate linear combination among columns $a_o,b_o,c_o,d_o$ (for each $\cJ_k=\{i,\cdots, o\}$) and among columns $B_j, C_j, D_j$ (for each $j\in \cJ_I$), and look at the rows $g_k$, $(p-D_j)D_j-B_jC_j$, we can find a $c\times c$ submatrix with rows $g_k, (p-D_j)D_j-B_jC_j$ of block triangular  form
\[\begin{pmatrix} P & * \\ 0 & Q\end{pmatrix}\]
where 
\begin{itemize}
\item $P$ is diagonal of size $|\cK|\times |\cK|$ whose entries belong to 
\[
\{p^{|\cJ_k\setminus \{i,o\}|}p,p^{|\cJ_k\setminus \{i,o\}|}D_oD_i,p^{|\cJ_k\setminus \{i,o\}|}pD_i,p^{|\cJ_k\setminus \{i,o\}|}D_o\};
\]
\item $Q$ is $p^2$ the identity matrix of size $|\cJ_I|$.
\end{itemize} 
It follows that this $c\times c$ minor divides
\[p^{2|\cJ_I|+\sum_k |\cJ_k\setminus \{i,o\}|}\prod_{j\in \cJ_I}D_j^2\]

\end{proof}

\subsection{Proof of Theorem \ref{thm:main:model}}\label{subsect:main proof}
Set $N=p-2-\max_j\langle \mu_j,\alpha^\vee\rangle\geq \frac{p-7}{2}$. Recall from Proposition \ref{prop:heart:inverse} that we have a diagram
\[
\xymatrix{
\tld{\cZ}^{\tau}\otimes_{\cO}\cO/p^{N-1}\ar^{\sim}[r]\ar@{^{(}->}[d]
& \tld{\cZ}^{\tmod,\tau}\otimes_{\cO}\cO/p^{N-1}\ar@{^{(}->}[d]\\
\left[L\rG^{\textnormal{bd}, (v+p)v^\mu}/_{\phz}\prod_{\cJ}L_1^+\rG\right]\otimes_{\cO}\cO/p^{N-1}\ar^-{\sim}[r]&\Gr_1^{\textnormal{bd}, (v+p)v^\mu}\otimes_{\cO}\cO/p^{N-1}
}
\]
In particular, from Proposition \ref {prop:naive vs sat model} we get a closed immersion
\[
\iota: \tld{\cZ}^{\tau}(\tld{z})\otimes_{\cO}\cO/p^{N-1}\stackrel{\sim}{\longrightarrow}\tld{\cZ}^{\textnormal{mod},\tau}(\tld{z})\otimes_{\cO}\cO/p^{N-1}\into\tld{\cZ}^{\nv,\tau}(\tld{z})\otimes_{\cO}\cO/p^{N-1}
\]
We now invoke \cite[Lemme 1]{Elkik} as in the proof of \cite[Proposition 3.3.9]{MLM}: By Proposition \ref{prop:elkik bound}, the integer $h$ in ~\emph{loc.cit.} can taken to be $4f$ while the integer $k$ is $0$ since $\tld{\cZ}^{\tau}(\tld{z})$ is $p$-torsion free. It follows that if $N-1>8f$, we can produce a map $\tld{\iota}:\tld{\cZ}^{\tau}(\tld{z})\to \tld{\cZ}^{\nv,\tau}(\tld{z})^{\wedge_p}$ which agrees with $\iota$ modulo $p^{N-1-4f}$. In particular, this implies $\tld{\iota}$ is also a closed immersion.
Since $\tld{\cZ}^{\tau}(\tld{z})$ is $\cO$-flat and the $\cO$-flat part of $\tld{\cZ}^{\nv,\tau}(\tld{z})^{\wedge_p}$ is exactly $\tld{\cZ}^{\textnormal{mod},\tau}(\tld{z})$, we can factorize $\tld{\iota}$
\[\tld{\cZ}^{\tau}(\tld{z})\into \tld{\cZ}^{\tmod,\tau}(\tld{z})\into \tld{\cZ}^{\nv,\tau}(\tld{z})^{\wedge_p}\]
such that the first map is an isomorphism modulo $p^{N-1-4f}$. The following Lemma then implies that the inclusion $\tld{\cZ}^{\tau}(\tld{z})\into \tld{\cZ}^{\tmod,\tau}(\tld{z})$ is in fact an isomorphism, thus finishing the proof of Theorem \ref{thm:main:model}.
\begin{lemma} Suppose we are given a surjection $\pi:R\onto S$ of Noetherian $p$-adically complete $\cO$-algebras. Assume that $\pi$ induces an isomorphism $\pi:R/\varpi\cong S/\varpi$. Then $\pi$ is an isomorphism.
\end{lemma}
\begin{proof}
Let $I=\ker \pi$, then since $S$ is $\cO$-flat we get a short exact sequence
\[\xymatrix{0\ar[r] & I/\varpi I\ar[r] & R/\varpi \ar[r]& S/\varpi \ar[r] & 0}\]
Since $\pi$ induces an isomorphism modulo $\varpi$, we learn that $I/\varpi I=0$. But $I$ is $p$-adically separated, so $I=0$.
\end{proof}
The following is immediate from our discussion
\begin{cor} \label{cor:model} Assume that either $p> 8f+3+\max_j\langle \mu_j,\alpha^\vee\rangle$ or $p>7$ and $K=\Qp$. 
Then $\tld{\cZ}^{\tau}(\tld{z})$ is isomorphic to the $p$-adic completion of the $p$-saturation of $\tld{\cZ}^{\nv,\tau}(\tld{z})$.
\end{cor}
In particular, this realizes $\tld{\cZ}^{\tau}(\tld{z})$ as the $p$-saturation of an explicitly presented affine $p$-adic formal scheme.
\subsection{Applications}

\subsubsection{Galois deformation rings}
\label{subsub:Gal:def}
Recall that we have a shifted conjugation action map
\[ \GL_2^{\cJ} \times \Gr_1^{\cJ}\to \Gr_1^{\cJ}\]
given by the formula
\[(g_j,A_j)\mapsto g_jA_jg_{j-1}^{-1}\]
This action of $ \GL_2^{\cJ}$ clearly factors through the quotient $ \GL_2^{\cJ}/\Delta Z$ by the diagonally embedded copy of the center $Z$ of $\GL_2$. Furthermore since
\[\prod \det g_jA_jg_{j-1}^{-1}=\prod \det A_j\]
we see that the quantity $\prod \det A_j$ modulo $\det L_1^+G$ is invariant along orbits.
\begin{lemma}\label{lem:action at semisimple point} Let $\tld{z}=(\tld{z}_j)=(z_jv^{\nu_j})\in\tld{\un{W}}^{\vee}$ such that $\langle\nu_j,\alpha^\vee\rangle\neq 0$ for some $j$.
\begin{enumerate}
\item
Suppose $\prod z_j=(12)$. Let $\ovl{x}=((1,\cdots 1),(\tld{z}_j)_j)\in \GL_2^{\cJ}(\F)\times \prod_{\cJ}L_1^{-}G(\F)\tld{z}_j$ and $\ovl{y}$ its image under the shifted conjugation action.
The shifted conjugation action map induces an isomorphism on completions at $\ovl{x}$:
\[\bigg(\GL_2^{\cJ}/\Delta Z \times \prod_{\cJ}L_1^{-}G\tld{z}_j\bigg)^{\wedge}_{\ovl{x}} \cong \bigg( \GL_2^{\cJ,\det=1}\times \prod_{\cJ} L_1^{-}G\tld{z}_j\bigg)^{\wedge}_{\ovl{y}}\]
(Here $ \GL_2^{\cJ,\det=1}$ denotes the kernel of the product of determinant map $\GL_2^{\cJ}\to \bG_m$.)
\item Suppose $\prod z_j=1$. Choose a transversal slice $V$ to $T^\vee$ near $1$ in $\GL_2$ such that the tangent space $T_1V=\mathfrak{n}\oplus\ovl{\mathfrak{n}}=\{\begin{pmatrix} 0 & b \\ c &0  \end{pmatrix}\}\subset \fgl_2$. Let $\ovl{x}\in \GL_2^{\cJ}(\F)\times T^\vee(\F)\tld{z}_0\times \prod_{j\neq 0}L_1^{-}G(\F)\tld{z}_j$ be the tuple $((1,1,\cdots 1),(\ovl{t}\tld{z}_0,(\tld{z}_j)_{j\neq 0}))$, and $\ovl{y}$ its image under the shifted conjugation action.
Then the shifted conjugation action map induces an isomorphism on completion at $\ovl{x}$
\[\bigg(V \times \GL_2^{\cJ\setminus \{0\}}\times (T^\vee\times L_1^{-}G\tld{z}_0)\times \prod_{j\neq 0}L_1^{-}G\tld{z}_j\bigg)^{\wedge}_{\ovl{x}} \cong \bigg( \GL_2^{\cJ}\times \prod_{\cJ} L_1^{-}G\tld{z}_j\bigg)^{\wedge}_{\ovl{y}}\]
\end{enumerate}
\end{lemma}
\begin{rmk} The case $\langle\nu_j,\alpha^\vee\rangle= 0$ for all $j$ happens exactly when $\tld{z}_j\in  W^{\vee}Z(\F(\!(v)\!))$ for all $j$. In this case, the same method of proof shows that for $\ovl{x}=((1,1,\cdots 1),(\ovl{\kappa}\tld{z}_0,(\tld{z}_j)_{j\neq 0}))$ with image $\ovl{y}$ under the shifted conjugation action
\[\bigg( \GL_2^{\cJ\setminus \{0\}}\times (\GL_2\times L_1^{-}G\tld{z}_0)\times \prod_{j\neq 0}L_1^{-}G\tld{z}_j\bigg)^{\wedge}_{\ovl{x}} \cong \bigg( \GL_2^{\cJ}\times \prod_{\cJ} L_1^{-}G\tld{z}_j\bigg)^{\wedge}_{\ovl{y}}\]
\end{rmk}
\begin{proof} 
\begin{enumerate}
\item On a tangent vector $(1+\eps K_j,(1+\eps L_j)\tld{z}_j)$ at $\ovl{x}$, the formula for the action is
\[(1+\eps K_j)(1+\eps L_j)\tld{z}_j(1-\eps K_{j-1})=(1+\eps K_j-\eps\Ad(\tld{z}_j)(K_{j-1})+\eps L_j)\tld{z}_j\]
where $K_j\in \mathrm{Lie}\GL_2=\fgl_2$, $L_j\in \mathrm{Lie}L_1^{-}G=\frac{1}{v}\fgl_2[\frac{1}{v}]$.

Hence we need to show the map
\[S:(\fgl_2\times \frac{1}{v}\fgl_2[\frac{1}{v}])^{\cJ} \to (v\fgl_2[\![v]\!]\backslash \fgl_2(\!(v)\!))^{\cJ}\cong  (\fgl_2\oplus \frac{1}{v}\fgl_2[\frac{1}{v}])^{\cJ}\]
given by
\[(K_j,L_j)\mapsto K_j-\Ad(\tld{z}_j)(K_{j-1})+L_j\]
is an isomorphism onto the subspace of the target consisting of tuples $(Y_j)$ whose projection to $\fgl_2^{\cJ}$ factors through the subspace $\fgl_2^{\cJ,\mathrm{tr}=0}$ with sum of tuples whose sum of traces is $0$. 

Define $\Pi_j:\fgl_2\to \fgl_2$ to be the linear endomorphism given by sending $X$ to the projection of $\Ad(\tld{z}_j)X\in \fgl_2(\!(v)\!)$ onto the $\fgl_2$ summand, i.e.~extracting the $v$-degree $0$ part of $\Ad(\tld{z}_j)X$. Since $\tld{z}_j=z_jv^{\nu_j}$ we see that 
\begin{itemize}
\item $\Pi_j(\begin{pmatrix}a & b \\ c& d\end{pmatrix})=\Ad(z_j)\begin{pmatrix}a & b \\ c& d\end{pmatrix}$ if $\langle \nu_j,\alpha^\vee\rangle =0$; and
\item $\Pi_j(\begin{pmatrix}a & b \\ c& d\end{pmatrix})=\Ad(z_j)\begin{pmatrix}a & 0 \\ 0& d\end{pmatrix}$ if $\langle \nu_j,\alpha^\vee\rangle \neq 0$
\end{itemize}
We first show that $S$ is injective.
Let $((K_j)_j,(L_j)_j)\in \ker S$ then
\[K_0=\Pi_0\Pi_{-1}\cdots \Pi_{-f+1}(K_0)\]
Set $K_0=\begin{pmatrix}a & b\\c& d \end{pmatrix}$. Then our hypothesis on $\tld{z}_j$ shows that the composite $\Pi_0\Pi_{-1}\cdots \Pi_{-f+1}$ sends $K_0$ to $\begin{pmatrix}d & 0\\0& a \end{pmatrix}$. Thus we get $K_0\in \mathrm{Lie}(Z)$, and hence $K_j=K_0\in \mathrm{Lie}(Z)$ for all $j$, and then $L_j=0$. This shows the injectivity of $S$.

For surjectivity, it suffices to prove the surjectivity after projecting to the $\fgl_2^{\cJ}$ summand.
Let $(Y_j)\in \fgl_2^{\cJ,\mathrm{tr}=0}$. By repeated substitution in the system
\[K_j-\Pi_j(K_{j-1})=Y_j\]
we see that the system has a solution if and only if
\[K_0-\Pi_0\Pi_{-1}\cdots \Pi_{-f+1}(K_0)=\Pi_0(Y_{-1})+\Pi_0\Pi_{-1}(Y_{-2})+\cdots \Pi_0\cdots\Pi_{-f+1}(Y_{-f})\]
has a solution $K_0=\begin{pmatrix}a & b\\c& d \end{pmatrix}$. Note that the right-hand side has trace $0$. But
\[K_0-\Pi_0\Pi_{-1}\cdots \Pi_{-f+1}(K_0)=\begin{pmatrix}a & b\\c& d \end{pmatrix}-\begin{pmatrix}d & 0\\0& a \end{pmatrix}=\begin{pmatrix}a-d & b\\c& d-a \end{pmatrix}\]
hence $K_0-\Pi_0\Pi_{-1}\cdots \Pi_{-f+1}(K_0)$ can become any trace $0$ matrix. This gives surjectivity of $S$.
\item The argument is similar to the previous case. The only difference is that at $j=0,1$, the tangent vector equations (on the $v$-degree $0$ part) we get are modified to 
\[V_0+T_0-\Pi_0(K_{-1})=Y_0\]
\[ K_1-\Pi_1(V_0)=Y_1\]
where $V_0\in T_1V=\mathfrak{n}\oplus\ovl{\mathfrak{n}}$ and $T_0\in \fT^\vee=\mathrm{Lie}T^\vee$.

The existence and uniqueness of solutions to the system now boils down to solvability (in terms of $V_0,T_0$) of 
\[V_0+T_0-\Pi_0\Pi_{-1}\cdots \Pi_{-f+1}(V_0)=\Pi_0(Y_{-1})+\Pi_0\Pi_{-1}(Y_{-2})+\cdots \Pi_0\cdots\Pi_{-f+1}(Y_{-f}),\]
which now always has a unique solution since the left-hand side reduces to $V_0+T_0$.
\end{enumerate}
\end{proof}
\begin{thm} 
\label{thm:main:Gal:def}
Let $\tau$ be a regular tame type with small presentation $(s,\mu)$ and assume that either $p> 8f+3+\max_j\langle \mu_j,\alpha^\vee\rangle$ or $p\geq 11$ and $K=\Qp$.  
Suppose $\tld{z}_j=\tld{w}_js_j^{-1}v^{\mu_j}$ with $\tld{w}=(\tld{w}_j)\in \Adm^{\vee}(\eta)$. Assume that for at least one $j$, $\tld{z}_j\notin W^{\vee}Z(\F(\!(v)\!))$. Let  $\ovl{t}\in T^\vee(\F)$ and $\rhobar$ be the unique semisimple Galois representation such that the matrix of the associated $\phz^f$-module is $\ovl{t}\tld{z}_0\phz(\tld{z}_{f-1})\cdots\phz^{f-1}(\tld{z}_{1})$ (see \S \ref{subsub:etale:Phi}).

\begin{enumerate} 
\item Suppose $\rhobar$ is absolutely irreducible, and let $R^{\eta,\tau}_{\rhobar}$ the potentially Barsotti-Tate deformation ring of type $\tau$. Then, up to enlarging $\F$, $R^{\eta,\tau}_{\rhobar}$ is isomorphic to the completion of the $p$-saturation of $\cZ^{\nv,\tau}(\tld{z})$ at $\tld{z}$ (see Definition \ref{defn:naive equations}).
\item 
\label{thm:main:Gal:def:2}
Suppose $\rhobar$ is reducible, and let $R^{\eta,\tau}_{\rhobar}$ be the framed potentially Barsotti-Tate deformation ring of type $\tau$. Then up to adding formal variables, $R^{\eta,\tau}_{\rhobar}$ is isomorphic to the completion of the $p$-saturation of $(\tld{\cZ}^{\tmod,\tau}(\tld{z})\times_{\GL_2^{\cJ}}(T\times \{1\}^{\cJ\setminus \{0\}})$ at $(\ovl{t}\tld{z}_0,\tld{z}_1,\cdots ,\tld{z}_{f-1})$.
\end{enumerate}
\end{thm}
\begin{proof} By Corollary \ref{cor:model}, Lemma \ref{lem:action at semisimple point} and the fact that $\tld{\cZ}^{\tmod,\tau}$ is invariant under the shifted conjugation action, $R^{\eta,\tau}_{\rhobar}$ has the desired description up to adding formal power series variables. This immediately gives the second item, and the first item up to adding power series variables.
To remove the potential extra power series variables, we use the following fact \cite[Theorem 5]{Hamann}: if $R$, $S$ are complete local rings such that $R[\![X]\!]\cong S[\![X]\!]$ then $R\cong S$.
\end{proof}

\begin{rmk}\label{rmk:simplified model}
It follows from the definitions that up to isomorphism, $\cZ^{\nv,\tau}(\tld{z})$ is described as follows:

Given $(\tld{w}_j)_{j\in\cJ}\in \Adm^\vee(\eta)$, the fragmentation $\cJ=\bigcup_{k\in\cK}\cJ_k$, and the type of the endpoints of each $\cJ_k$, $\cZ^{\textnormal{nv},\tau}(\tld{z})$ is the spectrum of $\otimes_{j\in\cJ}R_j/\sum_{k}\cI_{\cJ_k}$ where for each fragment $\cJ_k=(i,i-1,\dots,o+1,o)$, $\cI_{\cJ_k}$ is the ideal of $\otimes_{j\in\cJ_k}R_j$ generated by the equations 
\begin{equation}
\label{matrix:equation:explicit}
M_{\textnormal{out},o}\cdot\left( \prod_{\ell=o+1}^{i-1}T_\ell\right)\cdot M_{\textnormal{in},i}=0,
\end{equation} 
and $M_{\textnormal{out},o}$ (resp.~$M_{\textnormal{in},i}$, resp.~$T_\ell$) is the final matrix (resp.~initial matrix, resp.~transition matrix) appearing in the corresponding entry of Table \ref{Table:Geometric_genes} according to the type of $i$ (resp.~$o$).
\end{rmk}
\begin{table}[hbt]
\captionsetup{justification=centering}
\caption[Foo content]{\textbf{Equations for $\cZ^{\nv,\tau}(\tld{z})$}
}
\label{Table:Geometric_genes}
\centering
\adjustbox{max width=\textwidth}{
\begin{tabular}{| c || c | c | c | }
\hline
\hline
\backslashbox{type}{$\tld{w}_j$}&$t_\eta$&$w_0t_\eta$&$t_{w_0(\eta)}$\\
\hline
\hline
&&&\\
$II$&
$\begin{aligned}
R_{j}&=\cO[X_j,Y_j]
\\
&\\
\textnormal{Initial matrix:}&\begin{pmatrix}1\\-X_j\end{pmatrix}\\
\textnormal{Final matrix:}&\begin{pmatrix}Y_j&-p\end{pmatrix}
\end{aligned}$
&
$\begin{aligned}
R_{j}&=\cO[X_j,Y_j]
\\
&\\
\textnormal{Initial matrix:}&\begin{pmatrix}1\\-X_j\end{pmatrix}\\
\textnormal{Final matrix:}&\begin{pmatrix}-p&Y_j\end{pmatrix}
\end{aligned}$
&
$\begin{aligned}
R_{j}&=\cO[X_j,Y_j]
\\
&\\
\textnormal{Initial matrix:}&\begin{pmatrix}1\\-X_j\end{pmatrix}\\
\textnormal{Final matrix:}&\begin{pmatrix}0&1\end{pmatrix}
\end{aligned}$
\\
&&&\\
\hline
&&&\\
$I$&
$\begin{aligned}
R_j&=\frac{\cO[X_j,Y_j,Z_j]}{((p-Y_j)Y_j-X_jZ_j)}
\\
&\\
\textnormal{Initial matrix:}&\begin{pmatrix}Y_j&-X_j\\-Z_j&p-Y_j\end{pmatrix}\\
\textnormal{Final matrix:}&\begin{pmatrix}p-Y_j&-Z_j\\-X_j&Y_j\end{pmatrix}
\end{aligned}$
&
$\begin{aligned}
R_j&=\frac{\cO[X_j,Y_j,Z_j]}{((p-Y_j)Y_j-X_jZ_j)}
\\
&\\
\textnormal{Initial matrix:}&\begin{pmatrix}Y_j&-X_j\\-Z_j&p-Y_j\end{pmatrix}\\
\textnormal{Final matrix:}&\begin{pmatrix}Y_j&-X_j\\-Z_j&p-Y_j\end{pmatrix}
\end{aligned}$
&
$\begin{aligned}
R_j&=\frac{\cO[X_j,Y_j,Z_j]}{((p-Y_j)Y_j-X_jZ_j)}
\\
&\\
\textnormal{Initial matrix:}&\begin{pmatrix}Y_j&-X_j\\-Z_j&p-Y_j\end{pmatrix}\\
\textnormal{Final matrix:}&\begin{pmatrix}p-Y_j&X_j\\Z_j&Y_j\end{pmatrix}
\end{aligned}$
\\
&&&\\
\hline
&&&\\
$0$&
$\begin{aligned}
R_j&=\cO[X_j]
\\
&\\
\textnormal{Transition Matrix:}&\begin{pmatrix}
1&0\\-X_j&p
\end{pmatrix}
\end{aligned}$
&
$\begin{aligned}
R_j&=\cO[X_j]
\\
&\\
\textnormal{Transition Matrix:}&\begin{pmatrix}
0&-1\\-p&X_j
\end{pmatrix}
\end{aligned}$
&
$\begin{aligned}
R_j&=\cO[X_j]
\\
&\\
\textnormal{Transition Matrix:}&\begin{pmatrix}
p&-X_j\\0&1
\end{pmatrix}
\end{aligned}$
\\
&&&\\
\hline
\hline
\end{tabular}}
\captionsetup{justification=raggedright,
singlelinecheck=false
}
\end{table}

\subsubsection{Rational smoothness}\label{subsect:rational smoothness}
\begin{thm} \label{thm:rational smoothness}
\begin{enumerate}
\item
Let $\tld{\cZ}^{\tmod,\mathrm{nm}}$ be the normalization of $\tld{\cZ}^{\tmod,\tau}$. Then $\tld{\cZ}^{\tmod,\mathrm{nm}}$ is resolution-rational (cf.~\cite[Definition 9.1]{Kovacs}) and $\tld{\cZ}^{\tmod,\mathrm{nm}}$ is Gorenstein.
\item
Assume $p> 2+\max_j\langle \mu_j,\alpha^\vee\rangle$. Then the same statements hold for the normalization $\tld{\cZ}^{\mathrm{nm}}$ of $\tld{\cZ}^{\tau}$.
\end{enumerate}
\end{thm}

\begin{proof}
\begin{enumerate}
\item
The statement is local so it suffices to check it for $\tld{\cZ}^{\tmod,\mathrm{nm}}(\tld{z})$. 

Since
\begin{itemize}
\item $\tld{\cZ}^{\tmod,\tau}(\tld{z})$ is the scheme theoretic image of the proper map $\pi^{\tmod}:\tld{Y}^{\tmod,\eta,\tau}(\tld{z})\to \tld{U}(\tld{z})$,
\item $\pi^{\tmod}$ is a closed immersion after inverting $p$ (cf. Proposition \ref{prop:naive vs sat model}),
\item $\tld{Y}^{\tmod,\eta,\tau}(\tld{z})$ is normal,
\end{itemize}
it follows that $\pi^{\tmod}_*\cO_{\tld{Y}^{\tmod,\eta,\tau}(\tld{z})}=\cO_{\tld{\cZ}^{\tmod,\mathrm{nm}}(\tld{z})}$. But Proposition \ref{prop:R1:null} shows that in fact
\[R\pi^{\tmod}_*\cO_{\tld{Y}^{\tmod,\eta,\tau}(\tld{z})}=\cO_{\tld{\cZ}^{\tmod,\mathrm{nm}}(\tld{z})}\]
By Grothendieck duality and the properness of $\pi^{\tmod}$ we have
\begin{equation}
\label{eq:groth:dual}
\omega_{\tld{\cZ}^{\tmod,\mathrm{nm}}(\tld{z})/\cO}=R\pi^{\tmod}_*\omega_{\tld{Y}^{\tmod,\eta,\tau}(\tld{z})/\cO}
\end{equation}
and by Proposition \ref{prop:dualizing complex model} and Proposition \ref{prop:R1:null}, the RHS of \eqref{eq:groth:dual} is
\[R\pi^{\tmod}_*\cO_{\tld{Y}^{\tmod,\eta,\tau}(\tld{z})/\cO}=\cO_{\tld{\cZ}^{\tmod,\mathrm{nm}}(\tld{z})}.\]
Thus we learn that $\tld{\cZ}^{\tmod,\mathrm{nm}}(\tld{z})$ is Gorenstein. Since $\tld{Y}^{\tmod,\eta,\tau}(\tld{z})$ is easily seen to be resolution-rational (it is locally isomorphic to a product of $\Spec \cO[X,Y]/(XY-p)$), $\tld{\cZ}^{\tmod,\mathrm{nm}}(\tld{z})$ is also resolution-rational.
\item Let $\tld{\cZ}^{\mathrm{nm}}(\tld{z})$ be the normalization of $\tld{\cZ}^{\tau}$. The same argument as above shows that 
\[\pi_*\cO_{\tld{Y}^{\eta,\tau}(\tld{z})}=\cO_{\tld{\cZ}^{\mathrm{nm}}(\tld{z})}\]
By Lemma \ref{lem:strightening}
\[R\pi_*\cO_{\tld{Y}^{\eta,\tau}(\tld{z})}\otimes^{\bL}_{\cO}\F=R\pi_*(\cO_{\tld{Y}^{\eta,\tau}(\tld{z})}/\varpi)=R\pi^{\tmod}_*(\cO_{\tld{Y}^{\tmod,\eta,\tau}(\tld{z})}/\varpi)=R\pi^{\tmod}_*\cO_{\tld{Y}^{\tmod,\eta,\tau}(\tld{z})}\otimes^{\bL}_{\cO}\F\]
\end{enumerate}
which concentrates in degree $0$ by Proposition \ref{prop:R1:null}. Since $\pi$ is proper, it follows that 
\begin{equation}
\label{eq:proper pushforward 1}
R\pi_*\cO_{\tld{Y}^{\eta,\tau}(\tld{z})}=\cO_{\tld{\cZ}^{\mathrm{nm}}(\tld{z})}
\end{equation}
\begin{equation}
\label{eq:proper pushforward 2}
R\pi_*(\cO_{\tld{Y}^{\eta,\tau}(\tld{z})}/\varpi)=\cO_{\tld{\cZ}^{\mathrm{nm}}(\tld{z})}/\varpi=\cO_{\tld{\cZ}^{\tmod,\mathrm{nm}}(\tld{z})}/\varpi
\end{equation}
Applying Grothendieck duality to \eqref{eq:proper pushforward 2} as in the previous part, we conclude that $\omega_{(\tld{\cZ}^{\mathrm{nm}}(\tld{z})/\varpi)/\F}=\cO_{\tld{\cZ}^{\mathrm{nm}}(\tld{z})}/\varpi$ is trivial.
Since $\tld{\cZ}^{\mathrm{nm}}(\tld{z})$ is an affine $p$-adic formal scheme, we also get $\omega_{\tld{\cZ}^{\mathrm{nm}}(\tld{z})/\cO}=\cO_{\tld{\cZ}^{\mathrm{nm}}(\tld{z})}$ so $\tld{\cZ}^{\mathrm{nm}}(\tld{z})$ is Gorenstein.
This fact, \eqref{eq:proper pushforward 1} and the fact that $\tld{Y}^{\eta,\tau}$ is resolution-rational now implies $\tld{\cZ}^{\mathrm{nm}}(\tld{z})$ is resolution rational.
\end{proof}
We can also completely classify the non-normal locus of $\tld{\cZ}^{\tau}$:
\begin{thm}\label{thm:non-normal locus}
Suppose $p> 2+\max_j\langle \mu_j,\alpha^\vee\rangle$. 
Let $\tau$ have small presentation $(s,\mu)$. A Galois representation $\rhobar$ gives a non-normal point of $\tld{\cZ}^{\tau}$ exactly when both of the following holds:
\begin{enumerate}
\item For each $j$, $s_j=\Id$ and $\langle \mu_j,\alpha^\vee\rangle\in \{0,1\}$.
\item $\rhobar$ is Fontaine-Laffaille with inertial Hodge-Tate weights at embedding $j\in\cJ$ given by 
\begin{itemize}
\item $(1,0)$ if $\langle \mu_{j},\alpha^\vee\rangle=0$.
\item $(1,1)$ if $\langle \mu_{j},\alpha^\vee\rangle=1$.
\end{itemize}
\end{enumerate}
\end{thm}
\begin{proof} It follows from the proof of Theorem \ref{thm:rational smoothness} that the non-normal locus is exactly the support of 
\[\coker(\cO_{\tld{Z}^{\tau}}\to \pi_*\cO_{\tld{Y}^{\eta,\tau}})\]
and when intersecting with $\tld{U}(\tld{z})$, it equals the support of 
\[\coker(\cO_{\tld{U}(\tld{z})}/\varpi\to \pi_*\cO_{\tld{Y}^{\eta,\tau}}/\varpi)=\coker(\cO_{\tld{U}(\tld{z})}/\varpi\to \pi^{\tmod}_*\cO_{\tld{Y}^{\tmod,\eta,\tau}}/\varpi)\]
Thus we reduce to investigating the non-normal locus of $\tld{Z}^{\tmod,\tau}(\tld{z})$.
It follows from the proof of Proposition \ref{prop:coker} and Theorem \ref{thm:rational smoothness} that the non-normal locus is exactly the support of $R^1\textnormal{pr}_{\tld{B}*}\cI(\tld{z})$ in Corollary \ref{cor:Delta}. But Corollary \ref{cor:vanish} shows this support is non-empty exactly when for each $j$, either $\langle \mu_j,\alpha^\vee \rangle =0$ or $(\langle \mu_j,\alpha^\vee \rangle,s_j,\tld{w}_j)=(1,\Id,t_{w_0(\eta)})$. This shows that $(s,\mu)$ has the desired form.

We now suppose $(s,\mu)$ has the requisite form. Decompose $\cJ=\cJ_0\coprod \cJ_1$ where $\cJ_0=\{j | \langle \mu_j,\alpha^\vee \rangle =0 \}, \cJ_1=\{j | \langle \mu_j,\alpha^\vee \rangle =1 \}$.
Furthermore, if $(s,\mu)$ is of the right form, the proof of Lemma \ref{lem:coh:comp} furthermore identifies the support of $R^1\textnormal{pr}_{\tld{B}*}\cI(\tld{z})$ as the locus of tuples $(A_j)\in \tld{U}(\tld{z})(\F)$ such that 
\begin{itemize} 
\item $A_j=g_j\begin{pmatrix} v & 0 \\ 0& v\end{pmatrix}$ if $j\in \cJ_1$.
\item $A_j\in \{g_j\begin{pmatrix} v & 0 \\ 0& 1\end{pmatrix}\begin{pmatrix} 1 & 0 \\ C& 1\end{pmatrix},g_j\begin{pmatrix} 1 & 0 \\ 0& v\end{pmatrix}\begin{pmatrix} 1 & B \\ 0& 1\end{pmatrix}\}$ if $j\in\cJ_0$;
\end{itemize} 
for some $g_j\in \GL_2(\F)$ (and $B, C\in \F$). After modifying $(A_j)$ by the shifted conjugation action of $\GL_2^{\cJ}$, we can arrange so that $A_j\in \GL_2(\F)\begin{pmatrix} v & 0 \\ 0 &1 \end{pmatrix}$ for $j\in \cJ_0$.

Now set $\lambda\in X_*(T)^\vee$ be such that $\lambda_{j}=(1,0)$ for $j\in\cJ_0$ and $\lambda_{j}=(1,1)$ for $j\in \cJ_1$.
Analogous to \cite[Proposition 2.2.6]{LGC}, the moduli space $\mathrm{FL}_\lambda$ of mod $p$ Fontaine-Laffaille modules of weight $\lambda$ has the following description:
Let $H\cong \GL_2^{|\cJ_1|}\times B^{|\cJ_0|}\subset \GL_2^\cJ$ be the subgroup of tuples $(X_j)$ such that $X_{j-1}\in \GL_2$ if $j\in \cJ_1$ and $X_{j-1}\in B$ if $j\in \cJ_0$. Then $\mathrm{FL}_\lambda$ is the quotient $[\GL_2^\cJ/H]$ given by the action
\[(X_j)\cdot(g_j)\mapsto (X_jg_{j}\ovl{X}_{j-1}^{-1})\]
where the overline denotes the projection $B\onto T$ when $j\in \cJ_0$, and is the identity map when $j\in \cJ_1$.

But then, analogous to \cite[Proposition 8.2.4]{LGC}, the natural embedding 
$\mathrm{FL}_\lambda\into \Phi\text{-}\Mod^{\text{\'et},2}_{K}$ identifies with
\[(g_j)\mapsto \bigg((g_j\begin{pmatrix} v & 0 \\ 0 &v \end{pmatrix})_{j\in\cJ_1},(g_j\begin{pmatrix} v & 0 \\ 0 &1 \end{pmatrix})_{j\in\cJ_0}\bigg).\]
\end{proof}
\begin{rmk} By the definition of small presentation, Theorems \ref{thm:rational smoothness}, \ref{thm:non-normal locus} hold whenever $p>5$. However when $K=\mathbb{Q}_p$, the only non-trivial case of these theorems are when $\langle \mu,\alpha^{\vee}\rangle\leq 1$, and thus 
and we only need to impose $p>3$ for those cases.
\end{rmk}

\section{Equations of the deformation ring}\label{sec:equa}

In this section we apply the main results of Sections \ref{sec:models} and \ref{sec:geometry} (namely Theorems \ref{thm:main:model} and \ref{thm:main:Gal:def}) to prove the conjectures of the series of papers \cite{CDM1,CDM2,CDM3}, in particular that tamely potentially crystalline Barsotti--Tate deformation rings $R_{\rhobar}^{\tau}$ only depend on the combinatorial gene $\bX(\tau,\rhobar|_{I_K})$(Theorem \ref{thm:indepen}).
Throughout this section, except for \S \ref{sec:examples}, we assume that $K\neq \Qp$.
\\   

\subsection{Genetics}

We recall and formalize into an abstract setup the notion of genes as introduced in \cite{CDM2,CDM3}, and recall the main conjectures of \emph{loc.~cit}.

\subsubsection{Combinatorial genes}
Inspired by the terminology of \cite{CDM2,CDM3} we now define the notion of \emph{combinatorial gene} associated to a pair $(\gamma,h)\in\Z/(p^{f}-1)\times \Z/(p^{2f}-1)$.

Let $\gamma\in\Z/(p^{f}-1)$ and $h\in \Z/(p^{2f}-1)$ such that $h\not\cong 0$ modulo $q+1$ and $h-2\gamma-(\sum_{j=0}^{f-1}p^j)\not \cong 0$ modulo $p^f-1$.
Consider the $p$-expansions 
\begin{equation}
\label{eq:vi}
h - (p^f{+}1)\bigg(h-\gamma-\sum_{j=0}^{f-1}p^j\bigg) \equiv 
p^{2f-1} v_0 + p^{2f-2} v_1 + \cdots + p v_{2f-2} + v_{2f-1}
\pmod{p^{2f} - 1}
\end{equation}
with $v_{j'}\in \{0,\dots,p-1\}$ for all $j'\in\cJ'$.
A \emph{combinatorial gene associated to $(\gamma,h)$} is a $\cJ'$-tuple $\bX = \bX(\gamma,h)\in \{\gA, \gB, \gAB, \gO\}^{\cJ'}$ which satisfies the following properties (see \cite[Lemma B.1.3]{CDM3}):
\begin{enumerate}
\item 
\label{en:gene:1}
if $v_{j'} = 0$ and $\bX_{j'+1} = \gO$, then $\bX_{j'} = \gAB$;
\item if $v_{j'} = 0$ and $\bX_{j'+1} \neq \gO$, then $\bX_{j'} = \gA$;
\label{en:gene:3}
\item if $v_{j'} = 1$ and $\bX_{j'+1} = \gO$, then $\bX_{j'} = \gO$;
\label{en:gene:4}
\item if $v_{j'} = 1$ and $\bX_{j'+1} \neq \gO$, then $\bX_{j'} = \gB$;
\item 
\label{en:gene:5}
if $v_{j'} \geq 2$, then $\bX_{j'} = \gO$.
\end{enumerate}
By \cite[Lemma~1.3.3, Lemma B.1.7]{CDM3} a combinatorial gene $\bX$ associated to $(\gamma,h)$ is well defined, and is unique by the proof of \cite[Proposition 1.4.4]{CDM3}.
Moreover, by \cite[Proposition 1.3.2, Corollary 1.3.4]{CDM3}, the $\cJ'$-tuple $\bX=\bX(\gamma,h)$  satisfies the following conditions 
\begin{enumerate}[start=1,label={$\clubsuit$\arabic*}]
\item
\label{en:abstract:gene:1}
if $\bX_{j'+1} = \gO$, then $\bX_{j'}\in\{ \gAB,\gO\}$;
\item 
\label{en:abstract:gene:2}
if $\bX_{j'+1} \neq \gO$, then $\bX_{j'}\in\{ \gA,\gB,\gO\}$,
\item 
\label{en:abstract:gene:3}
there exists an integer $j'\in\cJ'$ such that $\bX_{j'}=\gO$ or $\bX_{j'}\not=\bX_{j'+1}$.
\end{enumerate}

The discussion after \cite[Lemme 2.1.7]{CDM2} shows that:
\begin{lemma}
\label{lem:eq:rel:gene}
Assume that $\gamma+\gamma'+(\sum_{j=0}^{f-1}p^j)\equiv h$ modulo $p^f-1$, $h\not\equiv 0\mod p^f+1$. Then:
\begin{equation*} 
\bX(\gamma',h)_{j'}=
\begin{cases}
\gA & \text{ if } \bX(\gamma,h)_{j'}=\gB,\\
\gB & \text{ if } \bX(\gamma,h)_{j'}=\gA,\\
\bX(\gamma,h)_{j'} & \text{ otherwise, }\\
\end{cases}
\end{equation*}
and 
\[
\bX(\gamma,p^fh)_{j'}=\bX(\gamma,h)_{j'+f}.
\]

\end{lemma}

Following the conditions \eqref{en:abstract:gene:1}--\eqref{en:abstract:gene:3} and Lemma \ref{lem:eq:rel:gene} we define an \emph{abstract combinatorial gene} as follows:
\begin{defn}
\label{defn:abs:comb:gene}
An \emph{abstract combinatorial gene} is an equivalence class of a $\cJ'$-tuple $\bX \in \{\gA, \gB, \gAB, \gO\}^{\cJ'}$ satisfying conditions \eqref{en:abstract:gene:1}, \eqref{en:abstract:gene:2},\eqref{en:abstract:gene:3}, by the equivalence relation generated by 
\begin{align}
\label{eq:rel:gene:1}
\bX'\sim\bX &&\text{ if }&\begin{cases}
\bX'_{j'}=\{\gA,\gB\}\setminus\{\bX_{j'}\}&\text{ when } \bX_{j'}\in\{\gA,\gB\}\\
\bX'_{j'}=\bX_{j'}&\text{otherwise}.
\end{cases}
\\
&\nonumber\\
\label{eq:rel:gene:2}
\bX'\sim\bX &&\text{ if }&\begin{cases}
\bX'_{j'}&=\bX_{j'+1}\\
\bX'_{j'+f}&=\bX_{j'+f+1}
\end{cases}
\end{align}
\end{defn}
Note that the relation \eqref{eq:rel:gene:2} implies that $(\bX_{j'})_{j'\in\cJ'}\sim (\bX_{j'+f})_{j'\in\cJ'}$, and that the notion of abstract combinatorial gene is independent of $p$.

Let $(\tau,\taubar')$ be a pair of tame inertial types satisfying the following determinant condition:
\begin{equation}
\tag{det}\label{eq:det}
\det(\tau)\otimes_{\cO}\F\equiv \det(\taubar')\otimes_{\F}\omega.
\end{equation} 
If $\tau=\tau(s,\mu)=\omega_f^{\gamma}\oplus \omega_f^{\gamma'}$ and $\taubar'=\taubar(\sigma,\nu)=\omega_{2f}^h\oplus \omega_{2f}^{p^fh}$ are a tame inertial type of niveau $f$ and a tame inertial $\F$-type of level $2f$ as in \S \ref{subsub:TIT}, then condition \eqref{eq:det} translates into the condition of Lemma \ref{lem:eq:rel:gene}.
We thus define the gene $\bX(\tau,\taubar')$ of the pair $(\tau,\taubar')$ as the abstract combinatorial gene associated to the $\cJ'$-tuple $\bX(\gamma,h)$.
This gives the motivation behind Definition \ref{defn:abs:comb:gene}: the relation \eqref{eq:rel:gene:1} is imposed by the isomorphism $\omega_f^{\gamma}\oplus \omega_f^{\gamma'}\cong \tau\cong \omega_f^{\gamma'}\oplus \omega_f^{\gamma}$ and the relation \eqref{eq:rel:gene:2} by the fact that the isomorphism class of $\tau$, $\taubar'$ does not depend on the choice of the embedding $\sigma_0$.
In particular, $\bX(\tau,\taubar')$ only depends on the isomorphism class of $\tau$ and $\taubar'$, and is insensitive to twist by characters $\chi:I_K\ra\cO^\times$.

\subsubsection{Genetic conjectures}
The computations of \cite{CDM2,CDM3} showed that the combinatorial genes contains non-trivial information on the generic and special fiber of Galois deformation rings with $p$-adic Hodge theory conditions.

Let $\rhobar:G_K\ra\GL_2(\F)$ be irreducible and $\tau$ a tame inertial type of niveau $f$.
Then $\rhobar|_{I_K}$ defines a tame inertial $\F$-type, so that the combinatorial gene $\bX(\tau,\rhobar|_{I_K})$ is defined.
The authors of \cite{CDM2,CDM3} propose the following conjecture (which is an integral version of  \cite[Conjecture 5.1.5]{CDM2})
\begin{conj''}[Conjecture 2 in \cite{CDM3}]
\label{conj'':geneR}
The deformation ring $R_{\rhobar}^{\eta,\tau}$ is determined by $\bX(\tau,\rhobar|_{I_K})$.
\end{conj''}

They furthermore refine Conjecture \ref{conj'':geneR} into the following
\begin{conj''}[Conjecture 5.2.7 \cite{CDM2}, Conjecture 3.1.2 \cite{CDM3}]\label{conj:decomp}
There exists a decomposition $\bX(\tau,\rhobar|_{I_K})=\cup_{i=0}^{r}(\bX_{j_i},\bX_{j_{i+f}})_{j_i\leq j\leq j_{i+1}}$ such that
\[
R_{\rhobar}^{\eta,\tau}\cong \widehat{\otimes}_{i=0}R_{i}
\]
where $R_i$ is a complete local Noetherian $\cO$ algebra depending only on $(\bX_{j_i},\bX_{j_{i+f}})_{j_i\leq j\leq j_{i+1}}$.
\end{conj''}
Moreover, even if not stated as a conjecture, they suggest that
\begin{conj''}[\cite{CDM3,CDM4}]
\label{conj'':geneR:indep:p}
The deformation ring $R_{\rhobar}^{\eta,\tau}$ is ``independent of $p$''.
\end{conj''}
Conjectures \ref{conj'':geneR}, \ref{conj:decomp} and \ref{conj'':geneR:indep:p} are proven in Theorem \ref{thm:indepen} for  $p>8f+3+\max_j\langle\mu_j,\alpha^\vee\rangle$.
In particular throughout sections \ref{subsec:gen:tra}, \ref{subsec:fibers} and \ref{sub:Naive:eq} we assume that $p>8f+3+\max_j\langle\mu_j,\alpha^\vee\rangle$.

Prior to this work, Conjecture \ref{conj'':geneR} was known when either $\tau$ has a presentation $(s,\mu)$ with $2<\langle\mu_j,\alpha^\vee\rangle <p-2$ for all $j\in\cJ$ or if $\binom{\bX(\tau,\rhobar|_{I_K})_{j+f}}{\bX(\tau,\rhobar|_{I_K})_j}=\binom{\gO}{\gO}$ for some $j\in\cJ$.
In the first case, we have $\gO\in\{\bX_{j},\bX_{j+f}\}$ for all $j\in\cJ$ which in turn implies that implies that the Kisin variety of type $(\eta,\tau)$ attached to $\rhobar$ is either empty or a single point by \cite[Th\'eor\`eme 2.2.1(A)]{CDM2}. Conjecture \ref{conj'':geneR} is true by Theorem \ref{thm:generic_cases} in this case follow from Theorem \ref{thm:generic_cases} and \cite[Theorem 3]{CDM3}.
In the second case, the Kisin variety is empty and the deformation ring is zero.

We now elaborate on Conjecture \ref{conj'':geneR:indep:p}.
Given an abstract combinatorial gene $\bX$ we prove the conjecture by constructing rings $R_{\bX}$ which are (possibly zero) quotients of polynomial rings over $\Z[t]$ modulo an ideal $I_{\bX}\subset R_{\bX}$.
These rings are independent of $p$ (since abstract combinatorial genes are) and the Conjecture \ref{conj'':geneR:indep:p} would be proven by explicitly showing that $R_{\rhobar}^{\eta,\tau}$ is isomorphic to the completion of $R_{\bX}/(t-p)\otimes \cO$ at the ideal generated by $t$ and the variables of $R_{\bX}$, where $\bX=\bX(\tau,\rhobar|_{I_K})$.
Again, by Theorem \ref{thm:generic_cases}, the conjecture is known to be true when $\tau$ has a $2$-generic lowest alcove presentation.

We conclude with the following observation.
By \cite[Proposition 4.1.3]{CDM2}, we have $R_{\rhobar}^{\eta,\tau}=0$ as soon as $(\bX(\tau,\rhobar|_{I_K})_{j+f},\bX(\tau,\rhobar|_{I_K})_{j})=(\gO,\gO)$ for some $j\in\cJ$.
Hence, in what follows, we will be interested in combinatorial genes $\bX$ satisfying the further condition
\begin{enumerate}[start=4,label={$\clubsuit$\arabic*}]
\item 
\label{en:abstract:gene:extra}
$(\bX_{j+f},\bX_j)\neq (\gO,\gO)$ for all $j\in\cJ$.
\end{enumerate}

\subsection{Genetic translation}\label{subsec:gen:tra}

We fix our setup as in Theorem \ref{thm:main:Gal:def}.
Hence, let $\tau:I_K\ra \GL_2(\cO)$ be a regular tame inertial type of niveau $f$ with a small presentation $(s,\mu)$ and let $\tld{w}=(w_jt_{\nu_j})_{j\in\cJ}\in \Adm^\vee(\eta)$. 
Up to twist, we can furthermore assume that $\mu_{j,2}=0$ for all $j\in \cJ$.
We abbreviate $k_j\defeq \langle \mu_j,\alpha^\vee\rangle$ and set $\tld{z}\defeq \tld{w}s^{-1}v^{\mu}$ in what follows.

In order to analyze genetic data associated to $\tau$ and $\tld{w}$ define $\lambda\in X^*(\un{T})$ by the condition
\begin{equation}
\label{eq:translation:first:step}
(z_0z_{f-1}\cdots z_1)v^{\left(\sum_{j\in\cJ}p^j\lambda_{j}\right)}=\tld{z}_0\phz(\tld{z}_{f-1})\cdots\phz^{f-1}(\tld{z}_{1}).
\end{equation}
and the $2f$-tuple $(v'_{j'})_{j'\in\cJ'}$ by  
$$
\begin{pmatrix}
v'_{2f-1-j}\\v'_{f-1-j}
\end{pmatrix}\defeq \lambda_{j}-\begin{pmatrix}
k_{f-j}\\k_{f-j}
\end{pmatrix}\delta_{s_{\orient,f-1-j}\neq\Id}
$$
where $j\in\{0,\dots,f-1\}$ and $s_{\orient,f-1-j}=\prod_{i=0}^{f-1-j}s_i$ (Lemma \ref{lem:s_or}).
By smallness of $(s,\mu)$, we have $v'_j, v'_{f+j}\in\{-\frac{p+1}{2},\ldots, \frac{p+3}{2}\}$ for all $j\in\{0,\dots,f-1\}$.
Note that the $2f$-tuple $(v'_{j'})_{j'\in\cJ'}$ depends on the triple $(\tld{w},s,\mu)$, but we omit this dependence for sake of readability.

We assume from now on that $\prod_{j\in\cJ}z_j=w_0$, $\sum_{j\in\cJ}p^j(\mu_{j,1}+\mu_{j,2})\equiv\sum_{j\in\cJ}p^j(\lambda_{j,1}+1+(\lambda_{j,2}+1))$ modulo $p^f-1$ and $\sum_{j\in\cJ}p^j(\lambda_{j,1}+p^f\lambda_{j,2})\not\equiv 0$ modulo $p^f+1$.
By construction, the $2f$-tuple $(v'_{j'})_{j'\in\cJ'}$ extracts precisely the LHS of equation \eqref{eq:vi} (using equations \eqref{eq:gamma} and \eqref{eq:h}), and hence produces a gene $\bX(v')$, satisfying items \eqref{en:abstract:gene:1}, \eqref{en:abstract:gene:2}, \eqref{en:abstract:gene:3} thanks to the assumptions in the previous sentence.
Tables \ref{TableGen1:new} and \ref{TableGen2:new} give respectively the explicit description of $\tld{z}_i$ and $\begin{pmatrix}
v'_{f+j}\\v'_{j}
\end{pmatrix}$ according to $(s,\mu)$ and $\tld{w}$ and are directly obtained from the definitions. 

\begin{table}[hbt]
\small
\captionsetup{justification=centering}
\caption[Foo content]{\textbf{Genetic Translation-I}
}
\label{TableGen1:new}
\centering
\adjustbox{max width=\textwidth}{
\begin{tabular}{| c || c | c | c |}
\hline
\hline
\backslashbox{$s_i$}{$\tld{w}_i$}&$t_\eta$&$w_0t_{\eta}$&$t_{w_0(\eta)}$ \\
\hline
&&&\\
$\Id$&$\begin{pmatrix}k_i+1\\0\end{pmatrix}$&$(12)\begin{pmatrix}k_i+1\\0\end{pmatrix}$&$\begin{pmatrix}k_i\\1\end{pmatrix}$
\\
&&&\\
\hline
&&&\\
$(12)$&$(12)\begin{pmatrix}k_i\\1\end{pmatrix}$&$\begin{pmatrix}k_i\\1\end{pmatrix}$&$(12)\begin{pmatrix}k_i+1\\0\end{pmatrix}$
\\
&&&\\
\hline
\hline
\end{tabular}}
\captionsetup{justification=raggedright,
singlelinecheck=false
}
\caption*{
This table records $\tld{z}_i\defeq\tld{w}_is_{i}^{-1}t_{\mu_i}$.
}
\end{table}

\begin{table}[hbt]
\small
\captionsetup{justification=centering}
\caption[Foo content]{\textbf{Genetic Translation-II
}
}
\label{TableGen2:new}
\centering
\adjustbox{max width=\textwidth}{
\begin{tabular}{| c || c || c | c | }
\hline
\hline
$\tld{w}_{j+1}$&\backslashbox{$s_{j+1}$}{$s_{\orient,j}$}&$\Id$&$(12)$\\
\hline
\hline
&&&\\
\multirow{2}{*}{$t_{\eta},\quad w_0t_\eta$}&$\Id$&$\Sigma_{j}\begin{pmatrix}k_{j+1}+1\\0\end{pmatrix}$&$\Sigma_{j}\begin{pmatrix}1\\-k_{j+1}\end{pmatrix}$\\
&&&\\
\cline{2-4}
&&&\\
&$(12)$&$\Sigma_{j}\begin{pmatrix}k_{j+1}\\1\end{pmatrix}$&$\Sigma_{j}\begin{pmatrix}0\\1-k_{j+1}\end{pmatrix}$\\
&&&\\
\hline
\hline
&&&\\
\multirow{2}{*}{$t_{w_0\eta}$}&$\Id$&$\Sigma_{j}\begin{pmatrix}k_{j+1}\\1\end{pmatrix}$&$\Sigma_{j}\begin{pmatrix}0\\1-k_{j+1}\end{pmatrix}$\\
&&&\\
\cline{2-4}
&&&\\
&$(12)$&$\Sigma_{j}\begin{pmatrix}k_{j+1}+1\\0\end{pmatrix}$&$\Sigma_{j}\begin{pmatrix}1\\-k_{j+1}\end{pmatrix}$\\
&&&\\
\hline
\hline
\end{tabular}}
\captionsetup{justification=raggedright,
singlelinecheck=false
}
\caption*{
The entries of the table record $\begin{pmatrix}v'_{j+f}\\v'_j\end{pmatrix}$ for $j\in\{0,\dots,f-1\}$, where $\Sigma_{j}\defeq \prod_{i=1}^{j}z^{-1}_i=(\prod_{i=1}^{j}s_i)(w_0)^{N_j}$ and $N_j\defeq\#\left\{i\in\{1,\dots,j\}, \tld{w}_i=w_0t_{\eta}\right\}$.
Note that $\Sigma_0=\Id$ by definition.}
\end{table}.

\subsection{Fibers of the map $(\tld{w},s,\mu)\mapsto \bX(v')$}
\label{subsec:fibers}

Given a triple $(\tld{w},s,\mu)$ we have produced a tuple $\big((s_{j+1}, s_{\orient,j},\Sigma_j,z_{j+1},\tld{w}_j, \textnormal{type at $j$}),(v'_{f+j},v'_j)\big)_{j\in\cJ}$ and a gene $\bX(v')$ attached to $(v'_{f+j},v'_j)_{j\in\cJ}$, hence a map
\begin{equation}
\label{eq:map:triple_to_gene}
(\tld{w},s,\mu)\mapsto
\big((s_{j+1}, s_{\orient,j},\Sigma_j,z_{j+1},\tld{w}_j, \textnormal{type at $j$}),(v'_{f+j},v'_j)\big)_{j\in\cJ}\mapsto (v'_{f+j},v'_j)_{j\in\cJ}\mapsto \bX(v').
\end{equation}
In this and the following section we analyze the shapes and types appearing in the fiber of the map \eqref{eq:map:triple_to_gene}.
\emph{From now onwards, we furthermore assume that $\bX(v')$ satisfies condition \eqref{en:abstract:gene:extra}.}

As a preliminary step, we record in Table \ref{TableGen13} the fiber of the map 
\begin{equation}
\label{eq:map:7tuple_to_gene}
\big((s_{j+1}, s_{\orient,j},\Sigma_j,z_{j+1},\tld{w}_j, \textnormal{type at $j$}),(v'_{f+j},v'_j)\big)_{j\in\cJ}\mapsto (v'_{f+j},v'_j)_{j\in\cJ}
\end{equation}
Table \ref{TableGen13} is obtained directly from Table \ref{TableGen2:new}, and the only relevant property to obtain the Table \ref{TableGen13} is whether $v'_{j}\geq 2$, $v'_{j}=1$, $v'_{j}=0$ or $v'_{j}<0$.

Recall that  we identify $\cJ'$ with $\Z/2f$, and in what follows objects such as $\Sigma_j$, $\bX_j$, $v'_j$ are indexed by $\cJ'$.
For objects that are previously indexed by $\cJ$ we extend by $f$ periodicity, with the exception of $\Sigma_j$ where the extension is given by $\Sigma_{j+f}=w_0\Sigma_j$.

\begin{table}[hbt]
\captionsetup{justification=centering}
\caption[Foo content]{\textbf{Fibre of the map \\$(s_{j+1}, s_{\orient,j},\Sigma_j,z_{j+1},\tld{w}_j, \textnormal{type}_j)\mapsto\left\{\begin{pmatrix}0\\\geq2\end{pmatrix},\begin{pmatrix}0\\1\end{pmatrix},\begin{pmatrix}0\\0\end{pmatrix},\begin{pmatrix}0\\<0\end{pmatrix}\right\}$}
}
\label{TableGen13}
\centering
\adjustbox{max width=\textwidth}{
\begin{tabular}{| c | c | c | c | c | c | c |}
\hline
\hline
&&&&&&\\
$s_{j+1}$&$s_{\orient,j}$&$\Sigma_j$&$\begin{pmatrix}v'_{j+f}\\v'_{j}\end{pmatrix}=\begin{pmatrix}0\\ \geq 2\end{pmatrix}$&$\begin{pmatrix}v'_{j+f}\\v'_{j}\end{pmatrix}=\begin{pmatrix}0\\ 1\end{pmatrix}$&$\begin{pmatrix}v'_{j+f}\\v'_{j}\end{pmatrix}=\begin{pmatrix}0\\ 0\end{pmatrix}$&$\begin{pmatrix}v'_{j+f}\\v'_{j}\end{pmatrix}=\begin{pmatrix}0\\ <0 \end{pmatrix}$
\\
&&&&&&\\
\hline
\hline
$\Id$&$\Id$&$\Id$& & $(\Id,t_{w_0(\eta)},0)$& &  \\
&&& && & \\
\hline
$\Id$&$\Id$&$w_0$& $(\Id,t_\eta,II)$&
$(\Id,t_\eta,0)$
& &  \\
&&& $(w_0,w_0t_\eta,II)$
&$(w_0,w_0t_\eta,0)$& &  \\
\hline
$\Id$&$w_0$&$\Id$&& $(\Id,t_{w_0(\eta)},0)$ &$(\Id,t_{w_0(\eta)},I)$&
$(\Id,t_{w_0(\eta)},II)$\\
&&& &&&  \\
\hline
$\Id$&$w_0$&$w_0$&& $(\Id,t_\eta,0)$&
$(\Id,t_{w_0(\eta)},I)$
&  \\
&&&& $(w_0,w_0t_\eta,0)$& &  \\
\hline
$w_0$&$\Id$&$\Id$& & & &  \\
&&& & & & \\
\hline
$w_0$&$\Id$&$w_0$&$(w_0,t_{w_0(\eta)},II)$ & & &  \\
&&&& & & \\
\hline
$w_0$&$w_0$&$\Id$&&& $(\Id,w_0t_{\eta},I)$&
$(\Id,w_0t_{\eta},II)$  \\
&&&&& $(w_0,t_{\eta},I)$&$(w_0,t_{\eta},II) $ \\
\hline
$w_0$&$w_0$&$w_0$&&& $(\Id,w_0t_{\eta},I)$&  \\
&&&&& $(w_0,t_{\eta},I)$& \\
\hline
\hline
\end{tabular}}\captionsetup{justification=raggedright,
singlelinecheck=false
}
\caption*{
Given $(v'_{f+j},v'_j)$ and $(s_{j+1},s_{\orient,j},\Sigma_j)$, the table entries record the triples $(z_{j+1},\tld{w}_j,\textnormal{type at $j$})$ such that $\big((s_{j+1}, s_{\orient,j},\Sigma_j,z_{j+1},\tld{w}_j, \textnormal{type at $j$}),(v'_{f+j},v'_j)\big)_{j\in\cJ}$ is a $7$-tuple associated to some  $(\tld{w},s,\mu)$.
(Here $j\in\Z/f\Z$.)
The table records the cases where $v'_{f+j}=0$, for $j\in\Z/f\Z$, but we note that we can replace $(v'_{f+j},v'_j)$ with $(v'_{j},v'_{f+j})$ at the cost of replacing $\Sigma_j$ with $w_0\Sigma_j$, and, similarly, we can replace $(v'_{j+f},v'_j)$ with $(1-v'_{j+f},1-v'_{j})$ at the cost of replacing $\Sigma_j$ and $s_{\orient,j}$ with $w_0\Sigma_j$ and  $w_0s_{\orient,j}$.
The blank boxes in the table correspond to configurations of $(s_{j+1},s_{\orient,j},\Sigma_j, (v'_{j+f},v'_j))$ which can not arise.
}
\end{table}

\subsubsection{Step one: fiber of the map $(v'_{f+j},v'_j)_{j\in\cJ}\mapsto \bX(v')$}\label{subsec:dec:gen}

Let $\bX\in\{\gA, \gB, \gAB, \gO\}^{\cJ'}$ satisfy conditions \eqref{en:abstract:gene:1},\eqref{en:abstract:gene:2},\eqref{en:abstract:gene:3}. In this subsection we determine the fiber above $\bX$ of the map $(v'_{f+j},v'_j)_{j\in\cJ}\mapsto \bX(v')$.

\begin{lemma}\label{lem:translation:v:v'}

Assume that $\bX=\bX(v')$ for some $(v'_{f+j},v'_j)_{j\in\cJ}$ associated to a triple $(\tld{w},s,\mu)$ as above
(in particular $\bX$ satisfies conditions \eqref{en:abstract:gene:1},\eqref{en:abstract:gene:2},\eqref{en:abstract:gene:3}).

Then, for each $j'\in\cJ'$, the values of $v'_{j'}$ are constraint by the third row of Table \ref{TableGen4} according to the pair $\bX_{f+j'},\bX_{f+j'+1}$.
\begin{table}[hbt]
\small
\captionsetup{justification=centering}
\caption[Foo content]{\textbf{Genetic Translation-II
}
}
\label{TableGen4}
\centering
\adjustbox{max width=\textwidth}{
\begin{tabular}{| c | c | c | c | c | c |}
\hline
\hline
&&&&&\\
$\bX_{f+j'}$&$\gA$&$\gB$&$\gAB$&$\gO$&$\gO$\\
&&&&&\\
\hline
&&&&&\\
$\bX_{f+j'+1}$&$\neq \gO$&$\neq \gO$&$\gO$&$\neq \gO$&$\gO$\\
&&&&&\\
\hline
&&&&&\\
$v'_{j'}$&$0$&$1$&$\{1,0\}$&$\{\geq 2,<0\}$&Any\\
&&&&&\\
\hline
\hline
\end{tabular}}
\end{table}

\end{lemma}
\begin{proof}
Let $(v_{j'})_{j'\in \cJ'}\in\{0,\ldots,p-1\}^{\cJ'}$ be the tuple defined by 
$$
\sum_{j'=0}^{2f-1}v_{j'}p^{2f-1-j'}\equiv \sum_{j'=0}^{2f-1}v'_{j'}p^{2f-1-j'}\mod p^{2f}-1.
$$
We will show that:
\begin{itemize}
\item
if $\bX_i\in\{\gA,\gB\}$ then $v_i=v'_i$;
\item 
if $\bX_i=\gAB$ then $v_i\in\{v_i',v_i'-1\}$;
\item 
if $\bX_i=\gO$  and $\bX_{i+1}\not=\gO$ then $v_i\in\{v_i',v_i'+p\}$; and
\item
if $\bX_i=\gO$  and $\bX_{i+1}=\gO$ then $v_i\in\{v_i', v_i'-1,v_i'+p,v_i'+p-1\}$.
\end{itemize}

By definition of combinatorial gene (equations \eqref{en:gene:1}-\eqref{en:gene:5}) we deduce that:
\begin{itemize}
\item
if $\bX_i\in\{\gAB,\gA\}$ then $v_i=0$;
\item
if $\bX_i=\gB$ then $v_i=1$;
\item
if $\bX_i=\bX_{i+1}=\gO$ then $v_i\in[1,p-1]$; and
\item
 if $\bX_i=\gO$  and $\bX_{i+1}\not=\gO$ then $v_i\in[2,p-1]$.
\end{itemize}
By (\ref{en:abstract:gene:3}) there exists $\ell\in\{0,\ldots, 2f-1\}$ such that $v_{\ell}\not=0$. 
Let $\varepsilon\in\Z$ be such that
$$\sum_{i=0}^{2f-1}v_ip^{2f-1-i}= \sum_{i=0}^{2f-1}v'_ip^{2f-1-i}+ \varepsilon(p^{2f}-1)$$
As $v_i'\in\{-(p+1)/2,\ldots, (p+3)/2\}$ we have
\[
-p^{2f}+1<\sum_{i=0}^{2f-1}v'_ip^{2f-1-i}< p^{2f}-1
\]
hence $\varepsilon\in\{0,1\}$.

Let $i_0\defeq\min\{i\geq 0, v'_i\not=0\}$ (this is well defined by (\ref{en:abstract:gene:3})). 
As 
\[
-p^{2f-i_0-1}<\sum_{j=i_0+1}^{2f-1}v'_jp^{2f-1-j}<p^{2f-i_0-1}
\] we have
\begin{equation}\label{equa:val:eps}
\varepsilon=\left\{\begin{array}{cc}0&\mbox{ if } v'_{i_0}>0,\cr
1& \mbox{ if } v'_{i_0}<0.\end{array}\right.\end{equation}
If $\varepsilon=1$ then $v_i=p-1$ for $0\leq i\leq i_0-1$ and moreover $v_{i_0}=p+v'_{i_0}\geq 2$ if $i_0<2f-1$ and $v_{i_0}=p-1+v'_{i_0}\geq 2$ if $i_0=2f-1$.
We conclude that $\bX_i=\gO$ for all $0\leq i\leq i_0$ if $\eps=1$.
\\
We have the following relation between $v_{2f-1}$ and $v_{2f-1}'$, and we define $\varepsilon_{2f-1}\in\{0,1\}$ as follows:
\begin{itemize}
\item
 if $v_{2f-1}'<0$ and $\varepsilon=0$ then $ v_{2f-1}=p+v_{2f-1}'\geq 2$, $\bX_{2f-1}=\gO$ and we define $\varepsilon_{2f-1}\defeq1$,
 \item
if $v_{2f-1}'\geq 0$ and $\varepsilon=0$ then $ v_{2f-1}=v_{2f-1}'$, and we define $\varepsilon_{2f-1}\defeq0$,
\item
if $v_{2f-1}'<1$ and $\varepsilon=1$ then $ v_{2f-1}=p+v_{2f-1}'-1 \geq 1$, $\bX_{2f-1}=\gO$ and we define $\varepsilon_{2f-1}=1$,
\item
if $v_{2f-1}'\geq 1$ and $\varepsilon=1$ then $ v_{2f-1}=v_{2f-1}'-1$ and we define $ \varepsilon_{2f-1}\defeq0$.
\end{itemize}
By decreasing induction, we deduce for $i\in\{2f-2,\ldots,0\}$ the following relations between $v_i$ and $v'_i$, and define $\eps_i\in\{0,1\}$ as follows:
\begin{enumerate}[label=(\alph*)]
\item
\label{it:def:v':1}
if $v_i'<0$ and $\varepsilon_{i+1}=0$ then $v_{i}=p+v_{i}'\geq 2$, $\bX_{i}=\gO$ and we define $\varepsilon_{i}\defeq1$,
\item
\label{it:def:v':2}
if $v_{i}'\geq 0$ and $\varepsilon_{i+1}=0$ then $v_{i}=v_{i}'$ and we define $\varepsilon_{i}\defeq0$,
\item
\label{it:def:v':3}
if $v_{i}'<1$ and $\varepsilon_{i+1}=1$ then $ v_{i}=p+v_{i}'-1 \geq 1$, $\bX_{i}=\gO$ and we define $\varepsilon_{i}\defeq1$;
\item 
\label{it:def:v':4}
if $v_{i}'\geq 1$ and $\varepsilon_{i+1}=1$ then $v_{i}=v_{i}'-1$ and we define $\varepsilon_{i}\defeq0$.
\end{enumerate}
By (\ref{equa:val:eps}) we have $\varepsilon=\varepsilon_0$ and hence the relations \ref{it:def:v':1}--\ref{it:def:v':4} between $v_i$ and $v'_i$ hold for \emph{any} $i\in\cJ'$, and define an element $(\eps_i)_{\cJ'}\in\{0,1\}^{\cJ'}$.

For any $i\in\{0,\ldots,2f-1\}$ such that $\varepsilon_{i+1}=1$, we have $\bX_{i+1}=\gO$. 
Hence $v_i=v'_i-1$ only occurs for $\bX_{i}\in\{\gO,\gAB\}$ by \eqref{en:abstract:gene:1}. 
The conclusion follows now from a direct application of conditions \eqref{en:gene:1}-\eqref{en:gene:5}
(for instance if $\bX_i=\gAB$ then $v_i=0$ by \ref{en:gene:1}, and from \ref{it:def:v':2}, \ref{it:def:v':4} above we conclude that $v'_i\in\{0,1\}$).

\end{proof}

The following lemma improves Lemma \ref{lem:translation:v:v'}.
\begin{lemma}
\label{lem:sequence:v'} 
Keep the assumptions of Lemma \ref{lem:translation:v:v'} let $i\in\{0,\ldots,2f-1\}$.
\begin{enumerate}[label=(\alph*)]
\item
\label{it:seq:v':a}
If $\bX_i=\gO$ and $v'_i=0$, then there exists $j\geq 0$ such that $v'_{j+i}<0$ and $v'_{i+\ell}=0$ for all $0\leq \ell<j$.
\item 
\label{it:seq:v':b}
If $\bX_i=\gO$ and $v'_i=1$, then there exists $j\geq 0$ such that $v'_{j+i}\geq 2$ and $v'_{i+\ell}=1$ for all $0\leq \ell<j$.
\item
\label{it:seq:v':c}
If $\bX_i=\gAB$ and $v_i'=1$, then there exists $j\geq 0$ such that $v'_{j+i}<0$ and $v'_{i+\ell}=0$ for all $0\leq \ell<j$.
\item
\label{it:seq:v':d} 
If $\bX_i=\gAB$  and $v_i'=0$, then $v'_{i+1}\geq 1$.
\end{enumerate}
\end{lemma}
\begin{proof}
In the notation of the proof of Lemma \ref{lem:translation:v:v'} we have:
\begin{enumerate}[label=(\alph*)]
\item
\label{pf:it:seq:v':a}
If $\bX_i=\gO$ and $v'_i=0$, then $\varepsilon_{i+1}=1$
and $\bX_{i+1}=\gO$ and $v'_{i+1}<1$. 
The claimed result follows now by induction.
\item
\label{pf:it:seq:v':b}
This is similar to \ref{pf:it:seq:v':a}.
\item
\label{pf:it:seq:v':c}
If $\bX_i=\gAB$ and $v_i'=1$ then $v_i=v'_i-1$, $\varepsilon_{i+1}=1$ and $\bX_{i+1}=\gO$. Thus $v_{i+1}'\leq 0$.  
We conclude by \ref{pf:it:seq:v':a}.
\item
\label{pf:it:seq:v':d}
If $\bX_i=\gAB$ and $v_i'=0$, then $v_i=v'_i$, $\varepsilon_{i+1}=0$ and $\bX_{i+1}=\gO$. Thus $v'_{i+1}\geq 1$.
\end{enumerate}
\end{proof}

\subsubsection{Step 2: Types and shapes in the fibers}\label{subsec:loc:equ:gen}
In this section we conclude our analysis on the fibers of \eqref{eq:map:7tuple_to_gene}.
The main result is Proposition \ref{prop:constraints}.

In the following we assume that $\bX\in\{\gA, \gB, \gAB, \gO\}^{\cJ'}$ satisfies conditions \eqref{en:abstract:gene:1}--\eqref{en:abstract:gene:extra}.
The following proposition analyzes the fiber of the composite map
\[
\big((s_{j+1}, s_{\orient,j},\Sigma_j,z_{j+1},\tld{w}_j, \textnormal{type at $j$}),(v'_{f+j},v'_j)\big)_{j\in\cJ}\mapsto (v'_{f+j},v'_j)_{j\in\cJ}\mapsto \bX(v')
\]
where the domain is the set of -tuples attached to triples $(\tld{w},s,\mu)$ as in the end of Section \ref{subsec:gen:tra}.

\begin{prop} \label{prop:constraints}
Let $\bX\in\{\gA, \gB, \gAB, \gO\}^{\cJ'}$ satisfy conditions \eqref{en:abstract:gene:1}--\eqref{en:abstract:gene:extra}.
Assume that $\bX=\bX(v')$ for some $(s_{j+1},s_{\orient,j}, \Sigma_j,z_{j+1},\tld{w}_j, \textnormal{type of $j$}, (v'_{j+f}, v'_j))$ associated to a triple $(\tld{w},s,\mu)$.
\begin{enumerate}
\item
\label{fibre:AO}
If $\binom{\bX_{j+f}}{\bX_{j}}=\binom{\gA}{\gO}$ then $(s_{j+1},s_{\orient,j}, \Sigma_j,z_{j+1},\tld{w}_j, \textnormal{type of $j$}, (v'_{j+f}, v'_j))$ belongs to the set
$$
\left\{\begin{array}{c}  (-,w_0,\Id,-,-,II, (0,<0)), (\Id,w_0,\Id,\Id,t_{w_0(\eta)}, I ,(0,0)),\cr (\Id,\Id,w_0,\Id,t_\eta,0, (0,1)), (-,\Id,w_0,-,-,II,(0,\geq 2))\cr\end{array}\right\}.$$
If $\binom{\bX_{j+f}}{\bX_{j}}=\binom{\gB}{\gO}$ then $(s_{j+1},s_{\orient,j}, \Sigma_j,z_{j+1},\tld{w}_j, \textnormal{type of $j$}, (v'_{j+f}, v'_j))$ belongs to the set
$$
\left\{\begin{array}{c}  (-,w_0,\Id,-,-,II, (1,<0)), (\Id,w_0,\Id,\Id,t_{\eta}, 0 ,(1,0)),\cr (\Id,\Id,w_0,\Id,t_{w_0(\eta)},I, (1,1)), (-,\Id,w_0,-,-,II,(1,\geq 2))\cr\end{array}\right\}.$$
\item
\label{fibre:OAB}
If $\binom{\bX_{j+f}}{\bX_{j}}=\binom{\gO}{\gAB}$ then $(s_{j+1},s_{\orient,j}, \Sigma_j,z_{j+1},\tld{w}_j, \textnormal{type of $j$}, (v'_{j+f}, v'_j))$ belongs to the set
$$
\left\{\begin{array}{c} (\Id,\Id,\Id,w_0,w_0t_{\eta},II, (\geq 2,0)),  (w_0,\Id,\Id,\Id,w_0t_{\eta},II, (\geq 2,1)),\cr
(w_0,w_0,w_0,\Id,w_0t_\eta,II, (<0,0)),(\Id,w_0,w_0,w_0,w_0t_{\eta},II, (<0,1))\cr\end{array}\right\}.$$
\item
\label{fibre:AAB}
If $\binom{\bX_{j+f}}{\bX_{j}}=\binom{\gA}{\gAB}$ then $(s_{j+1},s_{\orient,j}, \Sigma_j,z_{j+1},\tld{w}_j, \textnormal{type of $j$}, (v'_{j+f}, v'_j))$ belongs to the set
$$\left\{\begin{array}{cc}
(w_0,w_0,\Id,w_0,t_\eta,I,(0,0)),&
 (\Id,w_0,\Id,\Id,t_{w_0(\eta)},0,(0,1)),\cr
  (w_0,w_0,w_0,\Id,w_0t_\eta,I,(0,0)),&
   (\Id,w_0,w_0,w_0,w_0t_\eta, 0, (0,1))
   \cr\end{array}\right\}.$$      
  If $\binom{\bX_{j+f}}{\bX_{j}}=\binom{\gB}{\gAB}$,
  then $(s_{j+1},s_{\orient,j}, \Sigma_j,z_{j+1},\tld{w}_j, \textnormal{type of $j$}, (v'_{j+f}, v'_j))$ belongs to the set
$$\left\{\begin{array}{cc}
(w_0,\Id,w_0,w_0,t_\eta,I,(1,1)),&
 (\Id,\Id,w_0,\Id,t_{w_0(\eta)},0,(1,0)),\cr
  (w_0,\Id,\Id,\Id,w_0t_\eta,I,(1,1)),&
   (\Id,\Id,\Id,w_0,w_0t_\eta, 0, (1,0))
   \cr\end{array}\right\}.$$   
\item 
\label{fibre:BAB}
If $\binom{\bX_{j+f}}{\bX_{j}}\in\{\binom{\gA}{\gA},\binom{\gA}{\gB},\binom{\gB}{\gA},\binom{\gB}{\gB} \}$ then
$s_{j+1},s_{\orient,j}, (v'_{j+f},v'_{j})$ and $type$ of $j$ are determined by $\bX$. \\
Moreover if $\binom{\bX_{j+f}}{\bX_{j}}\in\{\binom{\gA}{\gA},\binom{\gB}{\gB}\}$, then $\bX$ determines if either $\tld{w}_{j}=t_{w_0(\eta)}$ or $\tld{w}_{j}\in\{w_0t_\eta,t_\eta\}.$ 
\end{enumerate}
The other cases are deduced by the transformation $\binom{\bX_{j+f}}{\bX_j}\mapsto \binom{\bX_{j}}{\bX_{j+f}}$   (see the caption of Table \ref{TableGen13}).
\end{prop}
\begin{proof}
Proof of (\ref{fibre:AO}).  
Assume $\bX_{j}=\gO$ and $v'_{j}=1$.
Let $j_1=\max\{j'\geq j \mbox{ such that } v'_{\ell}=1, j_0\leq \ell\leq j'\}$. 
Then $v'_{j_1+1}\geq 2$ and for all $j\leq j'\leq j_1$ we have $\binom{\bX_{f+j'}}{\bX_{j'}}\in
\left\{\binom{\gA}{\gO},\binom{\gB}{\gO}\right\}$,
 $(s_{j'+1},s_{\orient,j'},\Sigma_{j'},z_{j'+1})=(\Id,\Id,w_0,\Id)$ and
\[
(\tld{w}_{j'},\mbox{type of } j',(v'_{f+j'},v'_{j'}))=\left\{ \begin{array}{ccc}(t_\eta,0,(0,1))&\mbox{ if } \bX_{f+j'}=\gA,\cr
 (t_{w_0(\eta)},I,(1,1))&\mbox{ if }  \bX_{f+j'}=\gB.\cr\end{array}\right.
 \]
Indeed
by \eqref{en:abstract:gene:extra} and Lemma \ref{lem:sequence:v'}, $j_1$ is well defined,
$\bX_{j'}=\gO$, for all  $j\leq j'\leq j_1+1$ and $v'_{j_1+1}\geq 2$.
Hence $\binom{v'_{j_1+1+f}}{v'_{j_1+1}}\in\left\{\binom{0}{\geq 2},\binom{1}{\geq 2}\right\}$.\\ 
By Table \ref{TableGen13} and symmetry, the index $j_1+1$ is of type $II$ and $s_{\orient,j_1+1}=\Id,$ $\Sigma_{j_1+1}=w_0$.\\
By  \eqref{en:abstract:gene:1} and \eqref{en:abstract:gene:extra}, for all $j\leq j'\leq j_1$ we have $\bX_{j'+f}\in\{\gA,\gB\}$,  hence $v'_{f+j'}\in\{0,1\}$ is determined by $\bX_{f+j'}$. That is, for $j\leq j'\leq j_1$,  $\binom{v'_{j'+f}}{v'_{j'}}\in\left\{\binom{0}{1},\binom{1}{1}\right\}$. \\
By decreasing induction and Table \ref{TableGen13}, since $s_{\orient,j_1+1}=\Id$ and $\Sigma_{j_1+1}=w_0$, we have for $j\leq j'\leq j_1$
\begin{itemize}
\item 
if $\binom{v'_{j'+f}}{v'_{j'}}=\binom{0}{1}$,
$(s_{j'+1},s_{\orient,j'}, \Sigma_{j'},z_{j'+1},\tld{w}_{j'}, \textnormal{type of $j'$})=
(\Id,\Id, w_0,\Id,t_\eta,0),$ 
\item 
if $\binom{v'_{j'+f}}{v'_{j'}}=\binom{1}{1}$,$(s_{j'+1},s_{\orient,j'}, \Sigma_{j'},z_{j'+1},\tld{w}_{j'}, \textnormal{type of $j'$})=
(\Id,\Id, w_0,\Id,t_{w_0(\eta)},I).$ 
\end{itemize}
The proof for $v'_{j}=0$ can be deduced directly by symmetry. The cases $v'_{j}<0$ and $v'_{j}\geq 2$ follow from Table \ref{TableGen13}.

Proof of (\ref{fibre:OAB}). 
Assume $\binom{\bX_{j+f}}{\bX_{j}}=\binom{\gO}{\gAB}$. By (\ref{en:abstract:gene:1}) and item \ref{en:abstract:gene:extra} of this Proposition we have $\bX_{j+1+f}\not=\gO$. Hence  $v'_{f+j}\geq 2$ or $<0$. 
 By Table \ref{TableGen13} and (\ref{fibre:AO}) we have $(s_{\orient,j+1},\Sigma_{j+1})\in\{(w_0,\Id), (\Id,w_0)\}$. 

 Then
$$(s_{j+1},s_{\orient,j},\Sigma_j,z_{j+1})\in\{(w_0,\Id,\Id,\Id),(\Id,\Id,\Id,w_0),(\Id,w_0,w_0,w_0),(w_0,w_0,w_0,\Id)\}$$ and $\tld{w}_j=w_0t_\eta$ and the type of $j$ is $II$.

Proof of (\ref{fibre:AAB}). 
If $\binom{\bX_{j+f}}{\bX_{j}}\in\left\{\binom{\gA}{\gAB},\binom{\gB}{\gAB}\right\}$, then $\binom{v'_{j+f}}{v'_{j}}=\left\{\binom{0}{0,1},\binom{1}{0,1}\right\}$, $\bX_{j+1}=\gO$ and
$$v'_{j+1}=\left\{\begin{array}{ccccc} \geq 1&\mbox{if }& v_j'=0,&\mbox{ hence }& (s_{\orient,j+1},\Sigma_{j+1})= (\Id,w_0),\cr
\leq 0&\mbox{if }&v_j'=1,&\mbox{ hence }& (s_{\orient,j+1},\Sigma_{j+1})= (w_0,\Id).\cr\end{array}\right.$$
We conclude from Table \ref{TableGen13}.

Proof of (\ref{fibre:BAB}).
Assume $\bX_{j},\bX_{j+f}\in\{\gA,\gB\}$.
\begin{itemize}
\item Assume that there exists $j_1$ such that
$$\left\{\begin{array}{c} 
\bX_{j'}, \bX_{j'+f}\in\{\gA,\gB\}, j\leq j'\leq j_1-1,\cr
\gAB\in\{\bX_{j_1}, \bX_{j_1+f}\}.\cr   
\end{array}\right.$$
By (\ref{fibre:AAB}) we see that $s_{\orient,j_1}$ is determined by $\binom{\bX_{j_1+f}}{\bX_{j_1}}$.
Since $\bX_{j'},\bX_{j'+f}\in\{\gA,\gB\},$ for $j\leq j'\leq j_1-1$ we deduce by symmetry and decreasing induction from Table \ref{TableGen13} that $s_{j'+1}$, $s_{\orient,j'}$ and the type at $j'$ are determined by $\binom{\bX_{j'+f}}{\bX_{j'}}$ for $j\leq j'\leq j_1-1$.
\item
Assume that such $j_1$ does not exist, i.e.~that $\bX_{j'}\in\{\gA,\gB\}$ for all $j'\in\cJ'$. Then for all $j'\in\cJ'$ either $s_{j'+1}$ or $s_{\orient,j'}$ is determined by $\binom{\bX_{j'+f}}{\bX_j'}$ (cf.~Table \ref{TableGen13}).
If there exists $i\in\cJ$ such that 
$\binom{\bX_{i+f}}{\bX_i}\in\left\{\binom{\gA}{\gB},\binom{\gB}{\gA}\right\}$, there is an unique choice of $s_{\orient,i}$ such that $s_{\orient,f-1}=\Id$. 
If $\binom{\bX_{i+f}}{\bX_i}\in\left\{\binom{\gA}{\gA},\binom{\gB}{\gB}\right\}$ for all $i\in\Z/f\Z$ then the $\cJ$-tuple $(s_{\orient,i})$ is determined by $\bX$. 
By (\ref{en:abstract:gene:3}) there exists $j'\in\cJ'$ with $\bX_{j'}\not=\bX_{j'+1}$. 
By Tables \ref{TableGen13} and symmetry we obtain $s_{j'+1}=w_0$. 
By induction we conclude that $(s_{j+1},s_{\orient,j})_{j\in\cJ}$ and the type of $j$ are determined by $\bX$.
\end{itemize}
Moreover, for $\binom{\bX_{j+f}}{\bX_j}\in\{\binom{\gA}{\gA},\binom{\gB}{\gB}\}$, we deduce from Table \ref{TableGen13} (and symmetry) that the data of $(s_{j+1},s_{\orient,j})$ determines whether $\tld{w}_j=t_{w_0(\eta)}$ or  $\tld{w}_j\in\{w_0t_\eta, t_\eta\}$
\end{proof}

\subsection{Naive equations associated to a gene}
\label{sub:Naive:eq}

Let $\bX\in\{\gA, \gB, \gAB, \gO\}^{\cJ'}$ be a gene satisfying conditions \eqref{en:abstract:gene:1}--\eqref{en:abstract:gene:extra}.
Assume that $\bX=\bX(v')$ for some $(v'_{f+j},v'_j)_{j\in\cJ}$ associated to a triple $(\tld{w},s,\mu)$.\\

\begin{defn}
\label{def:cluster}
Let $\bX=(\bX_{j})_{j\in\cJ'}$ be a gene satisfying \eqref{en:abstract:gene:1}--\eqref{en:abstract:gene:extra}.
A \emph{cluster} for $\bX$ is a sequence $(\bX_{j+f},\bX_{j})_{j_0\leq j\leq j_1}$ such that there exists $\ell\in\{j_0,\ldots,j_1\}$ satisfying
\begin{itemize}
\item
$\bX_j=\gO$ for all $j\in\{\ell,\ldots,j_1\}$,  $\bX_{j_1+1}\not=\gO$  and $\bX_{\ell-1}=\gAB$,
\item 
$\gO\not\in\{\bX_j,\bX_{j+f}\}$, $j_0+1\leq j\leq \ell-1$,
\item
$\gO\in\{\bX_{j_0},\bX_{j_0+f}\}$.
\end{itemize}
We remark that $\bX=(\bX_j)_{j\in\cJ'}$ has either a \emph{unique} decomposition $\bX=\cup_{i=0}^{r}(\bX_{j_i},\bX_{j_{i+f}})_{j_i\leq j\leq j_{i+1}}$ into clusters (with the convention $j_{r+1}=j_0$), 
or does not have \emph{any} cluster (in which case $\bX_{j}\in\{\gA,\gB\}$ for all $j\in\cJ'$).
Note finally that $j_i$ are of type $II$ for all $i=0,\dots,r$ by Table \ref{TableGen4}.
\end{defn}

Let $\cup_{k\in\cK}\cJ_k$ be the fragmentation associated to the triple $(\tld{w},s,\mu)$ (Definition \ref{def:aux:not}\ref{def:fragmentation}). Assume $\bX$ admits a decomposition into clusters $\bX=\cup_{i=0}^{r}(\bX_{j_i},\bX_{j_{i+f}})_{j_i\leq j\leq j_{i+1}}$.
Note that given a cluster $(\bX_{j+f},\bX_{j})_{j_i\leq j\leq j_{i+1}}$ the sequence $(j_i,\dots,j_{i+1})$ is a union of fragments of $\cK$, as both $j_i$ and $j_{i+1}$ are of type $II$.
For $i\in\{0,\ldots,r\}$, we denote
$$R_{{[j_i,j_{i+1}]}}\defeq\cO[Y_{j_i}]\otimes \left(\otimes_{j_i+1\leq j\leq j_{i+1}-1} R_j\right) \otimes \cO[X_{j_{i+1}}]
\mbox{ and }
\cI_{[j_i,j_{i+1}]}\defeq \sum_{k\in \cK \cap [j_i,j_{i+1}]}\cI_{{k}}$$
where
for all $k\in \cK \cap [j_i,j_{i+1}]$, the ideal $\cI_{{k}}\subset R_{[j_i,j_{i+1}]}$ is generated by the entries of the matrix equation \eqref{matrix:equation:explicit} associated to $\cJ_{k}$.

\begin{prop}\label{prop:ring:naive}
Let $\bX=(\bX_j,\bX_{j+f})_{j\in\cJ}$ be a gene satisfying \eqref{en:abstract:gene:1}--\eqref{en:abstract:gene:3}.
\begin{enumerate}
\item
\label{prop:caseO}
Assume that $\bX$ has a decomposition into clusters $\bX=\cup_{i=0}^{r}(\bX_{j_i},\bX_{j_{i+f}})_{j_i\leq j\leq j_{i+1}}$. 
Let $(\tld{w},s,\mu)$ be a triple in the fiber above $\bX$ of the map \eqref{eq:map:triple_to_gene}.
Then
\[
\cZ^{\nv,\tau}(\tld{z})=\Spec\left(\otimes_{i=0}^{r} R_{[j_i,j_{i+1}]}/\cI_{[j_i,j_{i+1}]}\right).
\]
Moreover, the $p$-saturation of $\cZ^{\nv,\tau}(\tld{z})$ depends only on $\bX$.
\item
\label{prop:casenotO}
Assume $\bX$ does not admit a decomposition into fragment (i.e.~$\bX\in\{\gA,\gB\}^{\cJ'}$). 
Then 
$\cZ^{\nv,\tau}(\tld{z})$ depends only on $\bX$.
In particular, the $p$-saturation of $\cZ^{\nv,\tau}(\tld{z})$ depends only on $\bX$.
\end{enumerate}
\end{prop}

The proof of Proposition \ref{prop:ring:naive} relies on the analysis of each of the ideals $\cI_{[j_i,j_{i+1}]}$. 
This analysis is preformed, for a fixed cluster, in Lemmas \ref{lem:v'0:gO}, \ref{lem:trans:AB} and \ref{lem:seq:AB} below, which deal with the $\gO$, the $\gAB$, and the $(\gA,\gB)$-part of the cluster.\\

Thus, we fix once and for all a cluster $(\bX_j,\bX_{j+f})_{j_0\leq j\leq j_1}$ of $\bX$.   
We let 
$\{k_0,\ldots,k_s\}\defeq \cK\cap [j_0,j_1]$ so that
\begin{itemize}
\item 
for all $j=0,\ldots, s-1,$ $\cJ_{k_j}=(i_{k_j},i_{k_j}-1,\ldots,o_{k_j}+1,o_{k_j})\subset (j_1,j_1-1,\ldots,j_0)$ and $o_{k_j}=i_{k_{j+1}}$
\item
$i_{k_0}=j_1$, $o_{k_s}=j_0$.
\end{itemize}
Let $s'\in[0,s]$ such that $\ell-1\in(o_{k_{s'}},\ldots, i_{k_{s'}}-1)$.\\

Given $j\in[0,s]$,  we define $\overline{\cI}_{k_j}$ as the image of $\cI_{{k_j}}$ in $R_{[j_0,j_1]}/\sum_{j'=0}^{j-1}\cI_{{k_{j'}}}$.
We still denote by $X_i,Y_i,Z_i$ etc.~ the variables of $R_{[j_0,j_1]}$ in the quotient $R_{[j_0,j_1]}/\sum_{j'=0}^{j-1}\cI_{{k_{j'}}}$.

\begin{lemma}[the $\gO$-part of a cluster]
\label{lem:v'0:gO}
For any $0\leq i\leq s'-1$, the ideal $\overline{\cI}_{k_i}$ satisfies the following property: there exists $r_{k_i}\in\N$ such that
 \begin{enumerate}
\item 
\label{it:frag:II:II}
$\overline{\cI}_{k_i}=(Y_{o_{k_i}}) \mbox{ if the type of both $o_{k_i}$ and $i_{k_i}$ is } II,$
\item
\label{it:frag:II:I}
$\overline{\cI}_{k_i}=(f_{k_i}, g_{k_i}) \mbox{ if the type of $o_{k_i}$ is $I$ and the type of $i_{k_i}$ is $II$},$
\item
\label{it:frag:I:I}
$p^{r_{k_i}}\overline{\cI}_{k_i}=p^{r_{k_i}}(f_{k_i}, g_{k_i})(p,X_{i_{k_i}}) \mbox{ if the type of both $o_{k_i}$ and $i_{k_i}$ is } I,$ 
\item
\label{it:frag:I:II}
$ 
p^{r_{k_i}}\overline{\cI}_{k_i}=p^{r_{k_i}}(Y_{o_{k_i}})(p,X_{i_{k_i}})
 \mbox{ if the type of $o_{k_i}$ is } II \mbox{ and the type of $i_{k_i}$ is } I,$\end{enumerate} 
where
$$\left\{\begin{array}{c} f_{k_i}\defeq p-Y_{o_{k_i}}-X_{o_{k_i}+1}X_{o_{k_i}},\cr 
g_{k_i}\defeq Z_{o_{k_i}}-X_{o_{k_i}+1}(p-X_{o_{k_i}+1}X_{o_{k_i}}).\cr
\end{array} \right.$$

\end{lemma}
\begin{proof}
If $s'=0$, there is nothing to prove. We assume now $s'\geq 1$.

By definition of $\ell$, $\bX_j=\gO$ for all $j\in\{\ell,\ldots,j_1\}$ and $\bX_{j+f}\in\{\gA,\gB\}$ for all $j\in\{\ell,\ldots,j_1-1\}$.
Thus by Proposition \ref{prop:constraints}\eqref{fibre:AO} and an induction (using Table \ref{TableGen13} to deal when $j$ is of type $II$) we have for $j\in\{\ell,\ldots,j_1-1\}$:
\begin{equation}
\label{casewise:O}
(\tld{w}_{j},\mbox{type of } j)\in\{(t_\eta,0), (t_{w_0(\eta)},I),
(t_\eta,II),(t_{w_0(\eta)},II)\}.
\end{equation}
Recall that for any $k\in\cK$ and fragment $(i_k,i_k+1,\ldots,o_k-1,o_k)$, we know that $i_k$ and $o_k$ are of type either $I$ or $II$, and $i_k-1,\ldots,o_k+1$ are of type 0. 
Thus for any $i\in \{0,\ldots,s'-1\}$ we get from \eqref{casewise:O} and Table \ref{Table:Geometric_genes}:
$$\prod_{l=o_{k_i}+1}^{i_{k_i}-1}T_l=\prod_{l=o_{k_i}+1}^{i_{k_i}-1}\left(\begin{array}{cc}1 & 0\cr -X_l&p\cr\end{array}\right)=\left(\begin{array}{cc} 1&0\cr \sum_{l=0}^{i_{k_i}-o_{k_i}-2}-p^l X_{o_{k_i}+1+l}& p^{i_{k_i}-o_{k_i}-1}\cr\end{array}\right).$$
To prove the lemma we proceed by induction on $i\in\{0,\ldots,s'-1\}$ analyzing the matrices $M_{\textnormal{out},o_{k_i}}$ and $M_{\textnormal{in},i_{k_i}}$ on the fragments $\cJ_{k_i}$.
We abbreviate $k\defeq k_i$ in what follows.
For $i=0$, only the items \eqref{it:frag:II:II} and \eqref{it:frag:II:I} can happen (as $i_{k_0}$ is of type $II$).

\begin{enumerate}
\item[Proof of item \eqref{it:frag:II:II}.]
Assume $o_k$ and $i_k$ are of type $II$.
Then by \eqref{casewise:O} and Table \ref{Table:Geometric_genes}:
$$M_{\textnormal{out},o_k}
=\left\{\begin{array}{ccc}
(Y_{o_k},-p)&\mbox{ if }&\tld{w}_{o_k}=t_\eta,\cr
(0,1)&\mbox{ if }&\tld{w}_{o_k}=t_{w_0(\eta)},\cr
\end{array}\right. \mbox{ and }
 M_{\textnormal{in},i_k}=\binom{1}{-X_{i_k}}.$$
After the change of variable $X_{o_k+1}\mapsto \sum_{l=0}^{i_k-o_k-1}p^l X_{o_k+1+l}$ in $R_{[j_0,j_1]}$ we have
\begin{equation}
\label{eq:chg:var}
\left(\prod_{l=o_k+1}^{i_k-1}T_l\right)M_{\textnormal{in},i_k}
=\binom{1}{-X_{o_k+1}}.
\end{equation} 
(In particular, note that the ``inner'' variables of $R_{[j_0,j_1]}$ in the outcome of the Lemma are not the same as those used in Table \ref{Table:Geometric_genes}. 
This change of variable can be checked not to be relevant in the gluing of Proposition \ref{prop:ring:naive}, which uses equations from Table \ref{Table:Geometric_genes} involving only the ``outer'' variables of $ R_{[j_0,j_{1}]}$ for each cluster.)
Then
$$\cI_{k}=\left\{\begin{array}{ccc}
(Y_{o_k}-pX_{o_k+1})&\mbox{ if }&\tld{w}_{o_k}=t_\eta,\cr
(X_{o_k+1})&\mbox{ if }&\tld{w}_{o_k}=t_{w_0(\eta)}.\cr
\end{array}\right.$$
Thus, up to the change of variable $Y_{o_k}\mapsto Y_{o_k}-pX_{o_k+1}$ if $\tld{w}_{o_k}=t_\eta$ or $Y_{o_k}\mapsto X_{o_k+1}$, $X_{o_k+1}\mapsto Y_{o_k}$ if $\tld{w}_{o_k}=t_{w_0(\eta)}$, we have
$\cI_{k}=(Y_{o_k})$ and item \eqref{it:frag:II:II} holds.

\item[Proof of item \eqref{it:frag:II:I}.]
Assume $o_k$ is of type $I$ and $i_k$ is of type $II$. As before we have \eqref{eq:chg:var}
after replacing $X_{o_k+1}$ by $\sum_{l=0}^{i_k-o_k-1}p^l X_{o_k+1+l}$ and hence, using again \eqref{casewise:O} and Table \ref{Table:Geometric_genes}
\begin{equation}
\label{eq:simplify:2}M_{\textnormal{out},o_k}\left(\prod_{l=o_k+1}^{i_k-1}T_l\right)M_{\textnormal{in},i_k}=\left(\begin{array}{cc} p-Y_{o_k}&X_{o_k}\cr Z_{o_k}&Y_{o_k}\cr\end{array}\right) \binom{1}{-X_{o_k+1}}= 
\binom{p-Y_{o_k}-X_{o_k+1}X_{o_k}}{Z_{o_k}-Y_{o_k}X_{o_k+1}}
\end{equation}
Thus $\cI_{k}=(p-Y_{o_k}-X_{o_k+1}X_{o_k},Z_{o_k}-X_{o_k+1}(p-X_{o_k+1}X_{o_k}))$ and item \eqref{it:frag:II:I} holds.\\

\item[Proof of item \eqref{it:frag:I:I}.]
Assume $o_k$ and $i_k$ are both of type $I$. Since $j_1$ is of type $II$ and $i>0$, there exist $i'\in[0,i-1]$  such that $o_{k_{l}}=i_{k_{l+1}}$ is of type $I$ for $i'\leq l\leq i-1$ and $i_{k_{i'}}$ is of type $II$.
By item \eqref{it:frag:II:I} applied to $k_{i'}$ and noting that $o_{k_{i'}}=i_{k_{i'+1}}$ we see that $M_{\textnormal{in},i_{k_{i'+1}}}$ equals:
\begin{equation}
\label{eq:ind:case3}
\left(\begin{array}{cc} Y_{i_{k_{i'+1}}}& -X_{i_{k_{i'+1}}}\cr -Z_{i_{k_{i'+1}}}& p-Y_{i_{k_{i'+1}}}\cr\end{array}\right)\equiv
\binom{1}{-X_{i_{k_{i'+1}}+1}}(p-X_{i_{k_{i'+1}}+1}X_{i_{k_{i'+1}}},-X_{i_{k_{i'+1}}})\qquad\text{modulo $\cI_{{k_{i'}}}$}.
\end{equation}
We conclude by \eqref{casewise:O} and Table \ref{Table:Geometric_genes}) that the image $\overline{\cI}_{{k_{i'+1}}}$ of
$\cI_{{k_{i'+1}}}$ in $R_{[j_0,j_1]}/(\sum_{l=0}^{i'}\cI_{{k_l}})$ is generated by the equations
{\footnotesize 
\begin{equation}
\label{eq:ind:case3'}
\hspace{-1.5cm}\left(\begin{array}{cc} p-Y_{o_{k_{i'+1}}}&X_{o_{k_{i'+1}}}\cr Z_{o_{k_{i'+1}}}&Y_{o_{k_{i'+1}}}\cr\end{array}\right)\binom{1}{-(\sum_{l=0}^{i_{k_{i'+1}}-o_{k_{i'+1}}-2}p^lX_{o_{k_{i'+1}}+1+l}-p^{o_{k_{i'}}-o_{k_{i'+1}}-1}X_{i_{k_{i'}+1}})} (p-X_{i_{k_{i'+1}}+1}X_{i_{k_{i'+1}}}, -X_{i_{k_{i'+1}}}) =0.
\end{equation}
}
Hence, after the change of variable $X_{o_{k_{i'+1}}+1}\mapsto \sum_{l=0}^{i_{k_{i'+1}}-o_{k_{i'+1}}-2}p^lX_{o_{k_{i'+1}}+1+l}-p^{o_{k_{i'}}-o_{k_{i'+1}}-1}X_{i_{k_{i'}+1}}$, the image $\overline{\cI}_{k_{i'+1}}$ of
$\cI_{{k_{i'+1}}}$ in $R_{[j_0,j_1]}/(\sum_{i=0}^{i'}\cI_{{k_i}})$ satisfies
$\overline{\cI}_{k_{i'+1}}=(f_{k_{i'+1}},g_{k_{i'+1}})(p, X_{i_{k_{i'+1}}})$ where
$f_{k_{i'+1}}=p-Y_{o_{{k_{i'+1}}}}-X_{o_{k_{i'+1}}}X_{o_{k_{i'+1}+1}}$ and $g_{k_{i'+1}}=Z_{o_{k_{i'+1}}}-X_{o_{k_{i'+1}}+1}(p-X_{o_{k_{i'+1}}}X_{o_{k_{i'+1}+1}})$.\\
We now induct on $l\in\{0,\dots,i-i'-1\}$ the case $l=0$ being covered above.
Indeed, assuming by induction that $p^{l-1}\overline{\cI}_{k_{i'+l-1}}=p^{l-1}(f_{k_{i'+l-1}},g_{k_{i'+l-1}})(p, X_{i_{k_{i'+l-1}}})$ (note that the term $(p, X_{i_{k_{i'+l-1}}})$ only appears for $l>1$) we can repeat the same computations above (multiplying bot side of equations \eqref{eq:ind:case3}, \eqref{eq:ind:case3'} by $p^{l-1}$)
to show that the image $\overline{\cI}_{{k_{i'+l}}}$ of $\cI_{{k_{i'+l}}}$ in $R_{[j_0,j_1]}/\sum_{l'=0}^{l-1}\cI_{{k_{i'+l'}}}$
satisfies $p^{l}\overline{\cI}_{k_{i'+l'}}=p^{l}(f_{k_{i'+l'}},g_{k_{i'+l'}})(p, X_{i_{k_{i'+l'}}})$, where
$f_{k_{i'+l'}}=p-Y_{o_{{k_{i'+l'}}}}-X_{o_{k_{i'+l'}}}X_{o_{k_{i'+l'}+1}}$ and $g_{k_{i'+l'}}=Z_{o_{k_{i'+l'}}}-X_{o_{k_{i'+l'}}+1}(p-X_{o_{k_{i'+l'}}}X_{o_{k_{i'+l'}+1}})$.\\

\item[Proof of item \eqref{it:frag:I:II}.] 
Assume $i_k$ is of type $I$ and $o_k$ is of type $II$.
Then $k_{i-1}$ satifies the hypotheses of \eqref{it:frag:II:I} or \eqref{it:frag:I:I}
Similar computations as in \eqref{it:frag:I:I} show now that 
the image of
$p^{r_k}\cI_{{k}}$ in $R_{[j_0,j_1]}/(\sum_{l=0}^{i-1}\cI_{\cJ_{k_l}})$ is generated by the equations
{\small
$$ p^{r_k}M_{\textnormal{out},o_k}\binom{1}{-(\sum_{l=0}^{i_{k}-o_{k}-2}p^lX_{o_{k_{}}+1+l}-p^{o_{k_{i-1}}-o_{k_{}}-1}X_{i_{k_{}+1}})} (p-X_{i_{k}+1}X_{i_{k}}, -X_{i_{k_{}}})=0$$
}
where 
$$M_{\textnormal{out},o_k}
=\left\{\begin{array}{ccc}
(Y_{o_k},-p)&\mbox{ if }&\tld{w}_{o_k}=t_\eta,\cr
(0,1)&\mbox{ if }&\tld{w}_{o_k}=t_{w_0(\eta)},\cr
\end{array}\right.$$
Then (after exchanging and/or replacing the variables in the same way as we did for \eqref{it:frag:II:II}), we have $r_k\in\N$ sucht that
$p^{r_k}\overline{\cI}_k=p^{r_k}(Y_{o_k})(p, X_{i_k})$.
\end{enumerate}
\end{proof}

\begin{lemma}[the $\gAB$-part of a cluster]
\label{lem:trans:AB}
Let $k\defeq k_{s'}$.
\begin{enumerate}
\item
\label{item:cAB:I}
If the type of $\ell-1$ is $I$ then the image $\overline{\cI}_{k}$ of $\cI_{k}$ in $R_{[j_0,j_1]}/\sum_{i=0}^{s'-1}\cI_{k_i}$ satisfies the following property: there exists $r_k\in\N$ such that 
$$\left\{
\begin{array}{cc} 
\overline{\cI}_{k}= (f_{k},g_{k}) &\mbox{ if $i_k$  is of type $II$,}\cr
 p^{r_{k}}\overline{\cI}_{k}= p^{r_k}(f_{k},g_{k})(p,X_{i_k}) &\mbox{ if $i_k$ is of type $I$,}\cr
\end{array}\right.$$
with $f_{k}\defeq p-Y_{o_k}-X_{o_k+1} X_{o_k}$, $g_{k}\defeq Z_{o_k}-X_{o_k+1}(p-X_{o_k+1}X_{o_k}).$
Moreover there exists $r_{k_{s'+1}}\in\N$ such that 
 $p^{r_{k_{s'+1}}}\overline{\cI}_{k_{s'+1}}$ is generated by the equations
 $$p^{r_{k_{s'+1}}}M_{\textnormal{out},o_k}\left(\prod_{l=o_k+1}^{\ell-2}T_l\right)M_{\ell-1}=0$$
where,
$$M_{\ell-1}=\left\{\begin{array}{cc} 
 	\Sigma_{\ell-1}\binom{-X_{\ell-1}}{p-X_{\ell}X_{\ell-1}}&\mbox{ if } \bX_{\ell-1+f}=\gA,\cr
 	w_0\Sigma_{\ell-1}\binom{-X_{\ell-1}}{p-X_{\ell}X_{\ell-1}}&\mbox{ if } \bX_{\ell-1+f}=\gB.\cr
 \end{array}\right.$$
 
\item 
\label{item:cAB:0}
If the type $\ell-1$ is $0$, then the image $\overline{\cI}_{k}$ of $\cI_{k}$ in $R_{[j_0,j_1]}/\sum_{i=0}^{s'-1}\cI_{k_i}$ satisfies the following property: there exists $r_k\in\N$ such that 
 $p^{r_k}\overline{\cI}_k$ is generated by the equations
 $$p^{r_k}M_{\textnormal{out},o_k}\left(\prod_{l=o_k+1}^{\ell-2}T_l\right)M_{\ell-1}=0$$
where
\[
M_{\ell-1}\defeq
\left\{\begin{array}{cc}(w_0)^{\delta_{\bX_{f+\ell-1}=\gB}}\Sigma_{\ell-1} \binom{-X_{\ell-1}}{p-X_{\ell-1}X_{\ell}}& \mbox{ if $i_k$ is of type } II,\cr
 (w_0)^{\delta_{\bX_{f+\ell-1}=\gB}}\Sigma_{\ell-1} \binom{-X_{\ell-1}}{p-X_{\ell-1}X_{\ell}}(X_{\ell+1},p)& \mbox{ if $i_k$ is of type } I,\cr
  \end{array}\right.
  \]

\item
\label{item:cAB:II}
If the type of $\ell-1$ is $II$, then the image $\overline{\cI}_{k}$ of $\cI_{k}$ in $R_{[j_0,j_1]}/\sum_{i=0}^{s'-1} \cI_{k_i}$ satisfies the following property: there exists $r_{k}\in\N$ such that
$$\left\{\begin{array}{cc} \overline{\cI}_{k}= (h_{k})&\mbox{ if $i_{k}$ is of type $II$,}\cr
 p^{r_{k}}\overline{\cI}_{k}= p^{r_{k}}(h_{k})(p,X_{i_{k}})&\mbox{ if $i_{k}$ is of type $I$},\cr
 \end{array}\right.$$
where $h_{k}=(p+X_{o_{k}+1}Y_{o_{k}}).$
\end{enumerate}
\end{lemma}
\begin{proof}
If the type of $\ell-1$ is $I$ or 0, then $\ell-1>j_0$ (as $j_0$ is of type $II$) and $(\bX_{\ell-1},\bX_{\ell-1+f})\in\{(\gAB,\gA),(\gAB,\gB)\}$. By Proposition \ref{prop:constraints}\ref{fibre:AAB},
$$(\tld{w}_{\ell-1},\mbox{ type of } \ell-1)\in\{(t_\eta,I), (w_0t_\eta,I), (w_0t_\eta,0), (t_{w_0(\eta)},0)\}.$$
Proof of Item (\ref{item:cAB:I}). 
If $\ell-1$ is of type $I$, then $o_k=\ell-1$. 
Using \ref{lem:v'0:gO} \eqref{it:frag:II:I}-\eqref{it:frag:I:I} to describe $R_{[j_0,j_1]}/\sum_{l=0}^{s'-1}\cI_{{k_{l}}}$ and performing the reasoning of the proof of \emph{loc.~cit}.~(replacing the final matrix at $o_k$ in equations \eqref{eq:simplify:2}, \eqref{eq:ind:case3'} with the final matrix associated to either $(t_\eta,I)$ or $(w_0t_\eta,I)$ according to $\tld{w}_{o_k}$) there exists $r_k\in\N$ such that the image $\overline{\cI}_k$ of $\cI_{k}$ in $R_{[j_0,j_1]}/\sum_{l=0}^{s'-1}\cI_{{k_{l}}}$ satisfies
$$\left\{\begin{array}{cc}
\overline{\cI}_k=(f_k,g_k)&\mbox{ if $i_k$ is of type }II,\cr
p^{r_k}\overline{\cI}_k=p^{r_k}(f_k,g_k)(p, X_{i_k})&\mbox{ if $i_k$ is of type }I,\cr
\end{array}\right.$$
where 
\begin{equation}
\label{eq:chg:var:AAB}
\left\{\begin{array}{c} f_k\defeq p-Y_{o_k}-X_{o_k}X_{o_k+1},\cr
g_k\defeq Z_{o_k}-X_{o_k+1}(p-X_{o_k}X_{o_k+1})\cr
\end{array}\right.
\end{equation}
Again, note that the above equations are obtained by replacing $X_{o_k+1}$ by $-X_{o_k+1}$ and, more importantly in what follows, exchanging $X_{o_k}$ and $Z_{o_k}$ (resp.~replacing $Y_{o_k}$ by $p-Y_{o_k}$) when $\tld{w}_{\ell-1}=t_\eta$ (resp.~$\tld{w}_{\ell-1}=w_0t_\eta$).\\
Letting $r_{k_{s'+1}}$ be $r_{k_{s'}}+1$ if $i_k$ is of type $I$ and be $0$ if $i_k$ is of type $II$, we furthermore claim that
 $p^{r_{k_{s'+1}}}\overline{\cI}_{k_{s'+1}}$ is generated by the equations 
 $$p^{r_{k_{s'+1}}}M_{\textnormal{out},o_{k_{s'+1}}}\left(\prod_{l=o_{k_{s'+1}}+1}^{\ell-2}T_l\right)M_{\ell-1}=0$$
where, 
\[
M_{\ell-1}\defeq (w_0)^{\delta_{(\bX_{f+\ell-1}=\gB)}}\,
 \Sigma_{\ell-1}\binom{-X_{\ell-1}}{p-X_{\ell}X_{\ell-1}}.
 \]
Indeed, keeping in mind the change of variables $X_{o_k}\leftrightarrow Z_{o_k}$ (resp.~$Y_{o_k}\leftrightarrow p-Y_{o_k}$) when $\tld{w}_{\ell-1}=t_\eta$ (resp.~$\tld{w}_{\ell-1}=w_0t_\eta$) we obtain from row $I$ of Table \ref{Table:Geometric_genes} (recall $i_{k_{s'+1}}=o_{k_{s'}}=\ell-1$ is of type $I$) and  \eqref{eq:chg:var:AAB} 
\begin{align*}
p^{r_{k_{s'+1}}}M_{\textnormal{in},i_{k_{s'+1}}}&\equiv p^{r_{k_{s'}}+1}w_{\ell-1}\left(\begin{array}{cc} p-X_{\ell-1}X_{\ell}& -X_{\ell}(p-X_{\ell-1}X_{\ell})\cr-X_{\ell-1}& X_{\ell-1}X_{\ell}\cr\end{array}\right)w_{\ell-1}^{-1}\\
&=p^{r_{k_{s'+1}}}\binom{p-X_{\ell-1}X_{\ell}}{-X_{\ell-1}}(1,-X_{\ell})w_{\ell-1}^{-1}
\end{align*}
in $R_{[j_0,j_1]}/\sum_{l=0}^{s'}\cI_{{k_{l}}}$, where $w_{\ell-1}$ is the permutation part of $\tld{w}_{\ell-1}$.
As $i_{k_{s'+1}}$ is of type $I$ we deduce from Proposition \ref{prop:constraints}\eqref{fibre:AAB} that $w_{\ell-1}= (w_0)^{\delta_{\bX_{f+\ell-1}=\gB}}
 \Sigma_{\ell-1}$ and the claim follows noting that the term $(1,-X_{\ell})w_{\ell-1}^{-1}$ can be ignored when imposing condition \eqref{matrix:equation:explicit}.\\
Proof of item (\ref{item:cAB:0}). 
A similar reasoning as in the proof of Lemma \ref{lem:v'0:gO} shows that there exists $r_k\in\N$ and a change of variable for $X_\ell$ such that we have the following equality in $R_{[j_0,j_1]}/\sum_{i=0}^{s'-1}\cI_{k_i}$:
\begin{equation}
\label{eq:case:2:ell}
p^{r_k}\Big(\prod_{l=\ell}^{i_{k}-1}T_l\Big)_{\textnormal{in},i_k}=
\begin{cases}
p^{r_k}\begin{pmatrix}1\\-X_{\ell}\end{pmatrix}&\textnormal{if $i_k$ is of type $II$}\\
p^{r_k}\begin{pmatrix}1\\X_{\ell}\end{pmatrix}\begin{pmatrix}p-X_{i_k+1}X_{i_k},&-X_{i_k}\end{pmatrix}&\textnormal{if $i_k$ is of type $I$}
\end{cases}
\end{equation}
Moreover if $\ell-1$ of type $0$, then $(\tld{w}_{\ell-1},\mbox{ type of } \ell-1)\in\{((w_0t_\eta,0), (t_{w_0(\eta)},0)\}$ and by Table \ref{Table:Geometric_genes} we have
\begin{equation}
\label{eq:case:2:ell+1}
T_{\ell-1}=w_{\ell-1}\begin{pmatrix}
p&-X_{\ell-1}\\ 0& 1\end{pmatrix}
\end{equation}
 where $w_{\ell-1}$ is the permutation part of $\tld{w}_{\ell-1}$.

As in the previous item, we now claim that there exists $r_{k}\in\N$ such that $p^{r_k}\ovl{\cI}_{k}$ is generated by the equations 
\[
p^{r_k}M_{\textnormal{out},o_k}\left(\prod_{l=o_k+1}^{\ell-2}T_l\right)M_{\ell-1}=0
\]
where
\[
M_{{\ell-1}}=
\left\{\begin{array}{cc} (w_0)^{\delta_{\bX_{f+\ell-1}=\gB}}\Sigma_{\ell-1}\binom{p-X_{\ell-1}X_\ell}{-X_{\ell-1}}&\mbox{ if $i_k$ is of type } II, \cr
(w_0)^{\delta_{\bX_{f+\ell-1}=\gB}} \Sigma_{\ell-1}\binom{p-X_{\ell-1}X_\ell}{-X_{\ell-1}}(p, X_{\ell+1})&\mbox{ if $i_k$ is of type } I.
 \cr
 \end{array}\right.
\]
But this is clear from \eqref{eq:case:2:ell}, \eqref{eq:case:2:ell+1} noting that from Proposition \ref{prop:constraints}\eqref{fibre:AAB}) we have $w_{\ell-1}=(w_0)^{\delta_{\bX_{f+\ell-1}=\gB}}\Sigma_{\ell-1}$ and that, when $i_k$ is of type $I$, we can replace $(p-X_{i_k+1}X_{i_k},\ -X_{i_k})$ by $(p,\ -X_{i_k})$ to obtain a system of equations equivalent to \eqref{matrix:equation:explicit}.
\\
Proof of item (\ref{item:cAB:II}). 
 If $\ell-1$ is of type $II$ then $\ell-1=j_0=o_k$ and 
 $\binom{\bX_{o_k+f}}{\bX_{o_k}}=\binom{\gO}{\gAB}$.
The argument in the proof of Lemma \ref{lem:v'0:gO}\eqref{it:frag:II:II} and \eqref{it:frag:I:II} shows that 
 there exists $r_k\in\N$ such that $p^{r_k}\overline{\cI}_k$ is generated by the equations (after an adequate change of variables)
$$\left\{\begin{array}{cc}
p^{r_k}M_{\textnormal{out},o_k}\binom{1}{-X_{o_k+1}}=0, r_k=0 &\mbox{ if $i_k$ is of type } II,\cr
p^{r_k}M_{\textnormal{out},o_k}\binom{1}{-X_{o_k+1}} (p-X_{i_{k}+1}X_{i_{k}}, -X_{i_{k_{}}})=0&\mbox{ if $i_k$ is of type } I
\end{array}\right.$$
where $M_{\textnormal{out},o_k}=(-p,Y_{o_k})$ (since $\tld{w}_{o_k}=w_0t_\eta$ by Proposition \ref{prop:constraints}(\ref{fibre:OAB})). 
The result follows.
\end{proof}

\begin{lemma}[the $\gA,\gB$-part of a cluster]
\label{lem:seq:AB}
Assume $j_0\not=\ell-1$.
\begin{enumerate} 
\item
\label{item:l1}
There exists $r_{k_{s'}}\in\N$ such that
$$p^{r_{k_{s'}}}\overline{\cI}_{k_{s'}}=\left\{\begin{array}{cc}
p^{r_{k_{s'}}}\overline{\cI}'_{k_{s'}}& \mbox{ if type of $i_{k_{s'}}$ is } II,\cr
p^{r_{k_{s'}}}\overline{\cI}'_{k_{s'}}(p,X_{\ell+1})& \mbox{ if type of $i_{k_{s'}}$ is } I,\cr
 \end{array}\right.$$
where $\overline{\cI}'_{k_{s'}}$ is an ideal of $R_{[j_0,j_1]}/\sum_{i=0}^{s'-1}\cI_{k_{i}}$ generated by elements which do not depend (up to an automorphism induced by an explicit change of variables at $o_{k_{s'}}$) on the choice of $(\tld{w},s,\mu)$ in the fiber above $\bX$ of the map \eqref{eq:map:triple_to_gene}.
\item
\label{item:l2}
For all $s''\in\{s'+1,\ldots,s\}$,
$\overline{\cI}_{k_{s''}}$ is an ideal of $R_{[j_0,j_1]}/\sum_{l=0}^{s''-1}\cI_{k_{l}}$  generated by elements which do not depend on the choice of $(\tld{w},s,\mu)$ in the fiber above $\bX$ of the map \eqref{eq:map:triple_to_gene}.
\end{enumerate}
\end{lemma}
\begin{proof}
Since $j_0\not=\ell-1$, $\bX_{\ell-1}=\gAB$, $\bX_{\ell-1+f}\in\{\gA,\gB\}$ and $\bX_j,\bX_{j+f}\in\{\gA,\gB\}$ for all $j\in\{j_0+1,\ldots,\ell-2\}$. 
By Proposition \ref{prop:constraints} the triple $(s_{j+1},s_{\orient,j}, \mathrm{type~of~} j)$ are determined by $\bX$, for all $ j\in\{j_0+1,\ldots,\ell-2\}$. 
\\
Proof of item (\ref{item:l1}).
The case $o_{k_{s'}}=\ell-1$ follows from Lemma \ref{lem:trans:AB} \eqref{item:cAB:I}. 

Assume $o_{k_{s'}}<\ell-1$. 
By Lemma \ref{lem:trans:AB} (\ref{item:cAB:0}),
there exists $r_{k_{s'}}\in\N$ such that
 $p^{r_{k_{s'}}}\overline{\cI}_{k_{s'}}$
 is generated by the equations
\begin{equation}\label{equa:sys} p^{r_{k_{s'}}}M_{\textnormal{out},o_{k_{s'}}}\left(\prod_{l=o_k+1}^{\ell-2}T_l\right)M_{\ell-1}=0
\end{equation}
where
\[
M_{\ell-1}\defeq
\left\{\begin{array}{cc}(w_0)^{\delta_{\bX_{f+\ell-1}=\gB}}\Sigma_{\ell-1} \binom{-X_{\ell-1}}{p-X_{\ell-1}X_{\ell}}& \mbox{ if $i_{k_{s'}}$ is of type } II,\cr
(w_0)^{\delta_{\bX_{f+\ell-1}=\gB}} \Sigma_{\ell-1} \binom{-X_{\ell-1}}{p-X_{\ell-1}X_{\ell}}(X_{\ell+1},p)& \mbox{ if $i_{k_{s'}}$ is of type } I.\cr
  \end{array}\right.
  \]
 We first prove, by decreasing induction on $j\in\{o_{k_{s'}}+1,\ldots,\ell-1\}$, that in (\ref{equa:sys}) we can replace $\left(\prod_{l=j}^{\ell-2}T_l\right)M_{\ell-1}$ by
 $$\left\{\begin{array}{cc} \Sigma_jM'_j(p,X_{\ell+1}) &\mbox{ if $i_{k_{s'}}$ is of type } I,\cr
\Sigma_jM'_j&\mbox{ if $i_{k_{s'}}$ is of type } II,\cr\end{array}\right.
$$
where $M'_j$ depends only on $\bX$.

The result is true for $\ell-1$. 
For the inductive step, using the relation $M'_{j}=\Sigma_j^{-1}T_j\Sigma_{j+1}M'_{j+1}$, it is enough to prove that up to sign the matrix $\Sigma_j^{-1}T_j\Sigma_{j+1}=\Sigma_j^{-1}T_j\Sigma_{j}z_{j+1}$ only depends on $\bX$ for $j\in(o_{k_{s'}}+1,\ldots,\ell-2)$.
This is a casewise check using Table \ref{TableGen13} and Table \ref{Table:Geometric_genes}.
Indeed, as $j$ is of type $0$ for all $j\in(o_{k_{s'}}+1,\ldots,\ell-2)$ we have from Table \ref{TableGen13}  that
\[
((w_0)^{\delta_{\bX_{f+j}=\gB}}\Sigma_j,\tld{w}_j,z_{j+1})\in\{(\Id,t_{w_0(\eta)},\Id),(w_0,t_{\eta},w_0),(w_0,t_{\eta},\Id)\}
\]
(The factor $(w_0)^{\delta_{\bX_{f+j}=\gB}}$ is justified by Table \ref{TableGen4}) and an elemetary computation from Table \ref{Table:Geometric_genes} shows that
\[
\Sigma_j^{-1}T_j\Sigma_{j}z_{j+1}=(\pm)(w_0)^{\delta_{\bX_{f+j}=\gB}}\begin{pmatrix}p&-X_j\\0&1\end{pmatrix}(w_0)^{\delta_{\bX_{f+j}=\gB}}.
\]

We conclude that
$p^{r_{k_{s'}}}\overline{\cI}_{k_{s'}}$
 is generated by the equations
\begin{equation}\label{equa:sysf} p^{r_{k_{s'}}}M_{\textnormal{out},o_{k_{s'}}}M_{o_{k_{s'}}}=0
\end{equation}
where
$$M_{o_{k_{s'}}}= \left\{\begin{array}{cc} \Sigma_{o_{k_{s'}+1}}M'_{o_{k_{s'}+1}}(p,X_{\ell+1}) &\mbox{ if $i_{k_{s'}}$ is of type } I,\cr
\Sigma_{o_{k_{s'}+1}}M'_{o_{k_{s'}+1}}&\mbox{ if $i_{k_{s'}}$ is of type } II,\cr\end{array}\right.
$$
and $M'_{o_{k_{s'}+1}}$ depends only on $\bX$. 

We now perform a casewise analysis according to the type of $\jmath\defeq o_{k_{s'}}$.
\begin{enumerate}[start=1,label={$\spadesuit$\arabic*}]
\item
\label{it:cas:I:smiley}
Assume $\jmath$ is of type $I$.
Then by Tables \ref{TableGen4} and \ref{TableGen13} we have $\binom{\bX_{\jmath+f}}{\bX_{\jmath}}\in\{\binom{\gA}{\gA},\binom{\gB}{\gB}\}$.
The system of equations \eqref{equa:sysf} is equivalent to 
\[
p^{r_{k_{s'}}}\Sigma_{\jmath} M_{\textnormal{out},\jmath}\Sigma_{\jmath+1}M'_{\jmath+1}(p,X_{\ell+1})^{\delta_{i_{k_{s'}}=I}}=0.
\]
A direct check on Table \ref{TableGen13} and Table \ref{Table:Geometric_genes} shows that $\Sigma_{\jmath}M_{\textnormal{out},\jmath}\Sigma_{\jmath}z_{\jmath+1}$ only depends on whether $\tld{w}_j=t_{w_0\eta}$ or $\tld{w}_j\in\{t_\eta,w_0t_\eta\}$ up to the change of variables $X_{\jmath}\longleftrightarrow Z_{\jmath}$, $Y_{\jmath}\longleftrightarrow p-Y_{\jmath}$ (change of variables which happens exactly when $\Sigma_{\jmath}=w_0$).
By Proposition \ref{prop:constraints}\eqref{fibre:BAB} we conclude that the system of equation \eqref{equa:sysf} only depends on $\bX$ up to the the change of variables $X_{\jmath}\longleftrightarrow Z_{\jmath}$, $Y_{\jmath}\longleftrightarrow p-Y_{\jmath}$.

\item 
Assume $\jmath=o_{k_{s'}}$ is of type $II$.
By definition of cluster and the fact that $o_{k_{s'}}\leq \ell-1$, we conclude that $\jmath$ is of type $II$.
A direct check on Table \ref{TableGen13} (which provides the possible choices for $(\tld{w}_{\jmath},\Sigma_{\jmath}, z_{\jmath})$) and Table \ref{Table:Geometric_genes} shows that $M_{\textnormal{out},\jmath}\Sigma_{\jmath}z_{\jmath+1}$ only depends on whether $\tld{w}_j=t_{w_0\eta}$ or $\tld{w}_j\in\{t_\eta,w_0t_\eta\}$.
We conclude from Proposition \ref{prop:constraints}\eqref{fibre:BAB} that the system of equation \eqref{equa:sysf} only depends on $\bX$.
\end{enumerate}
Proof of item (\ref{item:l2}). 
In this case we necessarily have $s'<s$.
Starting from case \eqref{it:cas:I:smiley} above, we inductively analyze, for $s\geq s''> s'$, the system of equations
\[
M_{\textnormal{out},o_{k_{s''}}}\left(\prod_{l=o_{k_{s''}}+1}^{i_{k_{s''}}-1}T_l\right)\Sigma_{o_{k_{s''-1}}}M_{\textnormal{in},o_{k_{s''-1}}}\Sigma_{o_{k_{s''-1}}}=0
\]
where the presence of the $\Sigma_{o_{k_{s''-1}}}$-conjugation on the matrix  $M_{\textnormal{in},o_{k_{s''-1}}}$ (defined in Table \ref{Table:Geometric_genes}) is explained by the change of variables $X_{o_{k_{s''-1}}}\longleftrightarrow Z_{o_{k_{s''-1}}}$, $Y_{o_{k_{s''-1}}}\longleftrightarrow p-Y_{o_{k_{s''-1}}}$  when $\Sigma_{o_{k_{s''-1}}}=w_0$.
Since the expression $M_{\textnormal{in},o_{k_{s''-1}}}$ is independent of $\bX$, we can now perform the same argument appearing in the proof of item \eqref{item:l1} (where we replace $M'_{\ell-1}$ by $M_{\textnormal{in},o_{k_{s''-1}}}$ in the initial inductive argument there).
\end{proof}

\begin{proof}[Proof of Proposition \ref{prop:ring:naive}].
Let $(\bX_j,\bX_{j+f})_{j\in\cJ}$ be a gene satisfying \eqref{en:abstract:gene:1}--\eqref{en:abstract:gene:3}.
If there exists $i\in\cJ$ such that $(\bX_i,\bX_{i+f})=(\gO,\gO)$ then the deformation ring is zero and the result is obvious.
In what follows we assume that $(\bX_j,\bX_{j+f})_{j\in\cJ}$ satisfies \eqref{en:abstract:gene:1}-- \eqref{en:abstract:gene:extra}.\\

Assume that there exists $i\in\cJ'$ such that $\bX_i=\gO$ and let $\bX=\cup_{i=0}^{r}(\bX_{j_i},\bX_{j_{i+f}})_{j_i\leq j\leq j_{i+1}}$ be the decomposition of $\bX$ into clusters.\\
Proof if item (\ref{prop:caseO}).
By Lemmas \ref{lem:v'0:gO}-\ref{lem:seq:AB} applied on each fragments of $\bX$ we have
\[
\cZ^{\nv,\tau}(\tld{z})=\Spec\left(\otimes_{i=0}^{r} R_{[j_i,j_{i+1}]}/\cI_{[j_i,j_{i+1}]}\right).
\]
and each $\cI_{[j_i,j_{i+1}]}$ is a sum of the ideals described in

\begin{itemize}
\item Lemma \ref{lem:v'0:gO}, so that after $p$-saturation we can solve the variables $Y_i,Z_i$, while the variable $X_i$ is free, 
\item Lemma \ref{lem:trans:AB} and Lemma \ref{lem:seq:AB} (\ref{item:l1}) so that after $p$-saturation these ideals produce equations which do not depend on the choice of $(\tld{w},s,\mu)$ in the fiber of the map \eqref{eq:map:triple_to_gene} at $\bX$.
\item Lemma \ref{lem:seq:AB} (\ref{item:l2}), and these ideals admits generators which do not depend on the choice of $(\tld{w},s,\mu)$ in the fiber of the map \eqref{eq:map:triple_to_gene} at $\bX$. 
\end{itemize}
We concude that the $p$-saturation of $\cZ^{\nv,\tau}(\tld{z})$ depends only on $\bX$.\\
Proof of item (\ref{prop:casenotO}). 
The proof is similar to the proof of Proposition \ref{prop:constraints}\eqref{fibre:BAB}
If $\bX_{j'}\in\{\gA,\gB\}$ for all $j'\in\cJ'$ then either $s_{j+1}$ or $s_{\orient,i}$ is determined by $\bX_{j'}$, for all $j'\in\cJ$.  
If there exists $j_0\in\cJ$ such that 
$\binom{\bX_{j_0}}{\bX_{j_0+f}}\in\left\{\binom{\gA}{\gB},\binom{\gB}{\gA}\right\}$ then $s_{j_0+1}$ is uniquely determined (cf.~Table \ref{TableGen13}). 
As $s_{\orient,f-1}=\Id$ by Lemma \ref{lem:s_or} we conclude that $(s_{j+1},s_{\orient,j})_{j\in \cJ}$ is uniquely determined by $(\bX_j,\bX_{j+f})_{j\in \cJ}$.

If $\binom{\bX_j}{\bX_{j+f}}\in\left\{\binom{\gA}{\gA},\binom{\gB}{\gB}\right\}$ for all $j\in\cJ$ then $(s_{\orient,j})_{j\in\cJ}$ is determined by $\bX$. 
By \eqref{en:abstract:gene:3} there exists $j'\in\cJ'$ with $\bX_{j'}\not=\bX_{j'+1}$. By Table \ref{TableGen13}, we obtain $s_{j'+1}=(12)$ and by induction $(s_{j+1},s_{\orient,j})_{j\in \cJ}$ are determined.

As $(s_{j+1}, s_{\orient,j}, \text{type of $j$})$ are determined for all $j\in\cJ$,
$\cZ^{\nv,\tau}(\tld{z})$ is determined by $\bX$.
\end{proof}

\begin{thm}\label{thm:indepen}
Assume that $p>8f+3+\max_j\langle\mu_j,\alpha^\vee\rangle$.
Let $\tau$ be a regular tame inertial type of niveau $f$ and $\rhobar: G_K\ra \GL_2(\F)$ be absolutely irreucible and such that $\det (\rhobar)\otimes_{\F} \omega=\det(\tau)\otimes_{\cO}\F$.
Then $R_{\rhobar}^{\eta,\tau}$ depends only on $\bX(\tau,\rhobar|_{I_K})$.
Moreover, there exists an integer $r\geq 0$ and a decomposition $\bX(\tau,\rhobar|_{I_K})=\cup_{i=0}^{r}(\bX_{j_i},\bX_{j_{i+f}})_{j_i\leq j\leq j_{i+1}}$ such that
\[
R_{\rhobar}^{\eta,\tau}\cong \widehat{\otimes}_{i=0}R_{i}
\]
where $R_i$ is a complete local Noetherian $\cO$ algebra depending only on $(\bX_{j_i},\bX_{j_{i+f}})_{j_i\leq j\leq j_{i+1}}$.
\end{thm}
\begin{proof} 
This follows from Theorem \ref{thm:main:Gal:def} and Proposition \ref{prop:ring:naive}.
\end{proof}

\subsection{Examples}
\label{sec:examples}
We collect some examples computing potentially Barsotti--Tate deformation rings using the techniques of this article.

In what follows, given $\ovl{x}\in \F^\times$ we denote by $\mathrm{un}_{\ovl{x}}$ the unramified character of $G_K$ sending $p$ to $\ovl{x}$.
\subsubsection{Examples when f=1}\label{sec:examples f=1}

By Theorem \ref{thm:main:Gal:def}, we see immediately that if we are in a Type $II$ situation of Table \ref{Table:matrices} then $R^{\eta,\tau}_{\rhobar}$ is formally smooth over $\cO$ or $\cO[\![X,Y]\!]/(XY-p)$. We are thus left with the following cases
\begin{enumerate}
\item Case $1$: $\tau=\tau((12),(1,0))$ (thus $\tau=\omega_2\oplus \omega_2^p$) and $\tld{w}\in \{t_\eta,w_0t_\eta\}$, with $\tld{z}\in\{w_0t_{(1,1)},t_{(1,1)}\}$.
\item Case $2$: $\tau=\tau(\Id,(1,0))$ (thus $\tau=1\oplus \omega$) and $\tld{w}=t_{w_0(\eta)}$, $\tld{z}=t_{(1,1)}$.
\end{enumerate}
We wish to compute the $p$-saturation of $\tld{\cZ}^{\nv,\tau}(\tld{z})$ in those cases. 

We start with case $1$. Since the two subcases give isomorphic spaces, we will work with $\tld{w}=w_0t_{\eta}$, so that $\tld{z}=t_{(1,1)}$.
From Table \ref{Table:matrices}, $\tld{\cZ}^{\nv,\tau}(\tld{z})$ is presented as the quotient of $\cO[B,C,D,\alpha,\beta,\gamma,\delta]$ subject to $D(p-D)=BC$, $\alpha\delta-\beta\gamma$ invertible, and the relation
\begin{equation*}
\begin{pmatrix}D&-B\\-C&(p-D)\end{pmatrix}
\begin{pmatrix}\alpha&\beta\\\gamma&\delta\end{pmatrix}
\begin{pmatrix}D&-B\\-C&(p-D)\end{pmatrix}=0
\end{equation*}
Using $D(p-D)=BC$, a simple manipulation shows that the above matrix equation is equivalent to $DF=BF=CF=(p-D)F=0$ where $F=\gamma B+\beta C-\alpha D-\delta (p-D)$.
It follows that the $p$-saturation $\tld{\cZ}^{\tmod,\tau}(\tld{z})$ is given by the equations 
\[D(p-D)=BC, \gamma B+\beta C=\alpha D +\delta(p-D)\]
We check on Macaulay 2 that the ideal of $2\times 2$ minors of the Jacobian matrix of the above relations together with the relations themselves contains $p^2$, and has radical $(p,B,C,D)$.
Thus the non-smooth locus of $\tld{\cZ}^{\tmod,\tau}(\tld{z})/\cO$ is $p=B=C=D=0$, which correspond to $\rhobar\otimes \ovl{\varepsilon}^{-1}$ unramified.
After twisting $\rhobar$, we see that $R^{\eta,\tau}_{\rhobar}$ is isomorphic to the completion of 
\[\cO[B,C,D,\alpha,\beta,\gamma,\delta]/(D(p-D)-BC,\gamma B+\beta C-\alpha D -\delta(p-D))\]
at either $(p,B,C,D,\alpha-1,\delta-1,\beta,\gamma)$, $(p,B,C,D,\alpha-s,\delta-t,\beta,\gamma)$ (with $s\neq t \in \F^\times$) or $(p,B,C,D,\alpha-1,\delta-1,\beta-1,\gamma)$.

We note that the second ideal correspond to the point $\rhobar=\mathrm{un}_{s}\ovl{\varepsilon}\oplus \mathrm{un}_{t}\ovl{\varepsilon}$ with $s\neq t$, and after eliminating $D$ using the fact that $\alpha-\delta$ is unit, and making a change of variable on $B,C$, we see that the completion of $\tld{\cZ}^{\tmod,\tau}(\tld{z})$ at $\rhobar$ is a power series ring over $\cO[\![X,Y]\!]/(XY-p^2)$.

We now move to case $2$.
Reading off from Table \ref{Table:matrices} like in case $1$, $\tld{\cZ}^{\nv,\tau}(\tld{z})$ is presented as the quotient of $\cO[B,C,D,\alpha,\beta,\gamma,\delta]$ subject to $D(p-D)=BC$, $\alpha\delta-\beta\gamma$ invertible, and the relations
\begin{align*}
&BC\alpha+DC\beta+DB\gamma-BC\delta-pC\beta=0,\\
&DC\alpha-C^2\beta-BC\gamma-DC\delta+pD\gamma=0\\
&DB\alpha-BC\beta-B^2\gamma-DB\delta-pB\alpha-pD\beta+pB\delta+p^2\beta=0.
\end{align*}
One can check that $\tld{\cZ}^{\nv,\tau}(\tld{z})$ is already $p$-saturated, hence agrees with $\tld{\cZ}^{\tmod,\tau}(\tld{z})$.
We check on Macaulay 2 that the ideal of $2\times 2$ minors of the Jacobian matrix of the above relations together with the relations themselves contains $p^3$, and has radical $(p,B,C,D)$. Thus the non-smooth locus of $\tld{\cZ}^{\tmod,\tau}(\tld{z})/\cO$ is $p=B=C=D=0$, which correspond to $\rhobar\otimes \ovl{\varepsilon}^{-1}$ unramified.
We again conclude that after twisting $\rhobar$, $R^{\eta,\tau}_{\rhobar}$ is isomorphic to the completion of 
\[
\cO[B,C,D,\alpha,\beta,\gamma,\delta]\left/
\tiny{\left(
\begin{array}{l}D(p-D)-BC,\,
BC\alpha+DC\beta+DB\gamma-BC\delta-pC\beta,\\
DC\alpha-C^2\beta-BC\gamma-DC\delta+pD\gamma,\\
DB\alpha-BC\beta-B^2\gamma-DB\delta-pB\alpha-pD\beta+pB\delta+p^2\beta\end{array}\right)}\right.\]
at either $(p,B,C,D,\alpha-1,\delta-1,\beta,\gamma)$, $(p,B,C,D,\alpha-s,\delta-t,\beta,\gamma)$ (with $s\neq t \in \F^\times$) or $(p,B,C,D,\alpha-1,\delta-1,\beta-1,\gamma)$.

\begin{rmk}(Bounds on $p$ for $f=1$) 
\label{rmk:bd:Qp}
Our explicit computations in this section allows us to slightly relax the requirement on $p$ in Theorem \ref{thm:main:model}. Specifically, the improvement in the proof of Theorem \ref{thm:main:model} comes from two sources
\begin{itemize} 
\item We can get a map out of $\cO[\![X,Y]\!]/(XY-p^k)$ as soon as we have a map modulo $p^{k+1}$, instead of using the general Elkik bound which would have required a map modulo $p^{2k+1}$. This shows that our above models are valid for $p\geq 7$, unless (up to twists) $\rhobar\otimes \ovl{\varepsilon}^{-1}$ is unramified and has scalar semisimplification and $\tau=\omega_2\oplus \omega_2^p$, or $\rhobar\otimes \ovl{\varepsilon}^{-1}$ is unramified and $\tau=1\oplus \omega$.
We remark that these computations justify the claims made in \cite[\S 7.5.13]{EGH}.
\item In the remaining cases, we get a small saving in the bound required to apply Elkik's approximation theorem compared to the general bound of Proposition \ref{prop:elkik bound}.
This in particular shows that our models for these cases are valid for $p>7$.
\end{itemize}
\end{rmk}
\subsubsection{Examples with $f=3$}
Tables \ref{TableCDM1}, \ref{TableCDM2} record several examples of deformation rings corresponding to the examples of the left column (resp.~right column) of \cite[Table 4]{CDM2} (by completion at the ideal generated by the variables $X_0,X_1,X_2,Y_0\dots$).
In particular, contrary to the expectations of \cite[\S 5.3.2]{CDM2}, the deformation rings extracted from the first three rows of Table \ref{TableCDM1} are \emph{not} all of the form ``$XY+p^2$'', despite the fact that these examples share the same stratified Kisin variety and the type $\tau$ is non degenerate in the sense of \cite{CDM2}.
This gives a counterexample to \cite[Conjecture 5.1.5]{CDM2} when the coefficient ring $\cO_E$ of \emph{loc.~cit}.~is absolutely unramified (but not after enlarging it).

\begin{table}[H]
\captionsetup{justification=centering}
\caption[Foo content]{\textbf{Examples from \cite[\S 5.3]{CDM2}  for $f=3$
}
}
\label{TableCDM2}
\centering
\adjustbox{max width=\textwidth}{
\begin{tabular}{| c | c | c | }
\hline
&&\\
$\begin{matrix}\begin{tikzpicture}[xscale=0.8,yscale=0.7]
\draw [, thick] (-0.5,-0.5) rectangle (2.5,1.5);
\draw[thick,->>] (-0.5,-0.5)--(-0.5,1);
\draw[thick,->>] (2.5,1)--(2.5,0);
\node at (0, 1) { $\gO$ };
\node at (1, 1) { $\gB$ };
\node at (2, 1) { $\gB$ };
\node at (0, 0) { $\gA$ };
\node at (1, 0) { $\gA$ };
\node at (2, 0) { $\gAB$ };
\end{tikzpicture}
\end{matrix}$
 &$(II,w_0t_\eta),(0,t_{w_0(\eta)}),(0,t_{w_0(\eta)})$&$\begin{aligned}&R=\cO[X_0,Y_0,X_1,X_2]\\
 &I^{\textnormal{nv}}=I^{p\textnormal{-sat}}=(X_0Y_0+p^3)\end{aligned}$\\
&&\\
\hline
&&\\
$\begin{matrix}\begin{tikzpicture}[xscale=0.8,yscale=0.7]
\draw [, thick] (-0.5,-0.5) rectangle (2.5,1.5);
\draw[thick,->>] (-0.5,-0.5)--(-0.5,1);
\draw[thick,->>] (2.5,1)--(2.5,0);
\node at (0, 1) { $\gO$ };
\node at (1, 1) { $\gA$ };
\node at (2, 1) { $\gB$ };
\node at (0, 0) { $\gB$ };
\node at (1, 0) { $\gA$ };
\node at (2, 0) { $\gAB$ };
\end{tikzpicture}
\end{matrix}$
 &$(II,w_0t_\eta),(I,w_0t_\eta),(0,w_0t_\eta)$&$\begin{aligned}&R=\cO[X_0,Y_0,X_1,Y_1,Z_1,X_2]\\
 &I^{\textnormal{nv}}=\left(\begin{matrix}Y_1(p-Y_1)-X_1Z_1,\qquad pY_1+Y_0Z_1,\qquad 
 pX_1+Y_0(p-Y_1),\\
 Y_1X_0+X_1(p+X_0X_2),\qquad Z_1X_0+(p-Y_1)(p+X_0X_2)\end{matrix}\right)\\
 &I^{p\textnormal{-sat}}=\left(\begin{matrix}
Y_0Z_1+pY_1,\qquad Y_1^2+X_1Z_1-pY_1,\qquad Y_0Y_1-pY_0-pX_1,\\X_0X_1X_2+X_0Y_1+pX_1,\qquad X_0Y_1X_2-X_0Z_1-pX_0X_2+pY_1-p^2
 \end{matrix}\right)
 \end{aligned}$\\
&&\\
\hline
&&\\
$\begin{matrix}\begin{tikzpicture}[xscale=0.8,yscale=0.7]
\draw [, thick] (-0.5,-0.5) rectangle (2.5,1.5);
\draw[thick,->>] (-0.5,-0.5)--(-0.5,1);
\draw[thick,->>] (2.5,1)--(2.5,0);
\node at (0, 1) { $\gA$ };
\node at (1, 1) { $\gA$ };
\node at (2, 1) { $\gB$ };
\node at (0, 0) { $\gA$ };
\node at (1, 0) { $\gB$ };
\node at (2, 0) { $\gB$ };
\end{tikzpicture}
\end{matrix}$
 &$(I,w_0t_\eta),(0,t_{w_0(\eta)}),(I,t_{\eta})$&$\begin{aligned}&R=\cO[X_0,Y_0,Z_0,X_1,X_2,Y_2,Z_2]\\
  &I^{\textnormal{nv}}=I^{p\textnormal{-sat}}=\left(\begin{matrix}(p-Y_2)Y_0+Z_2Z_0,\qquad (p-Y_2)X_0+Z_2(p-Y_0)\\
 X_2Y_0+Z_0Y_2,\qquad X_2X_0+Y_2(p-Y_0),\\
 Y_0(pY_2+X_1Z_2)+X_0Z_2,\qquad Z_0(pY_2+X_1Z_2)+Z_2(p-Y_0),\\
 Y_0(pX_2+X_1(p-Y_2))+X_0(p-Y_2),\qquad  Z_0(pX_2+X_1(p-Y_2))+(p-Y_0)(p-Y_2),\\
Y_2(p-Y_2)-X_2Z_2,\qquad  Y_0(p-Y_0)-X_0Z_0
\end{matrix}\right) \end{aligned}$\\
&&\\
\hline
\hline
\end{tabular}}
\end{table}

\begin{table}[H]
\captionsetup{justification=centering}
\caption[Foo content]{\textbf{Examples from \cite[\S 5.3]{CDM2} for $f=3$
}
}
\label{TableCDM1}
\centering
\adjustbox{max width=\textwidth}{
\begin{tabular}{| c | c | c | }
\hline
\hline
&&\\
Gene&$(\text{type of } j,\,\tld{w}_j)_{j=0,1,2}$&Equations for $\tld{\cZ}^{\textnormal{nv},\tau}(\tld{z})$ and $(\tld{\cZ}^{\textnormal{nv},\tau}(\tld{z}))^{p\text{-sat}}$\\
&&\\
\hline
&&\\
$\begin{matrix}
\begin{tikzpicture}[xscale=0.8,yscale=0.7]
\draw [, thick] (-0.5,-0.5) rectangle (2.5,1.5);
\draw[thick,->>] (-0.5,-0.5)--(-0.5,1);
\draw[thick,->>] (2.5,1)--(2.5,0);
\node at (0, 1) { $\gO$ };
\node at (1, 1) { $\gB$ };
\node at (2, 1) { $\gA$ };
\node at (0, 0) { $\gA$ };
\node at (1, 0) { $\gAB$ };
\node at (2, 0) { $\gO$ };
\end{tikzpicture}
\end{matrix}$
 &$\begin{aligned}(II,t_\eta),(0,w_0t_\eta),(II,t_\eta)\end{aligned}$&$\begin{aligned}&R=\cO[X_0,X_1,X_2,Y_0]\\&I^{\textnormal{nv}}=I^{p\textnormal{-sat}}=(X_2Y_0+p^2)\end{aligned}$\\
&&\\
\hline
&&\\
$\begin{matrix}
\begin{tikzpicture}[xscale=0.8,yscale=0.7]
\draw [, thick] (-0.5,-0.5) rectangle (2.5,1.5);
\draw[thick,->>] (-0.5,-0.5)--(-0.5,1);
\draw[thick,->>] (2.5,1)--(2.5,0);
\node at (0, 1) { $\gO$ };
\node at (1, 1) { $\gB$ };
\node at (2, 1) { $\gB$ };
\node at (0, 0) { $\gA$ };
\node at (1, 0) { $\gAB$ };
\node at (2, 0) { $\gO$ };
\end{tikzpicture}
\end{matrix}$
 &$(II,t_\eta),(0,w_0t_\eta),(II,t_{w_0(\eta)})$&$\begin{aligned}&R=\cO[X_0,X_1,X_2,Y_0]\\&I^{\textnormal{nv}}=I^{p\textnormal{-sat}}=(X_2Y_0+p^2)\end{aligned}$\\
&&\\
\hline
&&\\
$\begin{matrix}\begin{tikzpicture}[xscale=0.8,yscale=0.7]
\draw [, thick] (-0.5,-0.5) rectangle (2.5,1.5);
\draw[thick,->>] (-0.5,-0.5)--(-0.5,1);
\draw[thick,->>] (2.5,1)--(2.5,0);
\node at (0, 1) { $\gO$ };
\node at (1, 1) { $\gB$ };
\node at (2, 1) { $\gB$ };
\node at (0, 0) { $\gB$ };
\node at (1, 0) { $\gA$ };
\node at (2, 0) { $\gAB$ };
\end{tikzpicture}
\end{matrix}$
 &$(II,t_\eta),(0,t_\eta),(0,w_0t_\eta)$&$\begin{aligned}&R=\cO[X_0,X_1,X_2,Y_0]\\
 &I^{\textnormal{nv}}=I^{p\textnormal{-sat}}=(X_0Y_0+p^3)\end{aligned}$\\
&&\\
\hline
&&\\
$\begin{matrix}\begin{tikzpicture}[xscale=0.8,yscale=0.7]
\draw [, thick] (-0.5,-0.5) rectangle (2.5,1.5);
\draw[thick,->>] (-0.5,-0.5)--(-0.5,1);
\draw[thick,->>] (2.5,1)--(2.5,0);
\node at (0, 1) { $\gA$ };
\node at (1, 1) { $\gB$ };
\node at (2, 1) { $\gB$ };
\node at (0, 0) { $\gA$ };
\node at (1, 0) { $\gA$ };
\node at (2, 0) { $\gA$ };
\end{tikzpicture}
\end{matrix}$
 &$(I,w_0t_\eta),(0,t_\eta),(0,w_0t_\eta)$&$
 \begin{aligned}
&R=\cO[X_0,X_1,X_2,Y_0,Z_0]\\
&I^{\textnormal{nv}}=\left(\begin{matrix}
(Y_0(Z_0+p^2Y_0+(p-Y_0)X_1),& Z_0(Z_0+p^2Y_0+(p-Y_0)X_1),&
  X_0(Z_0-p^2X_0+(p-Y_0)X_1),\\ (p-Y_0)(Z_0-p^2X_0+(p-Y_0)X_1),& (p-Y_0)Y_0-X_0Z_0\end{matrix}\right)\\
  &\\
 &I^{p\textnormal{-sat}}=\left(\begin{matrix}
Y_0^2+X_0Z_0-pY_0,\,\quad pX_0Y_0-X_1Y_0-pX_0Z_0-p^2X_0+pX_1+p^2Y_0+Z_0\\
pX_1Y_0+X_0Z_0+Y_0Z_0+p^3X_0-p^2X_1-pZ_0,\\
X_0^2Y_0+X_0Z_0^2+X_1Y_0+2pX_0Z_0+p^2X_0-pX_1-p^2Y_0-Z_0\\
pX_0^3-X_0^2X_1+pY_0Z_0^2+p^2X_0^2-pX_0X_1+(p^4-2p^2)X_0Z_0+(p^4+1)Y_0Z_0+p^3Y_0
\end{matrix}\right)
  \end{aligned}$\\
&&\\
\hline
&&\\
$\begin{matrix}\begin{tikzpicture}[xscale=0.8,yscale=0.7]
\draw [, thick] (-0.5,-0.5) rectangle (2.5,1.5);
\draw[thick,->>] (-0.5,-0.5)--(-0.5,1);
\draw[thick,->>] (2.5,1)--(2.5,0);
\node at (0, 1) { $\gA$ };
\node at (1, 1) { $\gB$ };
\node at (2, 1) { $\gA$ };
\node at (0, 0) { $\gA$ };
\node at (1, 0) { $\gA$ };
\node at (2, 0) { $\gB$ };
\end{tikzpicture}
\end{matrix}$
 &$(II,t_{w_0\eta}),(0,t_{w_0(\eta)}),(0,w_0t_\eta)$&$
 \begin{aligned}
 &R=\cO[X_0,X_1,X_2,Y_0,Z_0]
 \\
&I^{\textnormal{nv}}=\left(\begin{matrix}pZ_0(p-Y_0)+(X_0-X_1(p-Y_0))(-pY_0-Z_0X_2),
\qquad
-p(p-Y_0)^2+(X_0-X_1(p-Y_0))(pX_0+X_2(p-Y_0)),\\
  pZ_0^2+(-Z_0X_1+Y_0)(-pY_0-Z_0X_2),\qquad -pZ_0(p-Y_0)+(-Z_0X_1+Y_0)(pX_0+X_2(p-Y_0)),\\(p-Y_0)Y_0-X_0Z_0\end{matrix}\right)\\
&  \\
  &I^{p\textnormal{-sat}}=\left(\begin{matrix}
Y_0^2+X_0Z_0-pY_0,\qquad X_1X_2Z_0^2+pX_1Y_0Z_0-X_2Y_0Z_0+pX_0Z_0+pZ_0^2-p^2Y_0,\\
X_1X_2Y_0Z_0-pX_0X_1Z_0+X_0X_2Z_0-pX_1X_2Z_0+pX_0Y_0+pY_0Z_0-p^2Z_0,\\
X_0X_1X_2Z_0+pX_0X_1Y_0-X_0X_2Y_0+pX_1X_2Y_0+pX_0^2-p^2X_0X_1+pX_0X_2-p^2X_1X_2+pX_0Z_0+p^2Y_0-p^3\end{matrix}\right)
\end{aligned}$\\
&&\\
\hline
&&\\
$\begin{matrix}\begin{tikzpicture}[xscale=0.8,yscale=0.7]
\draw [, thick] (-0.5,-0.5) rectangle (2.5,1.5);
\draw[thick,->>] (-0.5,-0.5)--(-0.5,1);
\draw[thick,->>] (2.5,1)--(2.5,0);
\node at (0, 1) { $\gO$ };
\node at (1, 1) { $\gB$ };
\node at (2, 1) { $\gB$ };
\node at (0, 0) { $\gA$ };
\node at (1, 0) { $\gB$ };
\node at (2, 0) { $\gAB$ };
\end{tikzpicture}
\end{matrix}$
 &$(II,t_\eta),(I,t_{w_0\eta}),(0,w_0t_\eta)$&$
 \begin{aligned} &R=\cO[X_0,Y_0,X_1,Y_1,Z_1,X_2]
 \\
 &I^{\textnormal{nv}}=\left(\begin{matrix}
Y_0Y_1+pZ_1,\quad Y_0X_1+p(p-Y_1),\quad
  X_1(p+X_0X_2)-X_0(p-Y_1),\\ Y_1(p+X_0X_2)-X_0Z_1,\quad(p-Y_1)Y_1-X_1Z_1\end{matrix}\right)
   \\
   &\\
 &I^{p\textnormal{-sat}}=\left(\begin{matrix}
 Y_1^2+X_1Z_1-pY_1,\qquad Y_0Y_1+pZ_1,\qquad Y_0X_1-pY_1+p^2\\
 X_0Y_0+pX_0X_2+p^2,\qquad X_0Y_1X_2-X_0Z_1+pY_1,\qquad X_0X_1X_2+X_0Y_1-pX_0+pX_1
  \end{matrix}\right)
\end{aligned}$\\
&&\\
\hline
&&\\
$\begin{matrix}\begin{tikzpicture}[xscale=0.8,yscale=0.7]
\draw [, thick] (-0.5,-0.5) rectangle (2.5,1.5);
\draw[thick,->>] (-0.5,-0.5)--(-0.5,1);
\draw[thick,->>] (2.5,1)--(2.5,0);
\node at (0, 1) { $\gO$ };
\node at (1, 1) { $\gA$ };
\node at (2, 1) { $\gB$ };
\node at (0, 0) { $\gAB$ };
\node at (1, 0) { $\gO$ };
\node at (2, 0) { $\gAB$ };
\end{tikzpicture}
\end{matrix}$
 &$(II,w_0t_\eta),(II,t_\eta),(0,t_{w_0(\eta)})$&$\begin{aligned}&R=\cO[X_0,Y_0,X_1,Y_1,X_2]\\
&I^{\textnormal{nv}}=I^{p\textnormal{-sat}}=(p+Y_0X_1, X_0X_2Y_1+p(X_0+Y_1))\end{aligned}$\\
&&\\
\hline
&&\\
$\begin{matrix}\begin{tikzpicture}[xscale=0.8,yscale=0.7]
\draw [, thick] (-0.5,-0.5) rectangle (2.5,1.5);
\draw[thick,->>] (-0.5,-0.5)--(-0.5,1);
\draw[thick,->>] (2.5,1)--(2.5,0);
\node at (0, 1) { $\gO$ };
\node at (1, 1) { $\gA$ };
\node at (2, 1) { $\gB$ };
\node at (0, 0) { $\gA$ };
\node at (1, 0) { $\gB$ };
\node at (2, 0) { $\gAB$ };
\end{tikzpicture}
\end{matrix}$
 &$(II,t_\eta),(0,t_{w_0(\eta)}),(0,w_0t_{\eta})$&$\begin{aligned}&R=\cO[X_0,Y_0,X_1,X_2]\\
 &I^{\textnormal{nv}}=I^{p\textnormal{-sat}}=(pX_0Y_0+(p+X_1Y_0)(p+X_0X_2))\end{aligned}$\\
&&\\
\hline
\hline
\end{tabular}}
\caption*{
In the third column we express $\tld{\cZ}^{\textnormal{nv},\tau}(\tld{z})$ as $\Spec (R/I^{\textnormal{nv}})$, where $R=\otimes_{j=0}^2R_j$ and $I^{\textnormal{nv}}$, $(R_j)_{j=0,1,2}$ are extracted from the data of the second column and Table \ref{Table:Geometric_genes}}
\end{table}

\clearpage{}
\bibliography{Biblio}
\bibliographystyle{amsalpha}

\end{document}